\documentclass[12pt]{amsart}
\usepackage{graphicx}
\usepackage{amsmath}
\usepackage{amssymb}
\usepackage{setspace}
\usepackage{datetime}
\usepackage[colorlinks,
            linkcolor=black,
            anchorcolor=blue,
            citecolor=black]{hyperref}
\newtheorem{thm}{Theorem}[section]
\newtheorem{cor}[thm]{Corollary}

\newtheorem{lem}[thm]{Lemma}

\theoremstyle{definition}

\theoremstyle{remark}
\newtheorem{rem}[thm]{Remark}
\theoremstyle{conclusion}

\theoremstyle{question}
\numberwithin{equation}{section}
\begin{document}
\title[Uniformly elliptic nonlocal Bellman operator]{Maximum principles and the method of moving planes for the uniformly elliptic nonlocal Bellman operator and applications}

\author{Wei Dai, Guolin Qin}

\address{School of Mathematical Sciences, Beihang University (BUAA), Beijing 100083, P. R. China}
\email{weidai@buaa.edu.cn}

\address{Institute of Applied Mathematics, Chinese Academy of Sciences, Beijing 100190, and University of Chinese Academy of Sciences, Beijing 100049, P. R. China}
\email{qinguolin18@mails.ucas.ac.cn}

\thanks{Wei Dai is supported by the NNSF of China (No. 11971049) and the Fundamental Research Funds for the Central Universities.}

\begin{abstract}
In this paper, we establish various maximum principles and develop the method of moving planes and the sliding method (on general unbounded domains) for equations involving the uniformly elliptic nonlocal Bellman operator. As a consequence, we derive multiple applications of these maximum principles and the moving planes method. For instance, we prove symmetry, monotonicity and uniqueness results and asymptotic properties for solutions to various equations involving the uniformly elliptic nonlocal Bellman operator in bounded domains, unbounded domains, epigraph or $\mathbb{R}^{n}$. In particular, the uniformly elliptic nonlocal Monge-Amp\`{e}re operator introduced by Caffarelli and Charro in \cite{CC} is a typical example of the uniformly elliptic nonlocal Bellman operator.
\end{abstract}
\maketitle {\small {\bf Keywords:} Uniformly elliptic nonlocal Bellman operator; Uniformly elliptic nonlocal Monge-Amp\`{e}re operator; Maximum principles; Method of moving planes; Monotonicity, symmetry and uniqueness; Asymptotic properties.\\

}

\section{Introduction}

\subsection{Background and setting of the problem}
In this paper, we are concerned with the following nonlinear equations involving the uniformly elliptic nonlocal Bellman operator:
\begin{equation}\label{PDE}
  -\mathrm{F}_{s}u(x)=f\left(x,u(x),\nabla u(x)\right) \qquad \text{in} \,\,\, \Omega\subseteq\mathbb{R}^{n}
\end{equation}
with $0<s<1$ and $n\geq2$, where $\Omega$ is a bounded or unbounded domain in $\mathbb{R}^{n}$.  The uniformly elliptic Bellman integro-differential operator is defined by
\begin{eqnarray}\label{Def}
% \nonumber to remove numbering (before each equation)
 && \mathrm{F}_{s}u(x):=\inf\left\{P.V.\int_{\mathbb{R}^{n}}\frac{u(y)-u(x)}{|A^{-1}(y-x)|^{n+2s}}\mathrm{d}y \,\bigg|\, 0<\theta I\leq A\leq\Theta I\right\}  \\
 \nonumber && \qquad\quad\,\,\, =\inf\left\{\frac{1}{2}\int_{\mathbb{R}^{n}}\frac{u(x+y)+u(x-y)-2u(x)}{|A^{-1}y|^{n+2s}}\mathrm{d}y \,\bigg|\, 0<\theta I\leq A\leq\Theta I\right\},
\end{eqnarray}
where $\theta>0$ is an arbitrarily small constant and $\Theta\geq\theta$ is an arbitrarily large constant, P.V. stands for the Cauchy principal value and $B\leq A$ ($B<A$) means that $A-B$ is a non-negative (positive) definite square matrix. The condition $0<\theta I\leq A\leq\Theta I$ is equivalent to $\lambda_{min}(A)\geq\theta>0$ and $\lambda_{max}(A)\leq\Theta$. Let
\begin{equation}\label{space}
\mathcal{L}_{s}(\mathbb{R}^{n}):=\left\{u : \mathbb{R}^{n} \rightarrow \mathbb{R} \,\Bigg|\, \int_{\mathbb{R}^{n}} \frac{|u(x)|}{1+|x|^{n+2s}}dx<+\infty\right\}.
\end{equation}
Then, one can easily verify that for any $u \in C_{\text {loc }}^{1,1}\cap\mathcal{L}_s(\mathbb{R}^{n})$, the integral on the right hand side of the definition \eqref{Def} is well-defined. Hence $\mathrm{F}_{s}u$ makes sense for all functions $u \in C_{\text {loc }}^{1,1}\cap\mathcal{L}_{s}(\mathbb{R}^{n})$.

\smallskip

For any given unit vector $\mathbf{e}\in\mathbb{R}^n$, we will also consider the uniformly elliptic Bellman operator $\mathbf{F_{s}}$ ($s\in(0,1)$) associated with vector $\mathbf{e}$ defined by
\begin{eqnarray}\label{Defv}
% \nonumber to remove numbering (before each equation)
 && \mathbf{F_{s}}u(x):=\inf\left\{P.V.\int_{\mathbb{R}^{n}}\frac{u(y)-u(x)}{|A^{-1}(y-x)|^{n+2s}}\mathrm{d}y \,\bigg|\, 0<\theta I\leq A\leq\Theta I, \, \mathbf{e}\in EV(A) \right\}
\end{eqnarray}
for any $u \in C_{\text {loc }}^{1,1}\cap\mathcal{L}_{s}(\mathbb{R}^{n})$, where $EV(A)$ denotes the set of all eigenvectors of matrix $A$. One should note that the restriction $\mathbf{e}\in EV(A)$ is much weaker than assuming $A$ is diagonal.

\smallskip

By the definition \eqref{Def}, we get immediately the following comparison:
\begin{equation}\label{comparison}
  -\mathrm{F}_{s}u(x)\geq C_{n,s,\theta,\Theta}(-\Delta)^{s}u(x).
\end{equation}
The fractional Laplacian $(-\Delta)^{s}$ is also a nonlocal pseudo-differential operator, which is defined by (see e.g. \cite{CG,CLL,CLM,CQ,DQ})
\begin{eqnarray}\label{def-fracL}
(-\Delta)^s u(x):=C_{n,s} \, P.V. \int_{\mathbb{R}^{n}}\frac{u(x)-u(y)}{|x-y|^{n+2s}} \mathrm{d}y
\end{eqnarray}
for any $u\in C^{1,1}_{loc}\cap\mathcal{L}_{s}(\mathbb{R}^{n})$. It can also be defined equivalently via the Caffarelli and Silvestre's extension method (refer to \cite{CS}, see also \cite{CT,S}). The constants $C_{n,s,\theta,\Theta}$ in \eqref{comparison} and $C_{n,s}$ in \eqref{def-fracL} satisfy $C_{n,s,\theta,\Theta}C_{n,s}=1$ provided that $\theta\leq1\leq\Theta$.

\smallskip

In recent years, fractional order operators have attracted more and more attentions. Besides various applications in fluid mechanics, molecular dynamics, relativistic quantum mechanics of stars (see e.g. \cite{CV,Co}) and conformal geometry (see e.g. \cite{CG}), it also has many applications in probability and finance (see \cite{Be,CT}). The fractional Laplacians $(-\Delta)^{s}$ can be understood as the infinitesimal generator of a stable L\'{e}vy diffusion process (see \cite{Be}). The general pseudo-relativistic operators with singular potentials describes a spin zero relativistic particle of charge $e$ and mass $m$ in the Coulomb field of an infinitely heavy nucleus of charge $Z$.

\smallskip

However, the non-locality virtue of these fractional operators makes them difficult to be investigated. To overcome this difficulty, we basically have two approaches. One way is to define these fractional operators via Caffarelli and Silvestre's extension method (see \cite{CS}), so as to reduce the nonlocal problem into a local one in higher dimensions. Another approach is to derive the integral representation formulae of solutions (see \cite{CLO,CLM}). After establishing the equivalence between the fractional order equation and its corresponding integral equation, one can study the equivalent integral equations instead and consequently derive various properties of solutions to the PDEs involving nonlocal fractional operators. These two methods have been applied successfully to study equations involving nonlocal fractional operators, and a series of fruitful results have been derived (see \cite{BCPS,CLO,CLM,CS,CT,FLS,S} and the references therein).

\smallskip

Nevertheless, the above two approaches do not work for the uniformly elliptic nonlocal Bellman operator $\mathrm{F}_{s}$ and general fully nonlinear integro-differential operators (see e.g. \cite{Bi,CS2,CS3,RS}), for instance, the fractional $p$-Laplacians $(-\Delta)^{s}_{p}$ (see e.g. \cite{CL2,CLiu,CQ,CW1,CW3,DLW} for more details).

\smallskip

Therefore, it is desirable for us to develop the method of moving planes directly for the uniformly elliptic nonlocal Bellman operator $\mathrm{F}_{s}$ ($s\in(0,1)$) without going through extension methods or integral representation formulae. Direct moving planes method and sliding method have been introduced for fractional Laplacian $(-\Delta)^{s}$ in \cite{CLL,CW2,DSV}, for fractional $p$-Laplacians $(-\Delta)^{s}_{p}$ in \cite{CL2,CLiu,CW1,CW3} and for pseudo-relativistic Schr\"{o}dinger operators $(-\Delta+m^{2})^{s}$ in \cite{DQW}. These methods have been applied to obtain symmetry, monotonicity and uniqueness of solutions to various equations involving $(-\Delta)^{s}$, $(-\Delta)^{s}_{p}$ or $(-\Delta+m^{2})^{s}$. Sliding method for the uniformly elliptic nonlocal Monge-Amp\`{e}re operator $\mathrm{D}_{s}$ has been developed recently in \cite{CBL}, as applications, the authors derived the monotonicity of solutions to $-\mathrm{D}_{s}u=f(u)$ in bounded domains and the whole space. The uniformly elliptic nonlocal Monge-Amp\`{e}re operator $\mathrm{D}_{s}$ is a special case of the uniformly elliptic nonlocal Bellman operator $\mathrm{F}_{s}$.

\smallskip

The goal of this paper is to establish various maximum principles for the uniformly elliptic nonlocal Bellman operators $\mathrm{F}_{s}$ and $\mathbf{F_{s}}$, as consequences, introduce \emph{the method of moving planes} for $\mathbf{F_{s}}$ and \emph{the sliding method (on general unbounded domains)} for $\mathrm{F}_{s}$ and $\mathbf{F_{s}}$ and derive \emph{multiple applications}. For instance, under broad assumptions on the nonlinearity $f(x,u,\nabla u)$, we prove \emph{symmetry}, \emph{monotonicity} and \emph{uniqueness} results, and \emph{asymptotic properties} for solutions to equations \eqref{PDE} in \emph{bounded domains}, \emph{unbounded domains}, \emph{epigraph} or $\mathbb{R}^{n}$. For related literatures on the nonlocal Bellman equations or the regular second order Bellman equations, please refer to e.g. \cite{AK,Evans,K,Lions} and the references therein.

The methods of moving planes was invented by Alexandroff in the early 1950s. Later, it was further developed by Serrin \cite{Serrin}, Gidas, Ni and Nirenberg \cite{GNN1}, Caffarelli, Gidas and Spruck \cite{CGS}, Chen and Li \cite{CL}, Li \cite{LiC}, Lin \cite{Lin}, Chen, Li and Ou \cite{CLO} and many others. For more literatures on the methods of moving planes, see \cite{BCPS,BDGQ,BN3,CD,CDQ,Cheng1,CHL,CL2,CLL,CQ,CT,CW3,CY,DFQ,DLW,DQ,DQ2,DQW,Li1,Li2,LD,WX} and the references therein.

\subsection{Main results}
In this paper, inspired by the direct moving planes methods for $(-\Delta)^{s}$, $(-\Delta)^{s}_{p}$ and $(-\Delta+m^{2})^{s}$ established in \cite{CL2,CLL,CW3,DQW}, we will establish various maximum principles and introduce \emph{the method of moving planes} for the uniformly elliptic nonlocal Bellman operators $\mathbf{F_{s}}$ and \emph{the sliding method (on general unbounded domains)} for $\mathrm{F}_{s}$ and $\mathbf{F_{s}}$ with $s\in(0,1)$.

\smallskip

The main contents and results in our paper are arranged as follows.

\smallskip

In Section 2, we will establish various maximum principles for \emph{anti-symmetric functions} and give some \emph{immediate applications}. These maximum principles are \emph{key ingredients} in applying the \emph{method of moving planes} for the uniformly elliptic nonlocal Bellman operator $\mathbf{F_{s}}$.

\smallskip

In Section 3, by applying the maximum principles established in Section 2, we introduce the \emph{method of moving planes} for the uniformly elliptic nonlocal Bellman operator $\mathbf{F_{s}}$. As applications, under broad assumptions on the nonlinearity $f(x,u,\nabla u)$, we derive \emph{symmetry}, \emph{monotonicity} and \emph{uniqueness} results for solutions to equations \eqref{PDE} in \emph{bounded domains}, \emph{unbounded domains}, \emph{coercive epigraph} and $\mathbb{R}^{n}$. The admissible choices of the nonlinearity $f(x,u,\nabla u)$ include: $u^{p}(1+|\nabla u|^{2})^{\frac{\sigma}{2}}$ with $p\geq1$ and $\sigma\leq0$, $e^{\kappa u}(1+|\nabla u|^{2})^{\frac{\sigma}{2}}$ with $\kappa\in\mathbb{R}$ and $\sigma\leq0$, $K(x)(1+|\nabla u|^{2})^{\frac{\sigma}{2}}$ with $\sigma\in\mathbb{R}$ and $K(x)$ satisfying certain assumptions including the case of $K(x)\equiv1$, the De Giorgi type nonlinearity $u-u^{3}$ and the Schr\"{o}dinger type nonlinearity $u^p-u$ with $1<p<+\infty$.

\smallskip

Subsection 4.1 is devoted to proving various maximum principles in \emph{unbounded open sets} for $\mathrm{F}_{s}$ and $\mathbf{F_{s}}$. As applications, in subsections 4.2-4.4, under broad assumptions on the nonlinearity $f(u)$, by applying the sliding method for $\mathrm{F}_{s}$ and $\mathbf{F_{s}}$ and the method of moving planes for $\mathbf{F_{s}}$, we derive \emph{monotonicity} and \emph{uniqueness} results, and \emph{asymptotic properties} for solutions to
\begin{equation}\label{PDEf}
  -\mathrm{F}_{s}u(x)=f\left(u(x)\right) \qquad \text{and} \qquad -\mathbf{F_{s}}u(x)=f\left(u(x)\right)
\end{equation}
in epigraph $E$ and $\mathbb{R}^{n}_{+}$, where admissible choices of the nonlinearity $f(u)$ include: the De Giorgi type nonlinearity $u-u^{3}$ and $e^{\kappa u}$ with $\kappa\in\mathbb{R}$. Our results in subsections 4.2-4.3 can be regarded as extensions of the applications of the sliding methods in \cite{CBL} for the uniformly elliptic nonlocal Monge-Amp\`{e}re operator $\mathrm{D}_{s}$ on bounded domain or $\mathbb{R}^{n}$ to general nonlocal Bellman operator $\mathrm{F}_{s}$ on epigraph $E$.

The sliding method was developed by Berestycki and Nirenberg (\cite{BN1,BN2,BN3}). It was used to establish qualitative properties of solutions for PDEs (mainly involving the regular Laplacian $-\Delta$), such as symmetry, monotonicity and uniqueness $\cdots$. For more literatures on the sliding methods for $-\Delta$, $(-\Delta)^{s}$, $(-\Delta)^{s}_{p}$, $(-\Delta+m^{2})^{s}$ or $\mathrm{D}_{s}$, please refer to \cite{BCN1,BCN2,BHM,BN1,BN2,BN3,CBL,CLiu,CW1,CW2,DQW,DSV}.

\begin{rem}\label{rem20}
By using similar ideas and arguments, one can also develop the method of moving planes and the sliding method for the following general fully nonlinear nonlocal operators:
\begin{equation}\label{fno}
\mathrm{G}_{s}(u)(x):=\inf\left\{P.V.\int_{\mathbb{R}^{n}}\frac{G(u(y)-u(x))}{|A^{-1}(y-x)|^{n+2s}}\mathrm{d}y \,\bigg|\, 0<\theta I\leq A\leq\Theta I\right\},
\end{equation}
\begin{equation}\label{fnov}
\mathbf{G_{s}}(u)(x):=\inf\left\{P.V.\int_{\mathbb{R}^{n}}\frac{G(u(y)-u(x))}{|A^{-1}(y-x)|^{n+2s}}\mathrm{d}y \,\bigg|\, 0<\theta I\leq A\leq\Theta I, \, \mathbf{e}\in EV(A)\right\},
\end{equation}
where $G$ is a local Lipschitz continuous function satisfying $G(0)=0$ and $u$ belongs to some appropriate function space. If $G(t)=|t|^{p-2}t$ with $p\geq2$, we denote $\mathrm{G}_{s}:=\mathrm{F}_{s}^{p}$ and $\mathbf{G_{s}}:=\mathbf{F_{s}^{p}}$. It is clear that, when $G(t)=t$, $\mathrm{G}_{s}$ and $\mathbf{G_{s}}$ degenerate into the uniformly elliptic nonlocal Bellman operator $\mathrm{F}_{s}$ and $\mathbf{F_{s}}$ respectively. We leave these open problems to interested readers.
\end{rem}

\subsection{A typical example: the uniformly elliptic nonlocal Monge-Amp\`{e}re operators $\mathrm{D}_{s}$ and $\mathbf{D_{s}}$}
The uniformly elliptic nonlocal Monge-Amp\`{e}re operator $\mathrm{D}_{s}$ was first introduced by Caffarelli and Charro in \cite{CC}:
\begin{eqnarray}\label{Def-MA}
% \nonumber to remove numbering (before each equation)
 && \mathrm{D}_{s}u(x):=\inf\left\{P.V.\int_{\mathbb{R}^{n}}\frac{u(y)-u(x)}{|A^{-1}(y-x)|^{n+2s}}\mathrm{d}y \,\bigg|\, A>0, \det A=1, \lambda_{min}(A)\geq\theta\right\}  \\
 \nonumber && \qquad\quad\,\,\, =\inf\left\{\frac{1}{2}\int_{\mathbb{R}^{n}}\frac{u(x+y)+u(x-y)-2u(x)}{|A^{-1}y|^{n+2s}}\mathrm{d}y \,\bigg|\, A>0, \det A=1, \lambda_{min}(A)\geq\theta\right\},
\end{eqnarray}
where $\theta>0$ is an arbitrarily small constant and $A>0$ means that $A$ is a positive definite square matrix. The conditions $\det A=1$ and $\lambda_{min}(A)\geq\theta$ imply that $\lambda_{max}(A)\leq\theta^{1-n}$. Thus the uniformly elliptic nonlocal Monge-Amp\`{e}re operator $\mathrm{D}_{s}$ is actually a typical example of the uniformly elliptic nonlocal Bellman operator $\mathrm{F}_{s}$ with $\Theta=\theta^{1-n}$. Therefore, all the results in Sections 2-4 in our paper are valid for the uniformly elliptic nonlocal Monge-Amp\`{e}re operators $\mathrm{D}_{s}$ and $\mathbf{D_{s}}$, where $\mathbf{D_{s}}$ is defined by \eqref{Def-MA} with an extra restriction $\mathbf{e}\in EV(A)$ on matrices $A$ for some arbitrarily given vector $\mathbf{e}$.

In \cite{CC}, Caffarelli and Charro also introduced the fractional Monge-Amp\`{e}re operator without uniformly elliptic condition:
\begin{eqnarray}\label{Def-frac}
% \nonumber to remove numbering (before each equation)
 && \mathcal{D}_{s}u(x):=\inf\left\{P.V.\int_{\mathbb{R}^{n}}\frac{u(y)-u(x)}{|A^{-1}(y-x)|^{n+2s}}\mathrm{d}y \,\bigg|\, A>0, \det A=1\right\}  \\
 \nonumber && \qquad\quad\,\,\, =\inf\left\{\frac{1}{2}\int_{\mathbb{R}^{n}}\frac{u(x+y)+u(x-y)-2u(x)}{|A^{-1}y|^{n+2s}}\mathrm{d}y \,\bigg|\, A>0, \det A=1\right\}.
\end{eqnarray}
They also proved in Theorem 3.1 in \cite{CC} that, under certain conditions, the uniformly elliptic nonlocal Monge-Amp\`{e}re operator $\mathrm{D}_{s}$ may coincide with the fractional Monge-Amp\`{e}re operator $\mathcal{D}_{s}$. There is another nonlocal Monge-Amp\`{e}re operator introduced by Caffarelli and Silvestre in \cite{CS3}.

The fractional Monge-Amp\`{e}re operator $\mathcal{D}_{s}$ is closely related to the geometrically and physically interesting second order Monge-Amp\`{e}re operator. In fact, Caffarelli and Charro proved in Appendix A in \cite{CC} that, if $u$ is convex, asymptotically linear, then
\begin{equation}\label{convergence}
  \lim\limits_{s\rightarrow1}\left((1-s)\mathcal{D}_{s}u(x)\right)=\det(D^{2}u(x))^{\frac{1}{n}}
\end{equation}
up to a constant factor that depends only on the dimension $n$. For related literatures on the regular second order Monge-Amp\`{e}re equation
\begin{equation}\label{M-A}
  \det \, D^{2}u=f,
\end{equation}
please refer to e.g. \cite{Ca,CNS,DF,G,GTW,JW,TW,Yau,ZW} and the references therein.

\smallskip

In what follows, we will use $C$ to denote a general positive constant that may depend on $n$, $s$, $\theta$ and $\Theta$, and whose value may differ from line to line.

\section{Maximum principles for anti-symmetric functions}
In this section, we will establish various maximum principles for anti-symmetric functions w.r.t. hyper-planes $T$ perpendicular to $\mathbf{e}$ and give some immediate applications. These maximum principles are key ingredients in applying the method of moving planes for the uniformly elliptic nonlocal Bellman operator $\mathbf{F_{s}}$ associated with vector $\mathbf{e}$.

Let $T$ be any given hyper-plane in $\mathbb{R}^{n}$ perpendicular to $\mathbf{e}$ and $\Sigma$ be the half space on one side of the plane $T$ hereafter. Denote the reflection of a point $x$ with respect to $T$ by $\tilde{x}$. For any symmetric matrix $B$ such that $\mathbf{e}\in EV(B)$, one has
\begin{equation}\label{bi}
  |B(x-y)|\leq |B(x-\tilde{y})|,\qquad \forall x,y \in \Sigma.
\end{equation}

We need some basic properties on the uniformly elliptic nonlocal Bellman operators $\mathrm{F}_{s}$ and $\mathbf{F_{s}}$.
\begin{lem}\label{lem0}
For any $0<s<1$, the uniformly elliptic nonlocal Bellman operators $\mathrm{F}_{s}$ and $\mathbf{F_{s}}$ satisfy: \\
a) $\mathrm{F}_{s}$ is invariant under translation and rotation, $\mathbf{F_{s}}$ is invariant under translation and reflection w.r.t. $T$; \\
b) $\mathrm{F}_{s}(u+v)\geq \mathrm{F}_{s}u+\mathrm{F}_{s}v$, \, $\mathbf{F_{s}}(u+v)\geq \mathbf{F_{s}}u+\mathbf{F_{s}}v$.
\end{lem}

The proof of Lemma \ref{lem0} follows directly from the definitions \eqref{Def} and \eqref{Defv} of $\mathrm{F}_{s}$ and $\mathbf{F_{s}}$, we omit the details.

First, we can prove the following strong maximum principle for anti-symmetric functions.
\begin{lem}(Strong maximum principle for anti-symmetric functions)\label{SMP-anti}
Suppose that $w\in\mathcal{L}_{s}(\mathbb{R}^{n})$ satisfying $w\left(\tilde{x}\right)=-w(x)$ and $w\geq0$ in $\Sigma$. If there exists $x_{0}\in\Sigma$ such that, $w(x_{0})=0$, $w$ is $C^{1,1}$ near $x_{0}$ and $\mathbf{F_{s}}w(x_{0})\leq0$, then $w=0$ a.e. in $\mathbb{R}^n$.
\end{lem}
\begin{proof}
	Since there exists $x_0\in\Sigma$ such that $w(x_0)=\min\limits_{x\in\Sigma}w(x)=0$, it follows that
	\begin{align*}
	0&\geq\ \mathbf{F_{s}}w(x_0) \\
	& =\inf P.V. \int_{\mathbb{R}^{n}} \frac{w(y)-w(x_0)}{|A^{-1}(x_0-y)|^{n+2s}}\mathrm{d}y \\
	&=\inf P.V. \int_{\mathbb{R}^{n}} \frac{w(y)}{|A^{-1}(x_0-y)|^{n+2s}}\mathrm{d}y \\
	& =\inf P.V. \int_{\Sigma} \left(\frac{1}{|A^{-1}(x_0-y)|^{n+2s}}-\frac{1}{|A^{-1}(x_0-\tilde{y})|^{n+2s}}\right)w(y)\mathrm{d}y\\
	&\geq0.
	\end{align*}
	Thus we must have $w=0$ a.e. in $\Sigma$ and hence $w=0$ a.e. in $\mathbb{R}^{n}$. This finishes the proof of Lemma \ref{SMP-anti}.
\end{proof}

\subsection{Maximum principles for anti-symmetric functions in bounded sets}
\begin{thm}[Maximum principle for anti-symmetric functions]\label{MP Anti}
Let $\Omega$ be a bounded open set in $\Sigma$. Assume that $w\in \mathcal{L}_{s}(\mathbb{R}^{n})\cap C_{\text {loc}}^{1,1}(\Omega)$ and is lower semi-continuous on $\overline{\Omega}$. If
	\begin{equation}\label{MP_anti}
	\left\{\begin{array}{ll}{\mathbf{F_{s}} w(x)-c(x)w(x)\leq 0} & {\text {at points} \,\, x\in\Omega \,\, \text{where} \,\, w(x)<0} \\ {w(x) \geq 0} & {\text { in } \Sigma \setminus \Omega} \\ {w\left(\tilde{x}\right)=-w(x)} & {\text { in } \Sigma,}\end{array}\right.
	\end{equation}
	where $c(x)\geq 0$ for any $x\in\left\{x\in\Omega\,|\,w(x)<0\right\}$. Then $w(x) \geq 0$ in $\Omega$.
	
	Furthermore, assume that
	\begin{equation}\label{2MP-anti-2}
	\mathbf{F_{s}}{w}(x)\leq 0 \quad \text{at points} \,\, x\in\Omega \,\, \text{where} \,\, w(x)=0,
	\end{equation}
	then either $w>0$ in $\Omega$ or $w=0$ almost everywhere in $\mathbb{R}^{n}$.
	
	These conclusions hold for unbounded open set $\Omega$ if we further assume that
	$$
	\liminf\limits_{|x| \rightarrow \infty} w(x) \geq 0.
	$$
\end{thm}
\begin{proof}
	If $w$ is not nonnegative, then the lower semi-continuity of $w$ on $\overline{\Omega}$ indicates that
	there exists a $\hat{x}\in \overline{\Omega}$ such that
	$$
	w\left(\hat{x}\right)=\min _{\overline{\Omega}} w<0.
	$$
	One can further deduce from \eqref{MP_anti} that $\hat{x}$ is in the interior of $\Omega$. It follows that
	\begin{equation}\label{P:MP_inequa}
	\begin{aligned}
	&\quad \mathbf{F_{s}}w(\hat{x})\nonumber\\
	&=\inf P.V.\int_{\mathbb{R}^{n}}\frac{w(y)-w(\hat{x})}{|A^{-1}(\hat{x}-y)|^{n+2s}} \mathrm{d}y\nonumber\\
	&=\inf\left[ P.V.\int_{\Sigma}\frac{w(y)-w(\hat{x})}{|A^{-1}(\hat{x}-y)|^{n+2s}} \mathrm{d}y-\int_{\Sigma}\frac{w(\hat{x})+w(y)}{|A^{-1}(\hat{x}-\tilde{y})|^{n+2s}} \mathrm{d}y\right]\nonumber\\
	&=\inf \left[ P.V.\int_{\Sigma}\left(\frac{1}{|A^{-1}(\hat{x}-y)|^{n+2s}}-\frac{1}{|A^{-1}(\hat{x}-\tilde{y})|^{n+2s}}\right)\left(w(y)-w(\hat{x})\right) \mathrm{d}y\right.\\
	&\qquad\quad \left.-2w(\hat{x})\int_{\Sigma}\frac{1}{|A^{-1}(\hat{x}-\tilde{y})|^{n+2s}} \mathrm{d}y\right]\nonumber\\
	&\geq -2w(\hat{x})\inf\int_{\Sigma}\frac{1}{|A^{-1}(\hat{x}-\tilde{y})|^{n+2s}} \mathrm{d}y\nonumber\\
	&>0,\nonumber
	\end{aligned}
	\end{equation}
	which contradicts \eqref{MP_anti}. Hence $w(x)\geq 0$ in $\Omega$.
	
	Now we have proved that $w(x)\geq 0$ in $\Sigma$. If there is some point $\bar{x}\in\Omega$ such that $w\left(\bar{x}\right)=0$, then from \eqref{2MP-anti-2} and Lemma \ref{SMP-anti}, we derive immediately $w=0$ almost everywhere in $\mathbb{R}^{n}$. This completes the proof of Theorem \ref{MP Anti}.
\end{proof}

\begin{rem}\label{rem12}
It is clear from the proof that, in Theorem \ref{MP Anti}, the assumptions ``$w$ is lower semi-continuous on $\overline{\Omega}$" and ``$w\geq0$ in $\Sigma\setminus\Omega$" can be weaken into: ``if $w<0$ somewhere in $\Sigma$, then the negative minimum $\inf\limits_{\Sigma}w(x)$ can be attained in $\Omega$", the same conclusions are still valid. One can also notice that, we only need to assume that $c(x)\geq0$ at points $x\in\Omega$ where $w(x)=\inf\limits_{\Sigma}w<0$ in Theorems \ref{MP Anti}.
\end{rem}

\begin{thm}[Narrow region principle]\label{2NRP}
Let $\Omega$ be a bounded open set in $\Sigma$ which can be contained in the region between $T$ and $T_{\Omega}$, where $T_{\Omega}$ is a hyper-plane that is parallel to $T$. Let $d(\Omega):=dist(T,T_{\Omega})$. Suppose that $w\in \mathcal{L}_{s}(\mathbb{R}^{n})\cap C_{loc}^{1,1}(\Omega)$ and is lower semi-continuous on $\overline{\Omega}$, and satisfies
	\begin{equation}\label{NRP-anti}
	\left\{\begin{array}{ll}{\mathbf{F_{s}} w(x)-c(x)w(x)\leq 0} & {\text {at points} \,\, x\in\Omega \,\, \text{where} \,\, w(x)<0} \\ {w(x) \geq 0} & {\text {in } \Sigma \backslash \Omega} \\ {w\left(\tilde{x}\right)=-w(x)} & {\text {in } \Sigma,}\end{array}\right.
	\end{equation}
	where $c(x)$ is uniformly bounded from below (w.r.t. $d(\Omega)$) in $\{x\in\Omega\,|\,w(x)<0\}$. There exists a constant $C_{n,s,\theta}>0$ such that, if we assume $\Omega$ is narrow in the sense that
	\begin{equation}\label{2NRP-3}
	d(\Omega)^{2s}\left(-\inf\limits_{\left\{x\in\Omega\,|\,w(x)<0\right\}}c(x)\right)<C_{n,s,\theta},
	\end{equation}
	then, $w(x) \geq 0 \text { in } \Omega$. Furthermore, assume that
	\begin{equation}\label{2NRP-4}
	\mathbf{F_{s}}{w}(x)\leq 0 \quad \text{at points} \,\, x\in\Omega \,\, \text{where} \,\, w(x)=0,
	\end{equation}
	then either $w>0$ in $\Omega$ or $w=0$ almost everywhere in $\mathbb{R}^{n}$.
	
	These conclusions hold for unbounded open set $\Omega$ if we further assume that
	$$
	\liminf _{|x| \rightarrow \infty} w(x) \geq 0.
	$$
\end{thm}
\begin{proof}
	Without loss of generalities, we may assume that
	\[ \mathbf{e}=\mathbf{e_1}:=(1,0,\cdots,0),\quad T=\{x\in\mathbb{R}^{n}\,|\,x_{1}=0\} \quad \text{and} \quad \Sigma=\{x\in\mathbb{R}^{n}\,|\,x_{1}<0\},\]
	and hence $\Omega\subseteq\{x\in\mathbb{R}^{n}\,|-d(\Omega)<x_{1}<0\}$.
	
	If $w$ is not nonnegative in $\Omega$, then the lower semi-continuity of $w$ on $\overline{\Omega}$ indicates that, there exists a $\bar{x}\in\overline{\Omega}$ such that
	$$
	w\left(\bar{x}\right)=\min\limits_{\overline{\Omega}} w<0.
	$$
	One can further deduce from \eqref{NRP-anti} that $\bar{x}$ is in the interior of $\Omega$. It follows that
	\begin{equation}\label{2NRP-1}
	\begin{aligned}
	&\quad \mathbf{F_{s}}w(\bar{x}) \\
	&=\inf P.V.\int_{\mathbb{R}^{n}}\frac{w(y)-w(\bar{x})}{|A^{-1}(\bar{x}-y)|^{n+2s}} \mathrm{d}y \\
	&=\inf\left[ P.V.\int_{\Sigma}\frac{w(y)-w(\bar{x})}{|A^{-1}(\bar{x}-y)|^{n+2s}} \mathrm{d}y-\int_{\Sigma}\frac{w(\bar{x})+w(y)}{|A^{-1}(\bar{x}-\tilde{y})|^{n+2s}} \mathrm{d}y\right] \\
	&=\inf \left[ P.V.\int_{\Sigma}\left(\frac{1}{|A^{-1}(\bar{x}-y)|^{n+2s}}-\frac{1}{|A^{-1}(\bar{x}-\tilde{y})|^{n+2s}}\right)\left(w(y)-w(\bar{x})\right) \mathrm{d}y\right.\\
	&\qquad\quad \left.-2w(\bar{x})\int_{\Sigma}\frac{1}{|A^{-1}(\bar{x}-\tilde{y})|^{n+2s}} \mathrm{d}y\right] \\
	&\geq -2w(\bar{x})\inf\int_{\Sigma}\frac{1}{|A^{-1}(\bar{x}-\tilde{y})|^{n+2s}}\mathrm{d}y \\
	&\geq -C_{n,s,\theta}w(\bar{x})\int_{\Sigma}\frac{1}{|\bar{x}-\tilde{y}|^{n+2s}}\mathrm{d}y
	\end{aligned}
	\end{equation}
	Let
	$$D:=\left\{y=(y_1,y')\in\mathbb{R}^{n} \mid d(\Omega)<y_{1}-(\bar{x})_{1}<2d(\Omega),\left|y^{\prime}-\left(\bar{x}\right)^{\prime}\right|<2d(\Omega)\right\}.$$
	 Denote $t:=y_{1}-(\bar{x})_{1}$, $\tau:=\left|y^{\prime}-\left(\bar{x}\right)^{\prime}\right|$, we have
	\begin{equation}\label{2NRP-2}
	\begin{aligned}
	&\int_{\Sigma}\frac{1}{|\bar{x}-\tilde{y}|^{n+2s}} \mathrm{d}y\\
	\geq &\int_{D}\frac{1}{|\bar{x}-y|^{n+2s}} \mathrm{d}y \\
	=&\int_{d(\Omega)}^{2d(\Omega)}\int_{0}^{2d(\Omega)}\frac{\sigma_{n-1}\tau^{n-2}d\tau}{\left(t^{2}+\tau^{2}\right)^{\frac{n}{2}+s}}dt
	=\int_{d(\Omega)}^{2d(\Omega)}\int_{0}^{\frac{2d(\Omega)}{t}}\frac{\sigma_{n-1}(t\rho)^{n-2}t d\rho}{t^{n+2s}\left(1+\rho^{2}\right)^{\frac{n}{2}+s}}dt \\
	=&\int_{d(\Omega)}^{2d(\Omega)}\frac{1}{t^{1+2s}}\int_{0}^{\frac{2d(\Omega)}{t}}\frac{\sigma_{n-1}\rho^{n-2} d\rho}{\left(1+\rho^{2}\right)^{\frac{n}{2}+s}}dt
	\geq \int_{d(\Omega)}^{2d(\Omega)}\frac{1}{t^{1+2s}}\int_{0}^{1}\frac{\sigma_{n-1}\rho^{n-2} d\rho}{\left(1+\rho^{2}\right)^{\frac{n}{2}+s}}dt \\
	\geq&C_{n,s}\int_{d(\Omega)}^{2d(\Omega)}\frac{1}{t^{1+2s}}dt=\frac{C_{n,s}}{d(\Omega)^{2s}},
	\end{aligned}
	\end{equation}
	where we have used the substitution $\rho:=\tau/t$ and $\sigma_{n-1}$ denotes the area of the unit sphere in $\mathbb{R}^{n-1}$. Since $c(x)$ is uniformly bounded from below (w.r.t. $d(\Omega)$) in $\{x\in\Omega\,|\,w(x)<0\}$, then, from \eqref{2NRP-3}, \eqref{2NRP-1} and \eqref{2NRP-2}, we get
	$$
	\mathbf{F_{s}} w\left(\bar{x}\right)-c\left(\bar{x}\right)w\left(\bar{x}\right)
	\geq\left[-\frac{C_{n,s,\theta}}{d(\Omega)^{2s}}-\inf\limits_{\{x\in\Omega\,|\,w(x)<0\}}c(x)\right]w\left(\bar{x}\right)>0,
	$$
	which contradicts \eqref{NRP-anti}.
	
Now we have proved that $w(x)\geq 0$ in $\Sigma$. If there is some point $\bar{x}\in\Omega$ such that $w\left(\bar{x}\right)=0$, then from \eqref{2NRP-4} and Lemma \ref{SMP-anti}, we derive immediately $w=0$ almost everywhere in $\mathbb{R}^{n}$. This finishes the proof of Theorem \ref{2NRP}.
\end{proof}
\begin{rem}\label{rem13}
It is clear from the proof that, in Theorem \ref{2NRP}, the assumptions ``$w$ is lower semi-continuous on $\overline{\Omega}$" and ``$w\geq0$ in $\Sigma\setminus\Omega$" can be weaken into: ``if $w<0$ somewhere in $\Sigma$, then the negative minimum $\inf\limits_{\Sigma}w(x)$ can be attained in $\Omega$", the same conclusions are still valid. One can also notice that, in Theorem \ref{2NRP}, we only need to assume that $c(x)$ is uniformly bounded from below at the negative minimum points of $w$ and $\inf\limits_{\left\{x\in\Omega\,|\,w(x)<0\right\}}c(x)$ can be replaced by the infimum of $c(x)$ over the set of negative minimum points of $w$ in \eqref{2NRP-3}.
\end{rem}

\subsection{Maximum principles for anti-symmetric functions in unbounded sets and immediate applications}

\begin{thm}[Decay at infinity (I)]\label{P:decay}
	Suppose $0\notin\Sigma$. Let $\Omega$ be an unbounded open set in $\Sigma$. Assume $w\in \mathcal{L}_{s}(\mathbb{R}^{n})\cap C_{loc}^{1,1}(\Omega)$ is a solution of
	\begin{equation}\label{P:decay_eq}
	\left\{\begin{array}{ll}{\mathbf{F_{s}} w(x)-c(x)w(x)\leq 0} & {\text {at points} \,\, x\in\Omega \,\, \text{where} \,\, w(x)<0} \\ {w(x) \geq 0} & {\text { in } \Sigma \backslash \Omega} \\ {w\left(\tilde{x}\right)=-w(x)} & {\text { in } \Sigma}\end{array}\right.
	\end{equation}
	with
	\begin{equation}\label{decay_con}
	\liminf\limits_{\substack{x\in\Omega,\,w(x)<0 \\ |x| \rightarrow+\infty}}|x|^{2s}c(x)>-\frac{C_{n,s,\theta}}{4},
	\end{equation}
	where $C_{n,s,\theta}$ is the same constant as in the last inequality in \eqref{Decay-2}. Then there exists a constant $R_{0}>0$ (depending only on $c(x)$, $\theta$, $n$ and $s$, but independent of $w$ and $\Sigma$) such that, if $\hat{x}\in\Omega$ satisfying
	$$
	w\left(\hat{x}\right)=\min _{\overline{\Omega}} w(x)<0,
	$$
	then $\left|\hat{x}\right|\leq R_{0}$.
\end{thm}
\begin{proof}
	Without loss of generalities, we may assume that $\mathbf{e}=\mathbf{e_1}$ and  for some $\lambda\leq0$,
	\[T=\{x\in\mathbb{R}^{n}\,|\,x_{1}=\lambda\} \quad \text{and} \quad \Sigma=\{x\in\mathbb{R}^{n}\,|\,x_{1}<\lambda\}.\]
	
	Since $w\in\mathcal{L}_{s}(\mathbb{R}^{n})\cap C_{loc}^{1,1}(\Omega)$ and $\hat{x}\in\Omega$ satisfying $w\left(\hat{x}\right)=\min\limits_{\overline{\Omega}} w(x)<0$, through similar calculations as \eqref{2NRP-1}, we get
	\begin{equation}\label{Decay-1}
	\mathbf{F_{s}} w\left(\hat{x}\right)
	\geq -2w(\hat{x})\inf\int_{\Sigma}\frac{1}{|A^{-1}(\hat{x}-\tilde{y})|^{n+2s}} \mathrm{d}y.
	\end{equation}
	Note that $\lambda\leq0$ and $\hat{x}\in\Omega$, it follows that $B_{\left|\hat{x}\right|}\left(\bar{x}\right) \subset \left\{x \in \mathbb{R}^{n} | x_{1}>\lambda\right\}$, where $\bar{x}:=\left(2\left|\hat{x}\right|+(\hat{x})_{1},\left(\hat{x}\right)^{\prime}\right)$. Thus we derive that,
	\begin{equation}\label{Decay-2}
	\begin{aligned}
	&\inf\int_{\Sigma}\frac{1}{|A^{-1}(\hat{x}-\tilde{y})|^{n+2s}} \mathrm{d}y\\
	\geq &\inf\int_{B_{\left|\hat{x}\right|}\left(\bar{x}\right)}\frac{1}{|A^{-1}(\hat{x}-y)|^{n+2s}} \mathrm{d}y\\
	\geq &C_{n,s,\theta}\int_{B_{\left|\hat{x}\right|}\left(\bar{x}\right)}\frac{1}{|\hat{x}-y|^{n+2s}} \mathrm{d}y \\
	\geq &C_{n,s,\theta}\int_{B_{\left|\hat{x}\right|}\left(\bar{x}\right)}\frac{1}{3^{n+2s}|\hat{x}|^{n+2s}} \mathrm{d}y\\
	\geq &\frac{C_{n,s,\theta}}{|\hat{x}|^{2s}}.
	\end{aligned}
	\end{equation}
	Then we can deduce from \eqref{P:decay_eq}, \eqref{Decay-1} and \eqref{Decay-2} that
	\begin{align}\label{Decay-3}
	0\geq \mathbf{F_{s}} w(\hat{x})-c(\hat{x})w(\hat{x})\geq\left[-\frac{C_{n,s,\theta}}{|\hat{x}|^{2s}}-c(\hat{x})\right]w(\hat{x}).
	\end{align}
	It follows from $w(\hat{x})<0$ and \eqref{Decay-3} that
	\begin{equation}\label{Decay-4}
	|\hat{x}|^{2s}c(\hat{x})\leq-C_{n,s,\theta}<0.
	\end{equation}
	From \eqref{decay_con}, we infer that there exists a $R_{0}$ sufficiently large such that, for any $|x|>R_{0}$,
	\begin{equation}\label{Decay-5}
	|x|^{2s}c(x)\geq-\frac{C_{n,s,\theta}}{2}.
	\end{equation}
	Combining \eqref{Decay-4} and \eqref{Decay-5}, we arrive at $|\hat{x}|\leq R_{0}$. This completes the proof of Theorem \ref{P:decay}.
\end{proof}

\begin{rem}\label{rem14}
It is clear from the proofs of Theorems \ref{MP Anti}, \ref{2NRP} and \ref{P:decay} that, the assumption ``$\mathbf{F_{s}}w(x)-c(x)w(x)\leq0$ at points $x\in\Omega$ where $w(x)<0$" can be weaken into: ``$\mathbf{F_{s}}w(x)-c(x)w(x)\leq0$ at points $x\in\Omega$ where $w(x)=\inf\limits_{\Sigma}w<0$", the same conclusions in Theorems \ref{MP Anti}, \ref{2NRP} and \ref{P:decay} are still valid.
\end{rem}

\begin{thm}[Maximum principle for anti-symmetric functions in unbounded domains]\label{MP_anti_ubdd}
Assume that $w \in \mathcal{L}_{s}(\mathbb{R}^{n})\cap C_{\text {loc}}^{1,1}(\Sigma)$ is bounded from below in $\Sigma$ and $w\left(\tilde{x}\right)=-w(x)$ in $\Sigma$, where $\tilde{x}$ is the reflection of $x$ with respect to $T$. Suppose that, at any points $x \in\Sigma$ such that $w(x)<0$, $w$ satisfies
	\begin{equation}\label{MP-1}
		\mathbf{F_{s}} w(x)-c(x)w(x)\leq 0,
	\end{equation}
	where $c(x)\geq 0$ in $\{x\in\Sigma \mid w(x)<0\}$. Then
	\begin{equation}\label{MP-2}
		w(x) \geq 0, \quad \forall x \in\Sigma.
	\end{equation}
	Furthermore, assume that
	\begin{equation}\label{MP-condition}
		\mathbf{F_{s}}{w}(x)\leq0 \quad \text{at points} \,\, x\in\Sigma \,\, \text{where} \,\, w(x)=0,
	\end{equation}
	then either $w>0$ in $\Sigma$ or $w=0$ in $\mathbb{R}^{n}$.
\end{thm}		
\begin{proof}
	Suppose that \eqref{MP-2} is false, since $w$ is bounded from below, we have $-\infty<m:=\inf\limits_{\Sigma} w(x)<0$. Hence, there exists sequences $x^k\in\Sigma$ and $0<\alpha_k<1$ with $\alpha_k\rightarrow 1$ as $k\rightarrow \infty$ such that
	\begin{equation}\label{MP1-4}
		w(x^k)\leq \alpha_k m.
	\end{equation}
	Without loss of generalities, we may assume that
$$\mathbf{e}=\mathbf{e_1}=(1,0,\cdots,0), \quad T=\{x\in\mathbb{R}^n\mid x_1=0\}, \quad \Sigma=\{x\in\mathbb{R}^n\mid x_1<0\}.$$
	Then $\tilde{x}=(-x_1,x_2,\cdots,x_n)$. We denote $d_k:=\frac{1}{2}dist(x^k, T)$.
	Let
	\begin{equation*}
		\psi(x)=\begin{cases}e^{\frac{|x|^{2}}{|x|^2-1}}, \quad\,|x|<1\\ 0, \qquad\,\, \quad|x|\geq 1.\end{cases}
	\end{equation*}
	It is well known that $\psi\in C_0^\infty(\mathbb{R}^{n})$, thus $|\mathbf{F_{s}}\psi(x)|\leq C_{0}$ for all $x \in \mathbb{R}^n$. Moreover, $\mathbf{F_{s}}\psi(x)\sim|x|^{-n-2s}$ as $|x|\rightarrow +\infty$.
	
	Set \begin{equation}\label{MP-def}
	\psi_k(x):=\psi\left(\frac{x-\widetilde{(x^k)}}{d_k}\right) \quad \text{and} \quad \tilde{\psi_k}(x)=\psi_k(\tilde{x})=\psi\left(\frac{x-x^k}{d_k}\right).
	\end{equation}
	Then $\tilde{\psi_k}-\psi_k$ is anti-symmetric with respect to $T$. Now pick $\varepsilon_k=-(1-\alpha_k)m$, then we have
	$$w(x^k)-\varepsilon_k[\tilde{\psi_k}-\psi_k](x^k)\leq m.$$
	We denote $$w_k(x):=w(x)-\varepsilon_k[\tilde{\psi_k}-\psi_k](x).$$ Then $w_k$ is also anti-symmetric with respect to $T$.
	
	Since for any $x\in\Sigma\setminus B_{d_k}(x^k)$, $w(x)\geq m$ and $\tilde{\psi_k}(x)=\psi_k(x)=0$, we have
	$$w_k(x^k)\leq m\leq w_k(x), \quad \forall \,\,x\in\Sigma\setminus B_{d_k}(x^k).$$
	Hence the infimum of $w_k(x)$ in $\Sigma$ is achieved in $B_{d_k}(x^k)$. Consequently, there exists a point $\overline{x}^k\in B_{d_k}(x^k)$ such that
	\begin{equation}\label{MP-3}
		w_k(\overline{x}^k)=\inf\limits_{x\in\Sigma} w_k(x)\leq m<0.
	\end{equation}
	By the choice of $\varepsilon_k$, it is easy to verify that $w(\bar{x}^k)\leq \alpha_k m<0$.
	
	Next, we will evaluate the upper bound and the lower bound of $\mathbf{F_{s}}w_k(\bar{x}^k)$.
	
	We first obtain the upper bound by direct calculations:
	\begin{align}\label{MP-4}
		&\quad \mathbf{F_{s}} w_{k}(\bar{x}^k)\nonumber\\
		&=\mathbf{F_{s}}\{w-\varepsilon_k[\tilde{\psi_k}-\psi_k]\}(\bar{x}^k)\nonumber\\
		&\leq \mathbf{F_{s}} w(\bar{x}^k)-\varepsilon_k \mathbf{F_{s}}[\tilde{\psi_k}-\psi_k](\bar{x}^k)\\
		&\leq c(\bar{x}^k)w(\bar{x}^k)+\frac{2C_{0}\varepsilon_k}{d_k^{2s}}\nonumber\\
		&\leq \frac{2C_{0}\varepsilon_k}{d_k^{2s}}.\nonumber
	\end{align}
	
On the other hand, let
\[D_k:=\left\{x=(x_1,x')\in\mathbb{R}^{n} \mid -d_k<x_1<0, |x'-(\bar{x}^k)'|<4d_{k}\right\},\]
then through similar calculations as in \eqref{2NRP-2}, we get the following lower bound:
	\begin{align}\label{MP-5}
		&\quad \mathbf{F_{s}}w_k(\bar{x}^k)\nonumber\\
		&=\inf P.V.\int_{\mathbb{R}^{n}}\frac{w_k(y)-w_k(\bar{x}^k)}{|A^{-1}(\bar{x}^k-y)|^{n+2s}} \mathrm{d}y\nonumber\\
		&=\inf\left[ P.V.\int_{\Sigma}\frac{w_k(y)-w_k(\bar{x}^k)}{|A^{-1}(\bar{x}^k-y)|^{n+2s}} \mathrm{d}y-\int_{\Sigma}\frac{w_k(\bar{x}^k)+w_k(y)}{|A^{-1}(\bar{x}^k-\tilde{y})|^{n+2s}} \mathrm{d}y\right]\nonumber\\
		&=\inf \left[ P.V.\int_{\Sigma}\left(\frac{1}{|A^{-1}(\bar{x}^k-y)|^{n+2s}}-\frac{1}{|A^{-1}(\bar{x}^k-\tilde{y})|^{n+2s}}\right)\left(w_k(y)-w_k(\bar{x}^k)\right) \mathrm{d}y\right.\\
		&\qquad\left.-2w_k(\bar{x}^k)\int_{\Sigma}\frac{1}{|A^{-1}(\bar{x}^k-\tilde{y})|^{n+2s}} \mathrm{d}y\right]\nonumber\\
		&\geq -2w_k(\bar{x}^k)\inf\int_{\Sigma}\frac{1}{|A^{-1}(\bar{x}^k-\tilde{y})|^{n+2s}} \mathrm{d}y\nonumber\\
		&\geq -C_{n,s,\theta}w_k(\bar{x}^k)\int_{\Sigma}\frac{1}{|\bar{x}^k-\tilde{y}|^{n+2s}} \mathrm{d}y\nonumber\\
        &\geq -C_{n,s,\theta}w_k(\bar{x}^k)\int_{D_{k}}\frac{1}{|\bar{x}^k-\tilde{y}|^{n+2s}} \mathrm{d}y\nonumber\\
     	&\geq -\frac{C_{n,s,\theta}w_k(\bar{x}^k)}{d_k^{2s}}.\nonumber
	\end{align}
Combining \eqref{MP-4} and \eqref{MP-5}, we derive
	$$-\frac{C_{n,s,\theta}w_k(\bar{x}^k)}{d_k^{2s}}\leq \frac{2C_{0}\varepsilon_k}{d_k^{2s}}.$$
	Noticing that $w_k(\bar{x}^k)\leq m<0$ and $\varepsilon_k=-(1-\alpha_k)m$, we have
	$$C_{n,s,\theta}\leq 2C_{0}(1-\alpha_k),$$
which will lead to a contradiction if we let $k\rightarrow+\infty$. Thus we have proved that $w(x)\geq 0$ in $\Sigma$.

If there is some point $\bar{x}\in\Sigma$ such that $w\left(\bar{x}\right)=0$, then from \eqref{MP-condition} and Lemma \ref{SMP-anti}, we derive immediately $w=0$ almost everywhere in $\mathbb{R}^{n}$. This concludes our proof of Theorem \ref{MP_anti_ubdd}.
\end{proof}		

From the proof of Theorem \ref{MP_anti_ubdd}, we can deduce the following narrow region principle in unbounded open sets, which improves the \emph{Narrow region principle} (Theorem \ref{2NRP}).
\begin{thm}[Narrow region principle in unbounded open sets]\label{NRP-Anti-ubdd}
	Let $\Omega\subseteq\Sigma$ be an open set (possibly unbounded and disconnected) which can be contained in the region between $T$ and $T_{\Omega}$, where $T_{\Omega}$ is a hyper-plane that is parallel to $T$. Let $d(\Omega):=dist(T,T_{\Omega})$. Suppose that $w\in \mathcal{L}_{s}(\mathbb{R}^{n})\cap C_{loc}^{1,1}(\Omega)$ is bounded from below and satisfies
	\begin{equation}\label{NRP-anti-ubdd}
	\begin{cases}
	\mathbf{F_{s}} w(x)-c(x)w(x)\leq 0\qquad \text {at points} \,\, x\in\Omega \,\, \text{where} \,\, w(x)<0 \\  w(x) \geq 0\qquad  \text {in } \,\,\Sigma \backslash \Omega \\
    w\left(\tilde{x}\right)=-w(x)\qquad  \text {in}\,\, \Sigma,
   \end{cases}
	\end{equation}
where $c(x)$ is uniformly bounded from below (w.r.t. $d(\Omega)$) in $\{x\in\Omega \mid w(x)<0\}$. We assume $\Omega$ is narrow in the sense that
	\begin{equation}\label{2NRP-3-ubdd}
	d(\Omega)^{2s}\left(-\inf\limits_{\left\{x\in\Omega\,|\,w(x)<0\right\}}c(x)\right)<2^{2s-1}C_{n,s,\theta},
	\end{equation}
	where $C_{n,s,\theta}$ is the same constant as in \eqref{MP-5}. Then, $w(x)\geq 0$ in $\Omega$. Furthermore, assume that
	\begin{equation}\label{2NRP-5-ubdd}
	\mathbf{F_{s}}{w}(x)\leq0 \quad \text{at points} \,\, x\in\Omega \,\, \text{where} \,\, w(x)=0,
	\end{equation}
then either $w>0$ in $\Omega$ or $w=0$ almost everywhere in $\mathbb{R}^{n}$.
\end{thm}
\begin{proof}
Theorem \ref{NRP-Anti-ubdd} can be proved by using quite similar arguments as in the proof of Theorem \ref{MP_anti_ubdd}, we only mention some key ingredients.

Indeed, combining \eqref{MP-4} and \eqref{MP-5}, we derive
    \begin{align}\label{MP-6}
    -\frac{C_{n,s,\theta}w_k(\bar{x}^k)}{d_k^{2s}}\leq c(\bar{x}^k)w_k(\bar{x}^k)+\frac{2C_{0}\varepsilon_k}{d_k^{2s}}\leq \left(\inf\limits_{\left\{x\in\Omega\,|\,w(x)<0\right\}}c(x)\right)w_k(\bar{x}^k)+\frac{2C_{0}\varepsilon_k}{d_k^{2s}}.
    \end{align}
    For $k$ sufficiently large such that $\alpha_k>1-\frac{C_{n,s,\theta}}{4C_{0}}$, recall that $\varepsilon_{k}=-(1-\alpha_{k})m$ and $w_k(\bar{x}^k)\leq m$, we have $\frac{2C_{0}\varepsilon_k}{d_k^{2s}}\leq-\frac{C_{n,s,\theta}w_k(\bar{x}^k)}{2d_k^{2s}}$. Then, we infer from \eqref{MP-6} and $d_k\leq \frac{d(\Omega)}{2}$ that
    $$\inf\limits_{\left\{x\in\Omega\,|\,w(x)<0\right\}}c(x)\leq -\frac{4^sC_{n,s,\theta}}{2d(\Omega)^{2s}},$$
    which contradicts \eqref{2NRP-3-ubdd}. This finishes the proof of Theorem \ref{NRP-Anti-ubdd}.
\end{proof}
\begin{rem}\label{rem8}
	In Theorem \ref{NRP-Anti-ubdd}, we allow the open set $\Omega$ to be unbounded without the additional assumption $\liminf\limits_{|x|\rightarrow+\infty}w(x)\geq0$ in Theorem \ref{2NRP}.
\end{rem}

From the proof of Theorem \ref{MP_anti_ubdd}, we can also deduce the following maximum principle in unbounded domains, which improves the \emph{Decay at infinity (I)} (Theorem \ref{P:decay}).
\begin{thm}[Decay at infinity (II)]\label{dati}
Let $\Omega$ be an unbounded open set in $\Sigma$. Assume $w\in \mathcal{L}_{s}(\mathbb{R}^{n})\cap C_{loc}^{1,1}(\Omega)$ is bounded from below and satisfies
	\begin{equation}\label{dati-eq}
	\left\{\begin{array}{ll}{\mathbf{F_{s}} w(x)-c(x)w(x)\leq 0} & {\text {at points} \,\, x\in\Omega \,\, \text{where} \,\, w(x)<0} \\ {w(x) \geq 0} & {\text { in } \Sigma \backslash \Omega} \\ {w\left(\tilde{x}\right)=-w(x)} & {\text { in } \Sigma}\end{array}\right.
	\end{equation}
	with
	\begin{equation}\label{dati-con}
	\liminf\limits_{\substack{x\in\Omega,\,w(x)<0 \\ |x|\rightarrow \infty}}|x|^{2s}c(x)>-\frac{C_{n,s,\theta}}{4},
	\end{equation}
	where $C_{n,s,\theta}$ is the same constant as in \eqref{MP-5}.
	
	Then there exists a $R_{0}>0$ large enough and $\alpha_{0}\in (0,1)$ close enough to $1$ ($R_{0}$ and $\alpha_{0}$ are independent of $w$ and $\Sigma$) such that, if $\hat{x}\in\Omega$ satisfying
	$$
	w\left(\hat{x}\right)\leq\alpha_{0}\inf _{\Omega}w(x)<0,
	$$
	then $\left|\hat{x}\right|\leq R_{0}$.
\end{thm}
\begin{proof}
	Theorem \ref{dati} can be proved via similar contradiction arguments as Theorem \ref{MP_anti_ubdd}.
	
	Indeed, suppose on the contrary that there exists sequences $\{x^{k}\}\in\Omega$ and $\{\alpha_{k}\}\in(0,1)$ such that
	\begin{equation}\label{dati-1}
	|x^{k}|\rightarrow+\infty, \quad \alpha_{k}\rightarrow 1, \quad \text{and} \quad w(x^{k})\leq\alpha_{k}\inf _{\Omega}w(x)<0.
	\end{equation}
	Then, similar calculations as in the proof of Theorem \ref{MP_anti_ubdd} (see \eqref{MP-4} and \eqref{MP-5}) give that
	\begin{align}\label{MP-7}
	-\frac{C_{n,s,\theta}w_k(\bar{x}^k)}{d_k^{2s}}\leq c(\bar{x}^k)w_k(\bar{x}^k)+\frac{2C_{0}\varepsilon_k}{d_k^{2s}}.
	\end{align}
Now we take $k$ sufficiently large such that $\alpha_k>1-\frac{C_{n,s,\theta}}{4C_{0}}$. Recall that $d_{k}:=\frac{1}{2}dist(x^{k},T)$ and $\bar{x}^{k}\in B_{d_{k}}(x^{k})$, we infer from \eqref{MP-7} that, for $k$ large enough,
	$$c(\bar{x}^k)\leq -\frac{C_{n,s,\theta}}{2d_k^{2s}}\leq -\frac{C_{n,s,\theta}}{4|\bar{x}^k|^{2s}},$$
	which contradicts \eqref{dati-con} if we let $k\rightarrow+\infty$.
\end{proof}
\begin{rem}\label{rem11}
We say \emph{Decay at infinity (II)} Theorem \ref{dati} improved \emph{Decay at infinity (I)} Theorem \ref{P:decay} in the sense that, not only the positions of minimal points but also the positions of ``almost" negative minimal points were controlled by a radius $R_{0}$ in Theorem \ref{dati}. Theorem \ref{dati} also tell us that, if $\inf _{\Omega}w(x)<0$, then $\Omega\cap B_{R_{0}}(0)\neq\emptyset$ and the negative minimum can be attained in $\Omega\cap B_{R_{0}}(0)$.
\end{rem}

As an immediate application of Theorem \ref{MP_anti_ubdd}, we can obtain the following Liouville type Theorem in $\mathbb{R}^{n}$ in the direction $\mathbf{e}$, which indicates that the solution $u$ only depends on $n-1$ variables. For Liouville theorem on $s$-harmonic functions in $\mathbb{R}^{n}$, please refer to \cite{F} and the references therein.
\begin{thm}(Liouville Theorem in the direction $\mathbf{e}$)\label{Liouville}
Assume that $u\in\mathcal{L}_{s}(\mathbb{R}^{n})\cap C_{\text {loc}}^{1,1}(\mathbb{R}^{n})$ is bounded and satisfies
	\begin{equation}\label{MP-8}
	\mathbf{F_{s}}u(x)=0 \quad \text{in}\,\,\mathbb{R}^{n}.
	\end{equation}
Then
$$u(x)\equiv u(x+t\mathbf{e}), \qquad \forall \,x\in\mathbb{R}^{n},\,\, \forall \, t\in \mathbb{R}.$$
\end{thm}
\begin{proof}
Let $T$ be any hyper-plane perpendicular to $\mathbf{e}$, $\Sigma$ be the half space on one side of the plane $T$. Set $\tilde{u}(x)=u(\tilde{x})$ and $w(x)=\tilde{u}(x)-u(x)$ for all $x\in \Sigma$, where $\tilde{x}$ is the reflection of $x$ with respect to $T$. Then, $w\in\mathcal{L}_{s}(\mathbb{R}^{n})\cap C_{\text {loc}}^{1,1}(\mathbb{R}^{n})$ is bounded, and at any points $x\in\Sigma$ where $w(x)<0$, one has $\mathbf{F_{s}} w(x)\leq \mathbf{F_{s}}\tilde{u}(x)-\mathbf{F_{s}}u(x)=0$. Therefore, applying Theorem \ref{MP_anti_ubdd}, we arrive immediately $w\geq 0$ in $\Sigma$. Similarly, we can prove that $w\geq 0$ in $\mathbb{R}^n\setminus\Sigma$. Hence $w\equiv0$ in $\mathbb{R}^{n}$, and $u$ is symmetric with respect to $T$. Since the hyper-plane $T$ (perpendicular to $\mathbf{e}$) is arbitrary, we must have $u(x)\equiv u(x+t\mathbf{e})$ for all $x\in\mathbb{R}^{n}$ and $t\in \mathbb{R}$. This finishes the proof of Theorem \ref{Liouville}.
\end{proof}

Next, let us consider the following equation
\begin{equation}\label{MP-9}
-\mathbf{F_{s}} u(x)=f(u(x)), \quad \forall \,\, x\in \mathbb{R}^n.
\end{equation}
Assume without loss of generalities that $\mathbf{e}=\mathbf{e_n}:=(0,\cdots,0,1)$. As another application of Theorem \ref{MP_anti_ubdd}, we derive the following monotonicity result on \eqref{MP-9}.
\begin{thm}\label{Mono_1}
	Suppose $u\in\mathcal{L}_{s}(\mathbb{R}^{n})\cap C_{\text {loc}}^{1,1}(\mathbb{R}^{n})$ is a solution of \eqref{MP-9}, and
	$$|u(x)|\leq 1,\quad\forall \, x\in \mathbb{R}^{n},$$
	\begin{equation}\label{MP-10}
	\lim\limits_{x_{n}\rightarrow\pm \infty}u(x',x_n)=\pm1 \quad \text{uniformly w.r.t.} \,\, x' \in \mathbb{R}^{n-1}.
	\end{equation}
	Assume there exists a $\delta>0$ such that
	\begin{equation}\label{MP-11}
	f(t)\,\,\text{is non-increasing on}\,\, [-1,-1+\delta] \cup [1-\delta,1],
	\end{equation}
	then there exists $M>0$ such that, $u(x)$ is strictly monotone increasing w.r.t. $x_{n}$ provided that $|x_{n}|>M$.
\end{thm}
\begin{proof}
For arbitrary $\lambda\in\mathbb{R}$, let $T_{\lambda}:=\left\{x \in \mathbb{R}^{n} \,|\, x_{n}=\lambda\right\}$, $\Sigma_{\lambda}:=\left\{x \in \mathbb{R}^{n} \,|\, x_{n}>\lambda\right\}$ be the region above the plane, and $x^{\lambda}:=\left(x_{1}, x_{2}, \ldots, 2 \lambda-x_{n}\right)$ be the reflection of point $x$ about the plane $T_{\lambda}$.
	
We only need to show that $w_{\lambda}(x):=u_\lambda(x)-u(x)\leq0$ in $\Sigma_{\lambda}$ for any $\lambda$ with $|\lambda|$ sufficiently large, where $u_\lambda(x):=u(x^{\lambda})$. By the assumption \eqref{MP-10}, there exists $M>0$ such that $u(x)\in[-1,-1+\delta]\cup[1-\delta,1]$ for any $x$ with $|x_n|>M$. Consequently, for any $|\lambda|>M$, at any point $x\in \Sigma_{\lambda}$ where $w_\lambda(x)=u(x^\lambda)-u(x)>0$, we infer from assumption \eqref{MP-11} that $\mathbf{F_{s}}w_\lambda(x)\leq\mathbf{F_{s}}u_\lambda(x)-\mathbf{F_{s}}u(x)=-f(u(x^\lambda))+f(u(x))\geq 0$. Therefore, we deduce from Theorem \ref{MP_anti_ubdd} that $w_{\lambda}(x)\leq 0$ in $\Sigma_{\lambda}$ for all $\lambda$ with $|\lambda|>M$.
	
Now, suppose that there exists a $\widetilde{\lambda}\in(-\infty,-M)\cup(M,+\infty)$ and a point $\hat{x}\in\Sigma_{\widetilde{\lambda}}$ such that $w_{\widetilde{\lambda}}(\hat{x})=0$. Then, it follows that
	\begin{equation}\label{MP-10'}
	\mathbf{F_{s}}w_{\widetilde{\lambda}}(\hat{x})\leq -f(u((\hat{x})^{\widetilde{\lambda}}))+f(u(\hat{x}))=0,
	\end{equation}
	and hence we can derive from Lemma \ref{SMP-anti} immediately that $w_{\widetilde{\lambda}}(x)=0$ almost everywhere in $\mathbb{R}^{n}$, which contradicts assumption \eqref{MP-10}. Thus $w_{\lambda}(x):=u(x^{\lambda})-u(x)<0$ in $\Sigma_{\lambda}$ for all $\lambda$ with $|\lambda|>M$. This finishes our proof of Theorem \ref{Mono_1}.
\end{proof}

\begin{rem}\label{rem1}
One should note that the De Giorgi type nonlinearity $f(u)=u-u^{3}$ satisfies condition \eqref{MP-11}.
\end{rem}

\section{The method of moving planes for $\mathbf{F_{s}}$ and its applications}

In this Section, by using various maximum principles for anti-symmetric functions established in Section 2, we will apply the method of moving planes to investigate symmetry and monotonicity of solutions to various problems involving the uniformly elliptic nonlocal Bellman operators $\mathbf{F_{s}}$ ($s\in(0,1)$) associated with vector $\mathbf{e}$.

We investigate the monotonicity and symmetry properties of nonnegative solutions to the following nonlinear Dirichlet problem:
\begin{equation}\begin{cases}\label{fractional-MA}
-\mathbf{F_{s}}u(x)=f(x,u(x),\nabla u(x)) \qquad \text{in}\ \ \Omega,\\ \quad u>0 \qquad \text{in}\ \ \Omega,\\
\quad u\equiv0 \qquad \text{in}\ \ \mathbb{R}^n\setminus \Omega,
\end{cases}\end{equation}
where $\Omega$ is a (bounded or unbounded) domain, coercive epigraph or the whole space $\mathbb{R}^{n}$.

\subsection{Bounded domain}

Assume without loss of generalities that $\mathbf{e}=\mathbf{e_1}$ and let $\Omega$ be a bounded domain in $\mathbb{R}^n$ which is convex in $x_1$-direction. We say that a domain $\Omega$ is convex in $x_1$-direction, if and only if, $(\bar{x}_1, x')$, $(x_1, x') \in\Omega$ imply that $(t\bar{x}_1+(1-t)x_1, x')\in \Omega$ for any $0<t<1$.

Let $\mathcal{F}$ be the collection of functions $f(x, u, \mathbf{p}): \, \Omega\times\mathbb{R}\times\mathbb{R}^n\rightarrow\mathbb{R}$ which is locally Lipschitz in $u$, uniformly in $\mathbf{p}$ and locally uniformly in $x$: for any $M>0$ and any compact subset $K\subset\overline{\Omega}$, there exists $C_{K,M}>0$ such that, $\forall \, u_1,u_2\in[-M, M]$, $\forall \, x\in K$ and $\forall \, \mathbf{p}\in\mathbb{R}^{n}$,
$$|f(x, u_1, \mathbf{p})-f(x, u_2, \mathbf{p})|\leq C_{K,M}|u_1-u_2|.$$

By applying the method of moving planes for $\mathbf{F_{s}}$, we will prove the following monotonicity and symmetry results for \eqref{fractional-MA} in bounded domain $\Omega$. For related results for $-\Delta$, $(-\Delta)^{s}$, $(-\Delta+m^{2})^{s}$, $(-\Delta)^{s}_{p}$ or second order Monge-Amp\`{e}re operator, please refer to \cite{BN1,BN2,BN3,Cheng1,CHL,CL2,CLM,DLW,DQW,GNN1,Li1,Li2}.

\begin{thm}\label{boundedthm}
Let $\Omega\subset\mathbb{R}^{n}$ be a bounded domain which is convex in $x_1$-direction and symmetric w.r.t. $\{x\in\mathbb{R}^{n} \mid x_{1}=0\}$. Suppose that $u\in C^{1,1}_{loc}(\Omega)\cap C(\overline{\Omega})$ solves \eqref{fractional-MA} for $s\in(0,1)$ and $f(x, u, \mathbf{p})\in\mathcal{F}$. If $f(x,u,\mathbf{p})$ satisfies
\begin{equation}\begin{cases}\label{assump-f}
f(x_1, x', u, p_1, p_2,\cdots, p_n)\leq f(\bar{x}_1, x', u, -p_1, p_2,\cdots, p_n),\\
\ \ \ \ \ \forall \,\,-x_1, p_1\geq0, \,\, x_1\leq \bar{x}_1\leq -x_1.
\end{cases}\end{equation}
Then $u(x_1, x')$ is strictly increasing in the left half of $\Omega$ in $x_1$-direction and
$$u(x_1, x')\leq u(-x_1, x'), \quad \forall \,\, x_1<0, \,\,(x_1,x')\in \Omega.$$
Moreover, if $f(x_1, x', u, p_1, p_2,\cdots, p_n)=f(-x_{1}, x', u, -p_1, p_2,\cdots, p_n)$ for any $x_{1}p_{1}\leq0$, then
$$u(x_1, x')=u(-x_1, x'), \quad \forall \,\, x_1<0, \,\,(x_1,x')\in \Omega.$$
\end{thm}
\begin{proof}
Since $\Omega$ is bounded, we may assume that $\Omega\subset\{|x_1|\leq1\}$ and $\partial\Omega\cap\{x_1=-1\}\neq\emptyset$.

In order to carry out the moving planes procedure, we need to define some useful notations. For any $\lambda\in\mathbb{R}$, let $T_\lambda$ be the hyperplane in $\mathbb{R}^n$ given by
\begin{equation}\label{nota-1}
  T_\lambda:=\{x=(x_1,x')\in \mathbb{R}^n \mid x_1=\lambda\},
\end{equation}
and $x^\lambda:=(2\lambda-x_1,x')$ be the reflection of $x$ w.r.t. the plane $T_\lambda$. Denote
\begin{equation}\label{nota-2}
  \Sigma_\lambda:=\{x\in\mathbb{R}^n \mid x_1<\lambda\}, \qquad \tilde{\Sigma}_\lambda:=\{x\in\mathbb{R}^n \mid x_1>\lambda\}
\end{equation}
and
\begin{equation}\label{nota-3}
  u_\lambda(x):=u(x^\lambda), \quad w_\lambda(x):=u_\lambda(x)-u(x).
\end{equation}
Since $\Omega$ is convex in $x_1$-direction and symmetric w.r.t. $T_{0}$ and $u$ satisfies \eqref{fractional-MA}, one has $w_{\lambda}\geq0$ in $\Sigma_{\lambda}\setminus\Omega$ for any $\lambda\in(-\infty,0]$ and $w_{\lambda}\not\equiv0$ in $\Sigma_{\lambda}\setminus\Omega$ for any $\lambda\in(-\infty,0)$. Our goal is to show that $w_{\lambda}>0$ in $\Sigma_{\lambda}\cap\Omega$ for any $\lambda\in(-1,0)$.

\medskip

We will carry out the proof of Theorem \ref{boundedthm} by two steps.

\medskip

\emph{Step 1.} We will show via contradiction arguments that there exists $\epsilon>0$ small enough such that, for any $-1<\lambda\leq-1+\epsilon$,
\begin{equation}\label{start}
w_\lambda(x)\geq 0, \qquad \forall \,\, x\in\Sigma_\lambda\cap\Omega.
\end{equation}

Suppose \eqref{start} is not true, there exists a sequence $\{\lambda_{k}\}\subset(-1,0)$ satisfying $\lambda_k\rightarrow-1$ as $k\rightarrow+\infty$ such that
\begin{equation}\label{eqn-0}
  \inf\limits_{\Sigma_{\lambda_{k}}\cap\Omega}w_{\lambda_{k}}=\inf\limits_{\lambda\in(-1,\lambda_{k}]}\inf\limits_{x\in\Sigma_\lambda}w_\lambda(x)<0.
\end{equation}
Consequently, there exists $x^k\in\Sigma_{\lambda_{k}}\cap\Omega$ such that
\begin{equation}\label{eqn-1}
  w_{\lambda_{k}}(x^{k})=\inf\limits_{\Sigma_{\lambda_{k}}\cap\Omega}w_{\lambda_{k}}=\inf\limits_{\Sigma_{\lambda_{k}}}w_{\lambda_{k}}<0.
\end{equation}
It follows directly from \eqref{eqn-0} and \eqref{eqn-1} that $\frac{\partial w_{\lambda}}{\partial\lambda}|_{\lambda=\lambda_k}(x^k)\leq0$, and hence $(\partial_{x_1}u)[(x^k)^{\lambda_k}]\leq0$. Note that $x^k$ is the interior minimum of $w_{\lambda_k}(x)$, then one has $\nabla_{x}w_{\lambda_k}(x^k)=0$, i.e.,
\begin{equation}\label{nabla}
(\nabla_{x}u_{\lambda_k})(x^k)=(\nabla_{x}u)(x^k).
\end{equation}
By the assumption \eqref{assump-f} in Theorem \ref{boundedthm} and \eqref{fractional-MA}, we have
\begin{equation}\begin{split}\label{contradict}
& \quad \mathbf{F_{s}}w_{\lambda_{k}}(x^{k})\leq\mathbf{F_{s}}u_{\lambda_k}(x^k)-\mathbf{F_{s}}u(x^k)\\
&=-f\big((x^k)^{\lambda_k},u_{\lambda_k}(x^k),(\nabla_{x}u)((x^k)^{\lambda_k})\big)
+f\big(x^k,u(x^k),(\nabla_{x}u)(x^k)\big)\\
&\leq -f\big(x^k,u_{\lambda_k}(x^k),(\nabla_{x}u)(x^k)\big)
+f\big(x^k,u(x^k),(\nabla_{x}u)(x^k)\big)\\
&=:c(x^k)w_{\lambda_k}(x^k),
\end{split}\end{equation}
where
\begin{equation}\label{eqn-2}
  c(x^k):=-\frac{f\left(x^k,u_{\lambda_k}(x^k),(\nabla_xu)(x^k)\right)-f\left(x^k,u(x^k),(\nabla_xu)(x^k)\right)}{u_{\lambda_k}(x^k)-u(x^k)}
\end{equation}
is uniformly bounded independent of $k$, since $f(x, u, \mathbf{p})\in\mathcal{F}$ and $u \in C(\mathbb{R}^n)$ with compact support.

Note that $\Sigma_{\lambda_{k}}\cap\Omega$ is a narrow region for $k$ large enough. From \eqref{contradict}, \eqref{eqn-2} and the \emph{Narrow region principle} Theorem \ref{2NRP} (see Remark \ref{rem13} and \ref{rem14}), one can derive that, for $k$ sufficiently large,
\begin{equation}\label{eqn-3}
  w_{\lambda_{k}}>0 \qquad \text{in} \,\, \Sigma_{\lambda_{k}}\cap\Omega,
\end{equation}
which yields a contradiction with \eqref{eqn-0}. Hence there exists an $\epsilon>0$ small enough such that, \eqref{start} holds for any $-1<\lambda\leq-1+\epsilon$. Furthermore, suppose there exist $\hat{\lambda}\in(-1,-1+\epsilon]$ and $\hat{x}\in\Sigma_{\hat{\lambda}}\cap\Omega$ such that $w_{\hat{\lambda}}(\hat{x})=0$, then similar to \eqref{contradict}, we can deduce from the assumption \eqref{assump-f} in Theorem \ref{boundedthm} and \eqref{fractional-MA} that
\begin{equation}\begin{split}\label{eqn-4}
& \quad \mathbf{F_{s}}w_{\hat{\lambda}}(\hat{x})\leq\mathbf{F_{s}}u_{\hat{\lambda}}(\hat{x})-\mathbf{F_{s}}u(\hat{x})\\
&=-f\big((\hat{x})^{\hat{\lambda}},u_{\hat{\lambda}}(\hat{x}),(\nabla_{x}u)((\hat{x})^{\hat{\lambda}})\big)
+f\big(\hat{x},u(\hat{x}),(\nabla_{x}u)(\hat{x})\big)\\
&\leq -f\big(\hat{x},u_{\hat{\lambda}}(\hat{x}),(\nabla_{x}u)(\hat{x})\big)
+f\big(\hat{x},u(\hat{x}),(\nabla_{x}u)(\hat{x})\big)\\
&=0,
\end{split}\end{equation}
then it follows from the strong maximum principle Lemma \ref{SMP-anti} that $w_{\hat{\lambda}}=0$ a.e. in $\mathbb{R}^{n}$, which is absurd. Therefore, we have, for any $-1<\lambda\leq-1+\epsilon$,
\begin{equation}\label{aim}
w_\lambda(x)>0, \qquad \forall \,\, x\in\Sigma_\lambda\cap\Omega.
\end{equation}

\medskip

\emph{Step 2.} Move the plane continuously to the right until its limiting position. Step 1 provides a starting point for us to move planes. Next we will continue to move $T_{\lambda}$ to the right as long as \eqref{aim} holds.

To this end, let us define
\begin{equation}\label{221-def}
  \lambda_{0}:=\sup\left\{\lambda\in(-1,0] \mid w_{\mu}>0 \,\, \text{in} \,\, \Sigma_{\mu}\cap\Omega, \,\, \forall \, -1<\mu<\lambda\right\}.
\end{equation}
We aim to show that $\lambda_{0}=0$ via contradiction arguments.

Suppose on the contrary that $\lambda_{0}<0$, then we will be able to move $T_{\lambda}$ to the right a little bit further while \eqref{aim} still holds, which contradicts the definition \eqref{221-def} of $\lambda_{0}$.

Indeed, due to $\lambda_{0}<0$, one can infer from \eqref{fractional-MA} that $w_{\lambda_{0}}>0$ in $\left(\Omega^{\lambda_{0}}\setminus\Omega\right)\cap\Sigma_{\lambda_{0}}$ ($A^{\lambda}$ denotes the reflection of a set $A$ w.r.t. $T_{\lambda}$), and hence the strong maximum principle Lemma \ref{SMP-anti} yields that
\begin{equation}\label{221-2}
  w_{\lambda_{0}}(x)>0, \qquad \forall \,\, x\in\Sigma_{\lambda_{0}}\cap\Omega.
\end{equation}
Since $w_{\lambda_{0}}>0$ in $\Omega^{\lambda_{0}}\cap \Sigma_{\lambda_{0}}$, so there exists a compact subset $K\subset\subset\Omega^{\lambda_{0}}\cap \Sigma_{\lambda_{0}}$ and a constant $c>0$ such that
\begin{equation}\label{221-3}
  w_{\lambda_{0}}(x)\geq c>0, \qquad \forall \,\, x\in K\cap\Omega,
\end{equation}
and $(\Sigma_{\lambda_{0}}\cap\Omega)\setminus(K\cap\Omega)$ is a narrow region. Due to the continuity of $w_{\lambda}$ w.r.t. $\lambda$, we get, there exists a sufficiently small $0<\epsilon<\min\{-\lambda_{0},\lambda_{0}+1\}$ such that, for any $\lambda\in[\lambda_{0},\lambda_{0}+\epsilon]$,
\begin{equation}\label{221-4}
  w_{\lambda}(x)>0, \quad \forall \,\, x\in K\cap\Omega,
\end{equation}
and $(\Sigma_{\lambda_{0}+\epsilon}\cap\Omega)\setminus(K\cap\Omega)$ is also a narrow region.

For any $\lambda\in[\lambda_{0},\lambda_{0}+\epsilon]$, note that $(\Sigma_{\lambda}\cap\Omega)\setminus(K\cap\Omega)$ is a narrow region, we will deduce from the \emph{Narrow region principle} Theorem \ref{32NRP} that
\begin{equation}\label{eqn-5}
  w_{\lambda}(x)\geq0, \quad \forall \,\, x\in(\Sigma_{\lambda}\cap\Omega)\setminus(K\cap\Omega).
\end{equation}

Indeed, by \eqref{221-4}, if we suppose \eqref{eqn-5} does not hold, then there exists a $\tilde{\lambda}\in(\lambda_{0},\lambda_{0}+\epsilon]$ (depending on $\epsilon$) such that
\begin{equation}\label{eqn-6}
  \inf\limits_{(\Sigma_{\tilde{\lambda}}\cap\Omega)\setminus(K\cap\Omega)}w_{\tilde{\lambda}}=\inf\limits_{\Sigma_{\tilde{\lambda}}\cap\Omega}w_{\tilde{\lambda}}
  =\inf\limits_{\lambda\in(\lambda_{0},\tilde{\lambda}]}\inf\limits_{x\in\Sigma_\lambda}w_\lambda(x)<0.
\end{equation}
Consequently, there exists $\tilde{x}\in(\Sigma_{\tilde{\lambda}}\cap\Omega)\setminus(K\cap\Omega)$ such that
\begin{equation}\label{eqn-7}
  w_{\tilde{\lambda}}(\tilde{x})=\inf\limits_{(\Sigma_{\tilde{\lambda}}\cap\Omega)\setminus(K\cap\Omega)}w_{\tilde{\lambda}}
  =\inf\limits_{\Sigma_{\tilde{\lambda}}\cap\Omega}w_{\tilde{\lambda}}=\inf\limits_{\Sigma_{\tilde{\lambda}}}w_{\tilde{\lambda}}<0.
\end{equation}
Then, similar to \eqref{contradict} and \eqref{eqn-2}, by the assumption \eqref{assump-f} in Theorem \ref{boundedthm} and \eqref{fractional-MA}, we have
\begin{equation}\begin{split}\label{eqn-8}
& \quad \mathbf{F_{s}}w_{\tilde{\lambda}}(\tilde{x})\leq\mathbf{F_{s}}u_{\tilde{\lambda}}(\tilde{x})-\mathbf{F_{s}}u(\tilde{x})\\
&=-f\big((\tilde{x})^{\tilde{\lambda}},u_{\tilde{\lambda}}(\tilde{x}),(\nabla_{x}u)((\tilde{x})^{\tilde{\lambda}})\big)
+f\big(\tilde{x},u(\tilde{x}),(\nabla_{x}u)(\tilde{x})\big)\\
&\leq -f\big(\tilde{x},u_{\tilde{\lambda}}(\tilde{x}),(\nabla_{x}u)(\tilde{x})\big)
+f\big(\tilde{x},u(\tilde{x}),(\nabla_{x}u)(\tilde{x})\big)\\
&=:c(\tilde{x})w_{\tilde{\lambda}}(\tilde{x}),
\end{split}\end{equation}
where
\begin{equation}\label{eqn-9}
  c(\tilde{x}):=-\frac{f\left(\tilde{x},u_{\tilde{\lambda}}(\tilde{x}),(\nabla_{x}u)(\tilde{x})\right)-f\left(\tilde{x},u(\tilde{x}),(\nabla_{x}u)(\tilde{x})\right)}
  {u_{\tilde{\lambda}}(\tilde{x})-u(\tilde{x})}
\end{equation}
is uniformly bounded (independent of $\epsilon$ and $K$), since $f(x, u, \mathbf{p})\in\mathcal{F}$ and $u \in C(\mathbb{R}^n)$ with compact support. Note that, by choosing $K$ larger and $\epsilon$ smaller if necessary, $(\Sigma_{\tilde{\lambda}}\cap\Omega)\setminus(K\cap\Omega)$ is a narrow region. From \eqref{eqn-8}, \eqref{eqn-9} and the \emph{Narrow region principle} Theorem \ref{32NRP} (see Remark \ref{rem3}), one can derive that, for $\epsilon$ sufficiently small,
\begin{equation}\label{eqn-10}
  w_{\tilde{\lambda}}>0 \qquad \text{in} \,\, (\Sigma_{\tilde{\lambda}}\cap\Omega)\setminus(K\cap\Omega),
\end{equation}
which yields a contradiction with \eqref{eqn-6}. As a consequence, we have, for any $\lambda\in[\lambda_{0},\lambda_{0}+\epsilon]$, \eqref{eqn-5} holds. Furthermore, it follows from the strong maximum principle that
\begin{equation}\label{221-5}
  w_{\lambda}(x)>0, \qquad \forall \,\, x\in(\Sigma_{\lambda}\cap\Omega)\setminus(K\cap\Omega),
\end{equation}
and hence, for any $\lambda\in[\lambda_{0},\lambda_{0}+\epsilon]$,
\begin{equation}\label{221-6}
  w_{\lambda}(x)>0, \qquad \forall \,\, x\in\Sigma_{\lambda}\cap\Omega.
\end{equation}
This contradicts with the definition \eqref{221-def} of $\lambda_{0}$. Thus $\lambda_{0}=0$, or more precisely,
\begin{equation}\label{eq0}
  u(x_1, x')\leq u(-x_1, x'), \qquad \forall \,\, (x_1,x')\in \Omega, \,\, x_{1}<0.
\end{equation}

Furthermore, from the definition of $\lambda_{0}$, we can deduce that
\begin{equation*}
w_{\lambda}>0 \quad \text{in} \,\, \Sigma_\lambda\cap\Omega, \quad \forall \,\, \lambda<0.
\end{equation*}
For any $(x_1,x')$, $(\bar{x}_1,x')\in \Omega$ with $0>x_1>\bar{x}_1$, one can take $\lambda=\frac{x_{1}+\bar{x}_{1}}{2}$. Then we have
$$u(x_{1},x')>u(\bar{x}_{1},x'),$$
and hence $u(x_{1},x')$ is strictly increasing in the left half of $\Omega$ in $x_{1}$-direction.

Moreover, if $f(x_1, x', u, p_1, p_2,\cdots, p_n)= f(-x_{1}, x', u, -p_1, p_2,\cdots, p_n)$, then one can easily verify that $\hat{u}(x_{1},x'):=u(-x_{1},x')$ also solves \eqref{fractional-MA}. Thus we have derived that
\begin{equation}\label{eq2}
  \hat{u}(x_1,x')\leq \hat{u}(-x_1, x'), \qquad \forall \,\, (x_1,x')\in \Omega, \,\, x_{1}<0,
\end{equation}
or equivalently,
\begin{equation}\label{eq3}
  u(x_1, x')\geq u(-x_1, x'), \qquad \forall \,\, (x_1,x')\in \Omega, \,\, x_{1}<0.
\end{equation}
Combining this with \eqref{eq0} yields that
\begin{equation}\label{eq4}
  u(x_1, x')=u(-x_1, x'), \qquad \forall \,\, (x_1,x')\in \Omega, \,\, x_{1}<0,
\end{equation}
that is, $u$ is symmetric in the $x_1$ direction about $\{x\in\mathbb{R}^{n} \mid x_{1}=0\}$. This completes the proof of Theorem \ref{boundedthm}.
\end{proof}

\begin{rem}\label{rem5}
Typical forms of $f(x,u,\nabla u)$ which satisfies all the assumptions in Theorem \ref{boundedthm} include: $f(x,u,\nabla u)=u^{p}(1+|\nabla u|^{2})^{\frac{\sigma}{2}}$ with $p\geq1$ and $\sigma\leq0$, $f(x,u,\nabla u)=e^{\kappa u}(1+|\nabla u|^{2})^{\frac{\sigma}{2}}$ with $\kappa\in\mathbb{R}$ and $\sigma\leq0$, and $f(x,u,\nabla u)=K(x)(1+|\nabla u|^{2})^{\frac{\sigma}{2}}$ with $\sigma\in\mathbb{R}$, $K(x)=K(|x_{1}|,x')$ and $K(r,x')$ is nonincreasing w.r.t. $r\in[0,+\infty)$.
\end{rem}

As an immediate application of Theorem \ref{boundedthm}, we have the following corollary.

\begin{cor}\label{cor0}
Assume $u\in C_{loc}^{1,1}(B_{1}(0))\cap C(\overline{B_{1}(0)})$ solves
\begin{equation}\begin{cases}\label{nonlocal-MA-ball}
-\mathbf{F_{s}}u(x)=f(u(x)) \qquad \text{in}\ \ B_{1}(0),\\ \quad u>0 \qquad \text{in}\ \ B_{1}(0),\\
\quad u\equiv0 \qquad \text{in}\ \ \mathbb{R}^{n}\setminus B_{1}(0),
\end{cases}\end{equation}
where $0<s<1$, and $f(\cdot)$ is locally Lipschitz continuous. Then $u$ must be symmetric and strictly monotone decreasing (along the direction $\mathbf{e}$) w.r.t. the hyper-plane $\{x\in\mathbb{R}^{n}\mid x\cdot \mathbf{e}=0\}$.
\end{cor}

One can easily verify the validity of Corollary \ref{cor0}, since all the assumptions in Theorem \ref{boundedthm} are fulfilled and $f(u)$ satisfies the assumption \eqref{assump-f} in the direction $\mathbf{e}\in\mathbb{R}^{n}$.

\begin{rem}\label{rem4}
Typical forms of $f(u)$ satisfying the assumption in Corollary \ref{cor0} include: $f(u)=u^{p}$ with $p\geq1$ and $f(u)=e^{\kappa u}$ with $\kappa\in\mathbb{R}$.
\end{rem}

\subsection{Unbounded domain}

Suppose without loss of generalities that $\mathbf{e}=\mathbf{e_1}$. To state our monotonicity and symmetry results for unbounded domain $\Omega$, we need to assume the following condition on $f(x,u,\mathbf{p})$: for some $\beta>0$,
\begin{equation}\label{hypo-f}
\frac{|f(x, u_1, \mathbf{p})-f(x, u_2, \mathbf{p})|}{|u_1-u_2|}\leq C(|u_1|^{\beta}+|u_2|^{\beta}) \qquad \text{as}\ \ u_1, u_2\rightarrow0.
\end{equation}

Our monotonicity and symmetry results in unbounded domain $\Omega$ is the following theorem. For related results for $-\Delta$, $(-\Delta)^{s}$, $(-\Delta+m^{2})^{s}$, $(-\Delta)^{s}_{p}$ or second order Monge-Amp\`{e}re operator, please refer to \cite{BCN1,BCN2,BN1,BN2,BN3,CGS,Cheng1,CHL,CL2,CLM,DLW,DQW,GNN1,Li1,Li2}.
\begin{thm}\label{unboundedthm}
Let $\Omega$ be a unbounded domain in $\mathbb{R}^n$ which is convex in $x_1$-direction and symmetric about $\{x\in\mathbb{R}^{n} \mid x_1 = 0\}$. Suppose that $u\in C^{1,1}_{loc}(\Omega)\cap C(\overline{\Omega})\cap\mathcal{L}_{s}(\mathbb{R}^n)$ solves \eqref{fractional-MA} for $s\in(0,1)$ and $f(x,u,\mathbf{p})\in\mathcal{F}$. If $f(x,u,\mathbf{p})$ satisfies \eqref{assump-f}, \eqref{hypo-f} and $u(x)$ satisfies the following asymptotic properties:
\begin{equation}\label{asymp1}
\limsup\limits_{\substack{|x|\rightarrow\infty \\ x\in\Omega, \, x_1<0}}|x|^{2s}\left[u(x)\right]^{\beta}<\frac{C_{n,s,\theta}}{8C}.
\end{equation}
where $C_{n,s,\theta}$ is the same constant as in \eqref{decay_con} in Theorem \ref{P:decay} and $C$ is the constant in assumption \eqref{hypo-f}. Then we have\\
(i) If there exists a line $\mathcal{L}$ parallel to $x_{1}$-axis satisfying $\mathcal{L}\cap\Omega\neq\emptyset$ such that $\mathcal{L}\cap\Omega^{c}\neq\emptyset$, then $u(x_1, x')$ is strictly increasing in the left half of $\Omega$ in $x_1$-direction and
$$u(x_1, x')\leq u(-x_1, x'), \qquad \forall \,\, x_1<0, \,\, (x_1,x')\in \Omega.$$
(ii) If any line $\mathcal{L}$ parallel to $x_{1}$-axis such that $\mathcal{L}\cap\Omega\neq\emptyset$ must satisfy $\mathcal{L}\cap\Omega^{c}=\emptyset$, then there exists $\mu_{0}\leq0$ such that $u(x_1, x')$ is strictly increasing in $\Omega\cap\{x_{1}<\mu_{0}\}$ in $x_1$-direction and
$$\text{if} \,\, \mu_{0}=0, \qquad  u(x_1, x')\leq u(-x_1, x'), \quad \forall \,\, x_1<0, \,\, (x_1,x')\in \Omega,$$
$$\text{if} \,\, \mu_{0}<0, \qquad  u(x_1, x')=u(2\mu_{0}-x_1, x'), \quad \forall \,\, x_1<\mu_{0}, \,\, (x_1,x')\in \Omega.$$
\end{thm}
\begin{proof}
We will use the same notations as in the proof of Theorem \ref{boundedthm}. The proof of Theorem \ref{unboundedthm} will be carried out by two steps.

\medskip

\emph{Step 1.} We first show that there exists $R_0>0$ large enough such that
\begin{equation}\label{unbddstep2}
w_\lambda\geq 0 \quad \text{in} \,\, \Sigma_\lambda, \qquad \forall \,\, \lambda\leq -R_0.
\end{equation}

Indeed, since $u$ satisfies \eqref{fractional-MA}, we infer from the asymptotic property \eqref{asymp1} that, for any $\lambda\leq0$, the negative minimum of $w_{\lambda}$ can be attained in $\Sigma_{\lambda}\cap\Omega$. Suppose on the contrary that \eqref{unbddstep2} is not true, then there exists a sequence $\lambda_k\rightarrow-\infty$ as $k\rightarrow+\infty$ such that
\begin{equation}\label{eqn-11}
 \inf\limits_{x\in\Sigma_{\lambda_{k}}\cap\Omega}w_{\lambda_{k}}(x)=\inf\limits_{\lambda \leq\lambda_k}\inf\limits_{x\in\Sigma_\lambda}w_\lambda(x)<0
\end{equation}
for every $k=1,2,\cdots$. Moreover, for every $k=1,2,\cdots$, $\inf\limits_{\Sigma_{\lambda_{k}}\cap\Omega}w_{\lambda_{k}}$ can be attained at some $x^{k}\in\Sigma_{\lambda_{k}}\cap\Omega$, that is,
\begin{equation}\label{eqn-12}
w_{\lambda_{k}}(x^k)=\inf\limits_{\Sigma_{\lambda_{k}}\cap\Omega}w_{\lambda_{k}}=\inf\limits_{\Sigma_{\lambda_{k}}}w_{\lambda_{k}}<0.
\end{equation}

Then, similar to \eqref{contradict} and \eqref{eqn-2} in Step 1 in the proof of Theorem \ref{boundedthm}, by the assumptions \eqref{assump-f} in Theorem \ref{unboundedthm} and \eqref{fractional-MA}, we have
\begin{equation}\begin{split}\label{eqn-13}
& \quad \mathbf{F_{s}}w_{\lambda_{k}}(x^{k})\leq\mathbf{F_{s}}u_{\lambda_k}(x^k)-\mathbf{F_{s}}u(x^k)\\
&=-f\big((x^k)^{\lambda_k},u_{\lambda_k}(x^k),(\nabla_{x}u)((x^k)^{\lambda_k})\big)
+f\big(x^k,u(x^k),(\nabla_{x}u)(x^k)\big)\\
&\leq -f\big(x^k,u_{\lambda_k}(x^k),(\nabla_{x}u)(x^k)\big)
+f\big(x^k,u(x^k),(\nabla_{x}u)(x^k)\big)\\
&=:c_{k}(x^k)w_{\lambda_k}(x^k),
\end{split}\end{equation}
where
\begin{equation}\label{eqn-14}
  c_{k}(x):=\begin{cases}-\frac{f\left(x,u_{\lambda_k}(x),(\nabla_{x}u)(x)\right)-f\left(x,u(x),(\nabla_{x}u)(x)\right)}{u_{\lambda_{k}}(x)-u(x)},\quad \text{if}\,\,u_{\lambda_{k}}(x)\neq u(x),\\ \\ 0,\qquad \text{if}\,\,u_{\lambda_{k}}(x)=u(x).
  \end{cases}
\end{equation}

By the assumption \eqref{hypo-f} on $f$ and the asymptotic property \eqref{asymp1}, we have, for $k$ large enough, at any points $x\in\Sigma_{\lambda_{k}}\cap\Omega$ where $w_{\lambda_{k}}(x)<0$,
\begin{equation}\label{eq6}
|c_{k}(x)|\leq C\big(|u_{\lambda_k}(x)|^{\beta}+|u(x)|^{\beta}\big)\leq 2C\left[u(x)\right]^{\beta},
\end{equation}
and hence
\begin{equation}\label{eqn-17}
\liminf\limits_{\substack{x\in\Sigma_{\lambda_{k}}\cap\Omega, w_{\lambda_{k}}<0 \\ |x|\rightarrow+\infty}}|x|^{2s}c_{k}(x)\geq
-2C\limsup\limits_{\substack{x\in\Sigma_{\lambda_{k}}\cap\Omega, w_{\lambda_{k}}<0 \\ |x|\rightarrow+\infty}}|x|^{2s}\left[u(x)\right]^{\beta}>-\frac{C_{n,s,\theta}}{4}.
\end{equation}
By the \emph{Decay at infinity (I)} Theorem \ref{P:decay}, we have, there exists a $R_{0}>0$ such that
\begin{equation}\label{eqn-18}
  |x^{k}|\leq R_{0},
\end{equation}
which yields a contradiction with $|x^{k}|>-\lambda_{k}\rightarrow+\infty$ as $k\rightarrow+\infty$. This establishes \eqref{unbddstep2}.

Since $u(x)\rightarrow0$ as $|x|\rightarrow+\infty$ and $x\in\Sigma_{0}$, by choosing $R_{0}$ larger if necessary, we can actually deduce that $w_{\lambda}\not\equiv0$ in $\Sigma_{\lambda}$ for any $\lambda\leq-R_{0}$. Then, similar to \eqref{aim} in Step 1 in the proof of Theorem \ref{boundedthm}, it follows from the strong maximum principle Lemma \ref{SMP-anti} that
\begin{equation}\label{unbddstep2-p}
w_\lambda>0 \quad \text{in} \,\, \Sigma_\lambda\cap\Omega, \quad \forall \,\, \lambda\leq -R_0.
\end{equation}

\medskip

\emph{Step 2.} Let
$$\lambda_0:=\sup\{\lambda\in(-\infty,0] \mid w_\mu>0 \,\,\, \text{in} \,\, \Omega\cap\Sigma_\mu, \,\,\, \forall \,\, \mu\leq\lambda\}\in[-R_{0},0].$$

By the definition of $\lambda_0$ and the continuity of $u(x)$, we have $w_{\lambda_0}(x)\geq0$ for all $x\in \Sigma_{\lambda_0}$.

Next, we will carry out our proof by discussing two different cases.

\emph{Case (i).} There exists a line $\mathcal{L}$ parallel to $x_{1}$-axis satisfying $\mathcal{L}\cap\Omega\neq\emptyset$ such that $\mathcal{L}\cap\Omega^{c}\neq\emptyset$. In such case, we will show that
$$\lambda_0=0.$$

Now suppose on the contrary that $\lambda_{0}<0$. Note that $w_{\lambda_{0}}\not\equiv0$ in $\Sigma_{\lambda_{0}}$, then from the strong maximum principle Lemma \ref{SMP-anti}, we can derive that
\begin{equation}\label{positive-1}
  w_{\lambda_0}>0 \quad \text{in} \,\, \Omega\cap\Sigma_{\lambda_0}.
\end{equation}

Next, we will show that, there exists $\varepsilon>0$ small enough such that
\begin{equation}\label{ubddstep2}
w_\lambda\geq0 \quad \text{in} \,\, \Sigma_\lambda, \quad \forall \,\, \lambda_{0}<\lambda\leq\lambda_0+\varepsilon.
\end{equation}

Suppose \eqref{ubddstep2} is not true, then there exists a sequence $\{\lambda_{k}\}\subset(\lambda_{0},0)$ satisfying $\lambda_k\rightarrow\lambda_{0}$ as $k\rightarrow+\infty$ such that
\begin{equation}\label{eqn-11'}
 \inf\limits_{x\in\Sigma_{\lambda_{k}}\cap\Omega}w_{\lambda_{k}}(x)=\inf\limits_{\lambda \leq\lambda_k}\inf\limits_{x\in\Sigma_\lambda}w_\lambda(x)<0
\end{equation}
for every $k=1,2,\cdots$. Moreover, for every $k=1,2,\cdots$, $\inf\limits_{\Sigma_{\lambda_{k}}\cap\Omega}w_{\lambda_{k}}$ can be attained at some $x^{k}\in\Sigma_{\lambda_{k}}\cap\Omega$, that is,
\begin{equation}\label{eqn-12'}
w_{\lambda_{k}}(x^k)=\inf\limits_{\Sigma_{\lambda_{k}}\cap\Omega}w_{\lambda_{k}}=\inf\limits_{\Sigma_{\lambda_{k}}}w_{\lambda_{k}}<0.
\end{equation}

Then, similar to \eqref{contradict} and \eqref{eqn-2} in Step 1 in the proof of Theorem \ref{boundedthm}, by the assumptions \eqref{assump-f} in Theorem \ref{unboundedthm} and \eqref{fractional-MA}, we have
\begin{equation}\begin{split}\label{eqn-13'}
& \quad \mathbf{F_{s}}w_{\lambda_{k}}(x^{k})\leq\mathbf{F_{s}}u_{\lambda_k}(x^k)-\mathbf{F_{s}}u(x^k)\\
&=-f\big((x^k)^{\lambda_k},u_{\lambda_k}(x^k),(\nabla_{x}u)((x^k)^{\lambda_k})\big)
+f\big(x^k,u(x^k),(\nabla_{x}u)(x^k)\big)\\
&\leq -f\big(x^k,u_{\lambda_k}(x^k),(\nabla_{x}u)(x^k)\big)
+f\big(x^k,u(x^k),(\nabla_{x}u)(x^k)\big)\\
&=:c_{k}(x^k)w_{\lambda_k}(x^k),
\end{split}\end{equation}
where
\begin{equation}\label{eqn-14'}
c_{k}(x^{k}):=-\frac{f\left(x^{k},u_{\lambda_k}(x^{k}),(\nabla_{x}u)(x^{k})\right)-f\left(x^{k},u(x^{k}),(\nabla_{x}u)(x^{k})\right)}{u_{\lambda_{k}}(x^{k})-u(x^{k})}.
\end{equation}

Suppose that $\{x^k\}$ is not bounded, then up to a subsequence (still denote by $\{x^k\}$), $|x^{k}|\rightarrow+\infty$ as $k\rightarrow+\infty$. By the assumption \eqref{hypo-f} on $f$ and the asymptotic property \eqref{asymp1}, we have, for $k$ large enough,
\begin{equation}\label{eq6'}
|c_{k}(x^{k})|\leq C\big(|u_{\lambda_k}(x^{k})|^{\beta}+|u(x^{k})|^{\beta}\big)\leq 2C\left[u(x^{k})\right]^{\beta}.
\end{equation}
From \eqref{Decay-3} in \emph{Decay at infinity (I)} Theorem \ref{P:decay}, we infer that
\begin{align}\label{Decay-3'}
	0\geq \mathbf{F_{s}}w_{\lambda_{k}}(x^{k})-c_{k}(x^{k})w_{\lambda_{k}}(x^{k})
\geq\left[-\frac{C_{n,s,\theta}}{|x^{k}|^{2s}}-c_{k}(x^{k})\right]w_{\lambda_{k}}(x^{k}),
\end{align}
and hence
\begin{equation}\label{Decay-4'}
	-2C|x^{k}|^{2s}\left[u(x^{k})\right]^{\beta}\leq|x^{k}|^{2s}c_{k}(x^{k})\leq-C_{n,s,\theta}<0.
\end{equation}
This leads to a contradiction to the asymptotic property \eqref{asymp1}. Thus there exists $R_{\ast}>0$ such that $|x^k|<R_{\ast}$.

Indeed, due to $\lambda_{0}<0$, one can infer from \eqref{fractional-MA} that $w_{\lambda_{0}}>0$ in $\left(\Omega^{\lambda_{0}}\setminus\Omega\right)\cap\Sigma_{\lambda_{0}}$ ($A^{\lambda}$ denotes the reflection of a set $A$ w.r.t. $T_{\lambda}$), and hence \eqref{positive-1} yields that $w_{\lambda_{0}}>0$ in $\Omega^{\lambda_{0}}\cap \Sigma_{\lambda_{0}}$. So there exists a compact subset $K\subset\subset\Omega^{\lambda_{0}}\cap \Sigma_{\lambda_{0}}$ and a constant $c>0$ such that
\begin{equation}\label{221-3-ub}
  w_{\lambda_{0}}(x)\geq c>0, \qquad \forall \,\, x\in K\cap\Omega\cap B_{R_{\ast}}(0),
\end{equation}
and $(\Sigma_{\lambda_{0}}\cap\Omega\cap B_{R_{\ast}}(0))\setminus(K\cap\Omega\cap B_{R_{\ast}}(0))$ is a narrow region. Due to the continuity of $w_{\lambda}$ w.r.t. $\lambda$, we get, there exists a sufficiently small $0<\varepsilon<\min\{-\lambda_{0},\lambda_{0}+1\}$ such that, for any $\lambda\in[\lambda_{0},\lambda_{0}+\varepsilon]$,
\begin{equation}\label{221-4-ub}
  w_{\lambda}(x)>0, \quad \forall \,\, x\in K\cap\Omega\cap B_{R_{\ast}}(0),
\end{equation}
and $(\Sigma_{\lambda_{0}+\varepsilon}\cap\Omega\cap B_{R_{\ast}}(0))\setminus(K\cap\Omega\cap B_{R_{\ast}}(0))$ is also a narrow region. By \eqref{221-4-ub}, we deduce that, for $k$ large enough, $x^{k}\in(\Sigma_{\lambda_{0}+\varepsilon}\cap\Omega\cap B_{R_{\ast}}(0))\setminus(K\cap\Omega\cap B_{R_{\ast}}(0))$.

Since $f(x,u,\mathbf{p})\in\mathcal{F}$, $u$ solves \eqref{fractional-MA} and satisfies the asymptotic property \eqref{asymp1}, we have
\begin{equation}\label{eqn-15}
  c_{k}(x^k):=-\frac{f\big(x^k,u_{\lambda_k}(x^k),(\nabla_xu)(x^k)\big)-f\left(x^k,u(x^k),(\nabla_xu)(x^k)\right)}{u_{\lambda_k}(x^k)-u(x^k)}
\end{equation}
is uniformly bounded (independent of $k$ and $K$). Note that, by choosing $K$ and $k$ larger if necessary, $(\Sigma_{\lambda_{k}}\cap\Omega\cap B_{R_{\ast}}(0))\setminus(K\cap\Omega\cap B_{R_{\ast}}(0))$ is a narrow region. From \eqref{eqn-13'}, \eqref{eqn-15} and the \emph{Narrow region principle} Theorem \ref{32NRP} (see Remark \ref{rem3}), one can derive that, for $k$ sufficiently large,
\begin{equation}\label{eqn-16}
  w_{\lambda_{k}}>0 \qquad \text{in} \,\, (\Sigma_{\lambda_{k}}\cap\Omega\cap B_{R_{\ast}}(0))\setminus(K\cap\Omega\cap B_{R_{\ast}}(0)),
\end{equation}
which yields a contradiction with \eqref{eqn-12'}. Thus we have derived \eqref{ubddstep2}.

By the strong maximum principle Lemma \ref{SMP-anti}, we have either $w_{\lambda}>0$ or $w_{\lambda}\equiv0$ in $\Omega\cap\Sigma_{\lambda}$. Furthermore, since $w_{\lambda_{0}}>0$ in $\Omega\cap\Sigma_{\lambda_{0}}$, by continuity, choosing $\varepsilon>0$ smaller if necessary, we actually have
\begin{equation}\label{ubddstep2-1}
w_\lambda>0 \quad \text{in} \,\, \Omega\cap\Sigma_\lambda, \quad \forall \,\, \lambda_{0}\leq\lambda\leq\lambda_0+\varepsilon.
\end{equation}
This contradicts the definition of $\lambda_{0}$. Thus $\lambda_{0}=0$ and hence
\begin{equation}\label{eq0-1}
  u(x_1, x')\leq u(-x_1, x'), \quad \forall \,\, (x_1,x')\in \Omega, \,\, x_{1}<0.
\end{equation}
The strict monotonicity follows from $w_\lambda>0$ in $\Sigma_\lambda\cap\Omega$ for any $\lambda<\lambda_0$.

\emph{Case (ii).} Any line $\mathcal{L}$ parallel to $x_{1}$-axis such that $\mathcal{L}\cap\Omega\neq\emptyset$ must satisfy $\mathcal{L}\cap\Omega^{c}=\emptyset$. We will show that either $\lambda_{0}=0$ or $\lambda_{0}<0$ and $w_{\lambda_{0}}\equiv0$ in $\Sigma_{\lambda_{0}}$.

Assume that $\lambda_{0}<0$ but $w_{\lambda_{0}}\not\equiv0$ in $\Sigma_{\lambda_{0}}$. Then, similar to \eqref{positive-1}, we can derive that
\begin{equation}\label{positive-2}
  w_{\lambda_0}>0 \quad \text{in} \,\, \Omega\cap\Sigma_{\lambda_0}.
\end{equation}
Next, similar to \eqref{ubddstep2-1}, we can show that, there exists $\varepsilon>0$ small enough such that
\begin{equation}\label{ubddstep2-2}
w_\lambda>0 \quad \text{in} \,\, \Omega\cap\Sigma_\lambda, \quad \forall \,\, \lambda_{0}\leq\lambda\leq\lambda_0+\varepsilon.
\end{equation}
This contradicts the definition of $\lambda_{0}$. Thus we must have $\lambda_{0}=0$. This concludes the proof of Theorem \ref{unboundedthm}.
\end{proof}

\begin{rem}\label{rem6}
Typical forms of $f(x,u,\nabla u)$ which satisfies all the assumptions in Theorem \ref{unboundedthm} include: $f(x,u,\nabla u)=u^{p}(1+|\nabla u|^{2})^{\frac{\sigma}{2}}$ with $p>1$ and $\sigma\leq0$, and $f(x,u,\nabla u)=K(x)(1+|\nabla u|^{2})^{\frac{\sigma}{2}}$ with $\sigma\in\mathbb{R}$, $K(x)=K(|x_{1}|,x')$ and $K(r,x')$ is nonincreasing w.r.t. $r\in[0,+\infty)$.
\end{rem}

As immediate consequences of Theorem \ref{unboundedthm}, we have the following two corollaries below.

\begin{cor}\label{cor1}
Suppose $u\in C_{loc}^{1,1}(\mathbb{R}^{n})$ is a nonnegative solution to
\begin{equation}\label{eq13}
  -\mathbf{F_{s}}u(x)=f(u(x)) \qquad \text{in} \,\, \mathbb{R}^{n}
\end{equation}
with $0<s<1$, where $f(\cdot)$ is locally Lipschitz continuous and satisfies
\begin{equation}\label{ftiaojian2-1}
\frac{|f(u_1)-f(u_2)|}{|u_1-u_2|}\leq C(|u_1|^{\beta}+|u_2|^{\beta}) \quad \text{as}\ \ u_1, u_2\rightarrow0, \quad \text{for some} \,\, \beta>0.
\end{equation}
Moreover, assume that
\begin{equation}\label{condition}
  \limsup\limits_{|x|\rightarrow\infty}|x|^{2s}\left[u(x)\right]^{\beta}<\frac{C_{n,s,\theta}}{8C},
\end{equation}
where the constants $C_{n,s,\theta}$ and $C$ are the same as in Theorem \ref{unboundedthm}. Then $u(x)$ is symmetric and monotone decreasing (along the direction $\mathbf{e}$) w.r.t. some hyper-plane $T$ perpendicular to $\mathbf{e}$.
\end{cor}

One can easily verify the validity of Corollary \ref{cor1}, since all the assumptions in Theorem \ref{unboundedthm} are fulfilled and $f(u)$ satisfies \eqref{assump-f} in the direction $\mathbf{e}\in\mathbb{R}^{n}$.

\begin{rem}\label{rem7}
A typical type of nonlinearity which satisfies all the assumptions in Corollary \ref{cor1} is $f(u)=u^{p}$ with $p>1$.
\end{rem}

Another typical example is the so-called infinite cylinder $\mathcal{C}:=(-\infty,+\infty)\times D'$, where $D'\subset\mathbb{R}^{n-1}$ is a bounded domain.
\begin{cor}\label{cor2}
Suppose $u\in C^{1,1}_{loc}(\mathcal{C})\cap \mathcal{L}_{s}(\mathbb{R}^{n})\cap C(\overline{\mathcal{C}})$ solves \eqref{fractional-MA} for $\Omega=\mathcal{C}$, where $s\in(0,1)$ and $f(x, u, \mathbf{p})\in\mathcal{F}$ satisfies \eqref{hypo-f}. If $f(x, u, \mathbf{p})$ satisfies
\begin{equation}\label{assumption-f}
f(x_1, x', u, p_1, p_2,\cdots, p_n)\leq f(\bar{x}_1, x', u, -p_1, p_2,\cdots, p_n), \quad\, \forall \,\,\bar{x}_1\geq x_1, \, p_1\geq0,
\end{equation}
and $u(x)$ satisfies the asymptotic property
\begin{equation}\label{asymp-p}
  \limsup\limits_{\substack{x_{1}\rightarrow-\infty, \\ x\in\mathcal{C}}}|x|^{2s}\left[u(x)\right]^{\beta}<\frac{C_{n,s,\theta}}{8C},
\end{equation}
where the constants $C_{n,s,\theta}$ and $C$ are the same as in Theorem \ref{unboundedthm}. Then either there exists $\mu_{0}\in\mathbb{R}$ such that $u(x)$ is monotone increasing in $\mathcal{C}\cap\left\{x_{1}<\mu_{0}\right\}$ in $x_{1}$-direction and $u(x_{1},x')=u(2\mu_{0}-x_{1},x')$, or $u(x)$ is monotone increasing in $\mathcal{C}$ in $x_{1}$-direction.
\end{cor}

\begin{rem}\label{rem10}
Typical forms of $f(x,u,\nabla u)$ which satisfies all the assumptions in Corollary \ref{cor2} include: $f(x,u,\nabla u)=u^{p}(1+|\nabla u|^{2})^{\frac{\sigma}{2}}$ with $p>1$ and $\sigma\leq0$, and $f(x,u,\nabla u)=K(x)(1+|\nabla u|^{2})^{\frac{\sigma}{2}}$ with $\sigma\in\mathbb{R}$ and $K(x)$ nondecreasing w.r.t. $x_{1}$.
\end{rem}

\subsection{Coercive epigraph $\Omega$}
Assume without loss of generalities that $\mathbf{e}=\mathbf{e_n}:=(0,\cdots,0,1)$. A domain $\Omega \subseteq \mathbb{R}^n$ is a coercive epigraph if there exists a continuous function $\varphi : \mathbb{R}^{n-1}\rightarrow \mathbb{R}$ satisfying
\begin{equation}\label{M-1}
\lim\limits_{|x'|\rightarrow +\infty}\varphi(x')=+\infty,
\end{equation}
such that $\Omega=\{x=(x',x_{n})\in\mathbb{R}^{n} \mid x_n>\varphi(x')\}$, where $x':=(x_{1},\cdots,x_{n-1})\in\mathbb{R}^{n-1}$.

In this setting, we can prove the following monotonicity result via the method of moving planes for $\mathbf{F_{s}}$.
\begin{thm}\label{Mono-eg}
Let $\Omega$ be a coercive epigraph, and let $u\in\mathcal{L}_{s}(\mathbb{R}^{n})\cap C^{1,1}_{loc}(\Omega)\cap C(\overline{\Omega})$ be a solution to \eqref{fractional-MA} with $s\in(0,1)$. Suppose that $f(x, u, \mathbf{p})\in\mathcal{F}$ and satisfies
\begin{equation}\begin{cases}\label{assump-f-ce}
f(x', x_{n}, u, p_1, p_2,\cdots, p_n)\leq f(x', \bar{x}_{n}, u, p_1, p_2,\cdots, -p_n),\\
\ \ \ \ \ \forall \,\,x_n\geq\min\limits_{\mathbb{R}^{n-1}}\varphi, \,\, p_n\geq0, \,\, \bar{x}_{n}\geq x_{n}.
\end{cases}\end{equation}
Then $u$ is strictly monotone increasing in $x_{n}$.
\end{thm}
\begin{proof}
	Without loss of generality, we assume $$\inf\limits_{x\in\Omega}x_n=\min\limits_{\mathbb{R}^{n-1}}\varphi=0.$$
	For arbitrary $\lambda>0$, let
	$$
	T_{\lambda}:=\left\{x \in \mathbb{R}^{n} | x_{n}=\lambda\right\}
	$$
	be the moving planes,
	\begin{equation}
	\Sigma_{\lambda}:=\left\{x \in \mathbb{R}^{n} | x_{n}<\lambda\right\}
	\end{equation}
	be the region below the plane, and
	$$
	x^{\lambda}:=\left(x_{1}, x_{2}, \ldots, 2 \lambda-x_{n}\right)
	$$
	be the reflection of $x$ about the plane $T_{\lambda}$.
	
	Assume that $u$ is a solution to problem \eqref{fractional-MA}. To compare the values of $u(x)$ with $u_{\lambda}(x):=u\left(x^{\lambda}\right)$, we denote
	$$
	w_{\lambda}(x):=u_{\lambda}(x)-u(x).
	$$
	Since $\Omega$ is a coercive epigraph, $\Sigma_{\lambda}\cap \Omega$ is always bounded for every $\lambda>0$. One can easily obtain that, for any $\lambda>0$,
\begin{equation}\label{M-5}
	w_{\lambda}(x)\geq 0, \quad w_{\lambda}(x)\not\equiv0 \qquad \text{in}\,\, \Sigma_{\lambda}\setminus \Omega.
\end{equation}
We aim at proving that $w_{\lambda}>0$ in $\Sigma_{\lambda}\cap\Omega$ for every $\lambda>0$, which gives the desired strict monotonicity.

\medskip
	
We will carry out the method of moving planes in two steps.

\medskip
	
\emph{Step 1.} We will first show that, for $\lambda>0$ sufficiently close to $0$,
\begin{eqnarray}\label{M-6}
	w_{\lambda}\geq0 \qquad \text{in} \,\, \Sigma_{\lambda}\cap\Omega.
\end{eqnarray}

Suppose \eqref{M-6} does not hold, then there exists a sequence $\{\lambda_{k}\}$ satisfying $\lambda_{k}>0$ and $\lambda_k\rightarrow0$ as $k\rightarrow+\infty$ such that
\begin{equation}\label{eqn-0ce}
  \inf\limits_{\Sigma_{\lambda_{k}}\cap\Omega}w_{\lambda_{k}}=\inf\limits_{\lambda\in(0,\lambda_{k}]}\inf\limits_{x\in\Sigma_\lambda}w_\lambda(x)<0.
\end{equation}
Consequently, there exists $x^k\in\Sigma_{\lambda_{k}}\cap\Omega$ such that
\begin{equation}\label{eqn-1ce}
  w_{\lambda_{k}}(x^{k})=\inf\limits_{\Sigma_{\lambda_{k}}\cap\Omega}w_{\lambda_{k}}=\inf\limits_{\Sigma_{\lambda_{k}}}w_{\lambda_{k}}<0.
\end{equation}
It follows directly from \eqref{eqn-0ce} and \eqref{eqn-1ce} that $\frac{\partial w_{\lambda}}{\partial\lambda}|_{\lambda=\lambda_k}(x^k)\leq0$, and hence $(\partial_{x_{n}}u)[(x^k)^{\lambda_k}]\leq0$. Note that $x^k$ is the interior minimum of $w_{\lambda_k}(x)$, then one has $\nabla_{x}w_{\lambda_k}(x^k)=0$, i.e.,
\begin{equation}\label{nabla-ce}
(\nabla_{x}u_{\lambda_k})(x^k)=(\nabla_{x}u)(x^k).
\end{equation}
By the assumption \eqref{assump-f-ce} in Theorem \ref{Mono-eg} and \eqref{fractional-MA}, we have
\begin{equation}\begin{split}\label{contradict-ce}
& \quad \mathbf{F_{s}}w_{\lambda_{k}}(x^{k})\leq\mathbf{F_{s}}u_{\lambda_k}(x^k)-\mathbf{F_{s}}u(x^k)\\
&=-f\big((x^k)^{\lambda_k},u_{\lambda_k}(x^k),(\nabla_{x}u)((x^k)^{\lambda_k})\big)
+f\big(x^k,u(x^k),(\nabla_{x}u)(x^k)\big)\\
&\leq -f\big(x^k,u_{\lambda_k}(x^k),(\nabla_{x}u)(x^k)\big)
+f\big(x^k,u(x^k),(\nabla_{x}u)(x^k)\big)\\
&=:c(x^k)w_{\lambda_k}(x^k),
\end{split}\end{equation}
where
\begin{equation}\label{eqn-2ce}
  c(x^k):=-\frac{f\left(x^k,u_{\lambda_k}(x^k),(\nabla_xu)(x^k)\right)-f\left(x^k,u(x^k),(\nabla_xu)(x^k)\right)}{u_{\lambda_k}(x^k)-u(x^k)}
\end{equation}
is uniformly bounded independent of $k$, since $f(x, u, \mathbf{p})\in\mathcal{F}$ and $u\in L^{\infty}_{loc}(\mathbb{R}^{n})$.

Note that $\Sigma_{\lambda_{k}}\cap\Omega$ is a narrow region for $k$ large enough. From \eqref{contradict-ce}, \eqref{eqn-2ce} and the \emph{Narrow region principle} Theorem \ref{2NRP} (see Remark \ref{rem13} and \ref{rem14}), one can derive that, for $k$ sufficiently large,
\begin{equation}\label{eqn-3ce}
  w_{\lambda_{k}}>0 \qquad \text{in} \,\, \Sigma_{\lambda_{k}}\cap\Omega,
\end{equation}
which yields a contradiction with \eqref{eqn-0ce}. Hence there exists an $\epsilon>0$ small enough such that, \eqref{M-6} holds for any $0<\lambda\leq\epsilon$. Furthermore, it follows from \eqref{M-5} and the strong maximum principle Lemma \ref{SMP-anti} that, for any $0<\lambda\leq\epsilon$,
\begin{equation}\label{aim-ce}
w_\lambda(x)>0, \qquad \forall \,\, x\in\Sigma_\lambda\cap\Omega.
\end{equation}

\medskip
	
\emph{Step 2.} Inequality \eqref{aim-ce} provides a starting point for us to carry out the moving planes procedure. Now we increase $\lambda$ from close to $0$ to $+\infty$ as long as inequality \eqref{aim-ce} holds until its limiting position. Define
	\begin{equation}\label{M-7}
	\lambda_{0}:=\sup \left\{\lambda>0 \mid w_{\mu}> 0 \,\, \text{in} \,\, x \in\Sigma_{\mu}\cap\Omega, \,\, \forall \, 0<\mu<\lambda\right\}.
	\end{equation}
	We aim to prove that
	$$
	\lambda_{0}=+\infty.
	$$
	
Otherwise, suppose on the contrary that $0<\lambda_{0}<+\infty$, we will show that the plane $T_{\lambda_0}$ can be moved upward a little bit more, that is, there exists an $\varepsilon>0$ small enough such that
	\begin{eqnarray}\label{M-8}
	w_{\lambda}>0  \quad \text{in} \,\,\Sigma_{\lambda}\cap\Omega, \qquad \forall \,\, \lambda_{0}\leq\lambda\leq\lambda_{0}+\varepsilon,
	\end{eqnarray}
which contradicts the definition \eqref{M-7} of $\lambda_{0}$.
	
First, by the definition of $\lambda_{0}$, we have $w_{\lambda_{0}}\geq 0$ in $\Sigma_{\lambda_0}\cap\Omega$. Since $u>0$ in $\Omega$ and $u\equiv 0$ in $\mathbb{R}^n\setminus\Omega$, we have $w_{\lambda_{0}}(x)>0$ for any $x\in\Omega^{\lambda_{0}}\setminus\Omega$, where the notation $A^{\lambda}$ denotes the reflection of a given set $A$ w.r.t. the plane $T_{\lambda}$. Then, we can obtain from the strong maximum principle Lemma \ref{SMP-anti} that
\begin{equation}\label{221-2-ce}
  w_{\lambda_{0}}(x)>0, \qquad \forall \,\, x\in\Sigma_{\lambda_{0}}\cap\Omega.
\end{equation}	

Next, we choose $\varepsilon_1>0$ sufficiently small such that $\left(\Sigma_{\lambda_{0}+\varepsilon_1}\setminus\overline{\Sigma_{\lambda_{0}-\varepsilon_1}}\right)\cap \Omega$ is a bounded narrow region. By the fact that $w_{\lambda_{0}}>0$ in $\Omega^{\lambda_{0}}\cap\Sigma_{\lambda_{0}}$ and the continuity of $w_{\lambda_{0}}$, there exists $c_0>0$ such that
	$$w_{\lambda_{0}}(x)> c_0,\qquad \forall \,\, x\in\overline{\Sigma_{\lambda_0-\varepsilon_1}}\cap \Omega.$$
	Therefore, we can choose $0<\varepsilon_2<\varepsilon_1$ sufficiently small such that
	\begin{equation}\label{M-9}
	w_{\lambda}(x)>\frac{c_0}{2}>0,\qquad \forall \,\, x\in\overline{\Sigma_{\lambda_0-\varepsilon_1}}\cap \Omega,
	\end{equation}
for every $\lambda_0\leq\lambda\leq\lambda_0+\varepsilon_2$. For any $\lambda\in[\lambda_{0},\lambda_{0}+\varepsilon_{2}]$, since $\left(\Sigma_{\lambda}\setminus\overline{\Sigma_{\lambda_{0}-\varepsilon_1}}\right)\cap \Omega$ is a bounded narrow region, we will deduce from the \emph{Narrow region principle} Theorem \ref{2NRP} that
	\begin{equation}\label{eqn-5-ce}
		w_{\lambda}\geq0  \quad \text{in} \,\,\left(\Sigma_{\lambda}\setminus\overline{\Sigma_{\lambda_{0}-\varepsilon_1}}\right)\cap \Omega.
	\end{equation}

Indeed, by \eqref{M-9}, if we suppose \eqref{eqn-5-ce} does not hold, then there exists a $\hat{\lambda}\in(\lambda_{0},\lambda_{0}+\varepsilon_{2}]$ (depending on $\varepsilon_{2}$) such that
\begin{equation}\label{eqn-6-ce}
  \inf\limits_{\left(\Sigma_{\hat{\lambda}}\setminus\overline{\Sigma_{\lambda_{0}-\varepsilon_1}}\right)\cap\Omega}w_{\hat{\lambda}}
  =\inf\limits_{\Sigma_{\hat{\lambda}}\cap\Omega}w_{\hat{\lambda}}
  =\inf\limits_{\lambda\in(\lambda_{0},\hat{\lambda}]}\inf\limits_{x\in\Sigma_{\lambda}}w_\lambda(x)<0.
\end{equation}
Consequently, there exists $\bar{x}\in\left(\Sigma_{\hat{\lambda}}\setminus\overline{\Sigma_{\lambda_{0}-\varepsilon_1}}\right)\cap\Omega$ such that
\begin{equation}\label{eqn-7-ce}
  w_{\hat{\lambda}}(\bar{x})=\inf\limits_{\left(\Sigma_{\hat{\lambda}}\setminus\overline{\Sigma_{\lambda_{0}-\varepsilon_1}}\right)\cap\Omega}w_{\hat{\lambda}}
  =\inf\limits_{\Sigma_{\hat{\lambda}}\cap\Omega}w_{\hat{\lambda}}=\inf\limits_{\Sigma_{\hat{\lambda}}}w_{\hat{\lambda}}<0.
\end{equation}
Then, similar to \eqref{contradict-ce} and \eqref{eqn-2ce}, by the assumption \eqref{assump-f-ce} in Theorem \ref{Mono-eg} and \eqref{fractional-MA}, we have
\begin{equation}\begin{split}\label{eqn-8-ce}
& \quad \mathbf{F_{s}}w_{\hat{\lambda}}(\bar{x})\leq\mathbf{F_{s}}u_{\hat{\lambda}}(\bar{x})-\mathbf{F_{s}}u(\bar{x})\\
&=-f\big((\bar{x})^{\hat{\lambda}},u_{\hat{\lambda}}(\bar{x}),(\nabla_{x}u)((\bar{x})^{\hat{\lambda}})\big)
+f\big(\bar{x},u(\bar{x}),(\nabla_{x}u)(\bar{x})\big)\\
&\leq -f\big(\bar{x},u_{\hat{\lambda}}(\bar{x}),(\nabla_{x}u)(\bar{x})\big)
+f\big(\bar{x},u(\bar{x}),(\nabla_{x}u)(\bar{x})\big)\\
&=:c(\bar{x})w_{\hat{\lambda}}(\bar{x}),
\end{split}\end{equation}
where
\begin{equation}\label{eqn-9-ce}
  c(\bar{x}):=-\frac{f\left(\bar{x},u_{\hat{\lambda}}(\bar{x}),(\nabla_{x}u)(\bar{x})\right)-f\left(\bar{x},u(\bar{x}),(\nabla_{x}u)(\bar{x})\right)}
  {u_{\hat{\lambda}}(\bar{x})-u(\bar{x})}
\end{equation}
is uniformly bounded (independent of $\varepsilon_{2}$ and $\varepsilon_{1}$), since $f(x, u, \mathbf{p})\in\mathcal{F}$ and $u\in L^{\infty}_{loc}(\mathbb{R}^n)$. Note that, by choosing $\varepsilon_{1}$ and $\varepsilon_{2}$ smaller if necessary, $\left(\Sigma_{\hat{\lambda}}\setminus\overline{\Sigma_{\lambda_{0}-\varepsilon_{1}}}\right)\cap\Omega$ is a bounded narrow region. From \eqref{eqn-8-ce}, \eqref{eqn-9-ce} and the \emph{Narrow region principle} Theorem \ref{2NRP} (see Remark \ref{rem13} and \ref{rem14}), one can derive that, for $0<\varepsilon_{2}<\varepsilon_{1}$ sufficiently small,
\begin{equation}\label{eqn-10-ce}
  w_{\hat{\lambda}}>0 \qquad \text{in} \,\, \left(\Sigma_{\hat{\lambda}}\setminus\overline{\Sigma_{\lambda_{0}-\varepsilon_{1}}}\right)\cap\Omega,
\end{equation}
which yields a contradiction with \eqref{eqn-6-ce}. As a consequence, we have, for any $\lambda\in[\lambda_{0},\lambda_{0}+\varepsilon_{2}]$, \eqref{eqn-5-ce} holds. Furthermore, it follows from the strong maximum principle Lemma \ref{SMP-anti} that
\begin{equation}\label{221-5-ce}
  w_{\lambda}(x)>0, \qquad \forall \,\, x\in\left(\Sigma_{\lambda}\setminus\overline{\Sigma_{\lambda_{0}-\varepsilon_{1}}}\right)\cap\Omega,
\end{equation}
and hence, for any $\lambda\in[\lambda_{0},\lambda_{0}+\varepsilon_{2}]$,
\begin{equation}\label{221-6-ce}
  w_{\lambda}(x)>0, \qquad \forall \,\, x\in\Sigma_{\lambda}\cap\Omega.
\end{equation}
This contradicts the definition \eqref{M-7} of $\lambda_{0}$. Thus, we must have $\lambda_0=+\infty$. This completes the proof of Theorem \ref{Mono-eg}.
\end{proof}	

\begin{rem}\label{rem16}
Typical forms of $f(x,u,\nabla u)$ which satisfies all the assumptions in Theorem \ref{Mono-eg} include: $f(x,u,\nabla u)=u^{p}(1+|\nabla u|^{2})^{\frac{\sigma}{2}}$ with $p\geq1$ and $\sigma\leq0$, $f(x,u,\nabla u)=e^{\kappa u}(1+|\nabla u|^{2})^{\frac{\sigma}{2}}$ with $\kappa\in\mathbb{R}$ and $\sigma\leq0$, and $f(x,u,\nabla u)=K(x)(1+|\nabla u|^{2})^{\frac{\sigma}{2}}$ with $\sigma\in\mathbb{R}$ and $K(x)$ nondecreasing w.r.t. $x_{n}$.
\end{rem}

\begin{rem}\label{rem9}
Theorem \ref{Mono-eg} is counterpart for the monotonicity results in Theorem 1.3 in Dipierro, Soave and Valdinoci \cite{DSV} for $(-\Delta)^{s}$, Theorem 2.24 in Dai, Qin and Wu \cite{DQW} for $(-\Delta+m^{2})^{s}$, Theorem 1.3 in Berestycki, Caffarelli and Nirenberg \cite{BCN2} and Proposition II.1 in Esteban and Lions \cite{EL} for $-\Delta$.
\end{rem}

\subsection{Schr\"odinger equations in $\mathbb{R}^{n}$}
Consider the \emph{static Schr\"odinger equations} involving the uniformly elliptic nonlocal Bellman operator:
\begin{equation}\label{sta_poly}
-\mathbf{F_{s}}u(x)+u(x)=u^p(x), \qquad \forall \,\,x\in\mathbb{R}^{n}.
\end{equation}

We will prove the following symmetry and monotonicity result for nonnegative solution to \eqref{sta_poly} via the method of moving planes for $\mathbf{F_{s}}$.
\begin{thm}\label{T:sub_symm}
	Assume that $u \in\mathcal{L}_{s}(\mathbb{R}^{n})\cap C_{loc}^{1,1}(\mathbb{R}^{n})$ is a nonnegative solution of \eqref{sta_poly} with $1<p<+\infty$. If
	\begin{equation}\label{sub_con}
	\limsup\limits_{|x|\rightarrow+\infty}u(x)=l<\left(\frac{1}{p}\right)^{\frac{1}{p-1}},
	\end{equation}
	then $u$ must be symmetric and monotone decreasing (along the direction $\mathbf{e}$) w.r.t. some hyper-plane $T$ perpendicular to $\mathbf{e}$.
\end{thm}
\begin{proof}
Without loss of generalities, we assume $\mathbf{e}=\mathbf{e}_{1}:=(1,0,\cdots,0)$. In order to apply the method of moving planes, we need some notations. For arbitrary $\lambda\in\mathbb{R}$, let
	$$
	T_{\lambda}:=\left\{x \in \mathbb{R}^{n} | x_{1}=\lambda\right\}
	$$
	be the moving planes,
	\begin{equation}
	\Sigma_{\lambda}:=\left\{x \in \mathbb{R}^{n} | x_{1}<\lambda\right\}
	\end{equation}
	be the region to the left of the plane, and
	$$
	x^{\lambda}:=\left(2 \lambda-x_{1}, x_{2}, \cdots, x_{n}\right)
	$$
	be the reflection of $x$ about the plane $T_{\lambda}$.
	
	Assume that $u$ is a nonnegative solution of the  Schr\"odinger equations \eqref{sta_poly}. To compare the values of $u(x)$ with $u\left(x^{\lambda}\right)$, we define
	$$
	w_{\lambda}(x):=u\left(x^{\lambda}\right)-u(x), \quad \forall \,\, x\in\Sigma_{\lambda}.
	$$
	Then, for any $\lambda\in\mathbb{R}$, at points $x\in\Sigma_{\lambda}$ where $w_{\lambda}(x)<0$, we have
	\begin{equation}\label{sub_diff}
	\mathbf{F_{s}} w_{\lambda}(x)-c(x)w_{\lambda}(x)\leq 0,
	\end{equation}
	where $c(x):=1-pu^{p-1}(x)$. From the assumption \eqref{sub_con}, we infer that, for any $\lambda\in\mathbb{R}$,
	\begin{equation}\label{223-0}
	\liminf\limits_{\substack{x\in\Sigma_{\lambda},\, w_{\lambda}(x)<0 \\ |x|\rightarrow+\infty}}c(x)>0.
	\end{equation}
	
\medskip

	We carry out the moving planes procedure in two steps.
	
\medskip

	\emph{Step 1.} We use Theorem \ref{dati} (\emph{Decay at infinity (II)}) to show that, for sufficiently negative $\lambda$,
	\begin{equation}\label{223-aim}
	w_{\lambda}(x)\geq 0, \quad \forall \,\, x\in\Sigma_{\lambda}.
	\end{equation}
	
	In fact, from assumption \eqref{sub_con}, we know that $u$ is bounded from above and hence $w_{\lambda}$ is bounded from below for any $\lambda\in\mathbb{R}$. Suppose that $\inf\limits_{\Sigma_{\lambda}}w_{\lambda}<0$. By \eqref{sub_diff} and \eqref{223-0}, we can deduce from Theorem \ref{dati} (\emph{Decay at infinity (II)}) that, there exist $R_{0}>0$ large and $0<\gamma_{0}<1$ close to $1$ (independent of $\lambda$) such that, if $\hat{x}\in\Sigma_{\lambda}$ satisfying $w_{\lambda}(\hat{x})\leq\gamma_{0}\inf\limits_{\Sigma_{\lambda}}w_{\lambda}<0$, then $|\hat{x}|\leq R_{0}$. This will lead to a contradiction provided that $\lambda\leq-R_{0}$. Thus we have, for any $\lambda\leq-R_{0}$, $w_{\lambda}\geq0$ in $\Sigma_{\lambda}$.

\medskip
	
\emph{Step 2.} Step 1 provides a starting point, from which we can now move the plane $T_\lambda$ to the right as long as \eqref{223-aim} holds to its limiting position.
	
	To this end, let us define
	\begin{equation}\label{223-5}
	\lambda_{0}:=\sup\left\{\lambda\in\mathbb{R} \mid w_{\mu}\geq 0 \,\, \text{in} \,\, \Sigma_{\mu}, \,\, \forall \,\mu \leq \lambda\right\}.
	\end{equation}
	It follows from Step 1 that $-R_{0}\leq\lambda_{0}<+\infty$. One can easily verify that
	\begin{equation}\label{223-1}
	w_{\lambda_{0}}(x)\geq0, \quad \forall \,\, x\in\Sigma_{\lambda_{0}}.
	\end{equation}
	
	Next, we are to show via contradiction arguments that
	\begin{equation}\label{223-2}
	w_{\lambda_{0}}(x) \equiv 0, \quad \forall \,\, x \in \Sigma_{\lambda_{0}}.
	\end{equation}
	Suppose on the contrary that
	\begin{equation}\label{223-11}
	w_{\lambda_{0}}\geq 0 \,\, \text { but } \,\, w_{\lambda_{0}}\not\equiv 0 \quad \text { in } \,\, \Sigma_{\lambda_{0}},
	\end{equation}
	then we must have
	\begin{equation}\label{223-3}
	w_{\lambda_{0}}(x)>0, \quad \forall \,\, x\in\Sigma_{\lambda_{0}}.
	\end{equation}
	In fact, if \eqref{223-3} is violated, then there exists a point $\hat{x}\in \Sigma_{\lambda_{0}}$ such that
	$$
	w_{\lambda_{0}}(\hat{x})=\min\limits_{\Sigma_{\lambda_{0}}}w_{\lambda_{0}}=0.
	$$
	Then it follows from \eqref{sta_poly} that
	\begin{equation}\label{223-4}
	\mathbf{F_{s}} w_{\lambda_{0}}(\hat{x})\leq 0,
	\end{equation}
	and hence Lemma \ref{SMP-anti} implies that $w_{\lambda_{0}}\equiv0$ in $\Sigma_{\lambda_{0}}$, which contradicts \eqref{223-11}. Thus $w_{\lambda_{0}}(x)>0$ in $\Sigma_{\lambda_{0}}$.
	
	Then we will show that the plane $T_{\lambda}$ can be moved a little bit further from $T_{\lambda_{0}}$ to the right. More precisely, there exists an $\delta>0$, such that for any $\lambda\in\left[\lambda_{0},\lambda_{0}+\delta\right]$, we have
	\begin{equation}\label{223-6}
	w_{\lambda}(x)\geq 0, \quad \forall \,\, x \in \Sigma_{\lambda}.
	\end{equation}
	
	In fact, \eqref{223-6} can be achieved by using the \emph{Narrow region principle} Theorem \ref{2NRP} and the \emph{Decay at infinity (II)} Theorem \ref{dati}. First, since $c(x):=1-pu^{p-1}(x)$ is uniformly bounded, we can choose $\delta_{1}>0$ small enough such that $\left(\Sigma_{\lambda_{0}+\delta_{1}}\setminus\overline{\Sigma_{\lambda_{0}-\delta_{1}}}\right)\cap B_{R_{\ast}}(0)$ is a narrow region, where $R_{\ast}:=R_{0}+|\lambda_{0}|\geq R_{0}$ with $R_{0}$ given by \emph{Decay at infinity (II)} Theorem \ref{dati}. From \eqref{223-3}, we deduce that, there exists a $c_{0}>0$ such that
	\begin{equation}\label{223-7}
	w_{\lambda_{0}}(x)\geq c_{0}>0, \qquad \forall \,\, x\in\overline{\Sigma_{\lambda_{0}-\delta_{1}}\cap B_{R_{\ast}}(0)}.
	\end{equation}
	As a consequence, due to the continuity of $w_{\lambda}$ w.r.t. $\lambda$, there exists a $0<\delta_{2}<\delta_{1}$ sufficiently small such that, for any $\lambda\in[\lambda_{0},\lambda_{0}+\delta_{2}]$,
	\begin{equation}\label{223-8}
	w_{\lambda}(x)>0, \qquad \forall \,\, x\in\overline{\Sigma_{\lambda_{0}-\delta_{1}}\cap B_{R_{\ast}}(0)}.
	\end{equation}
	For any $\lambda\in[\lambda_{0},\lambda_{0}+\delta_{2}]$, if we suppose that $\inf\limits_{\Sigma_{\lambda}}w_{\lambda}(x)<0$, then the \emph{Decay at infinity (II)} Theorem \ref{dati} implies that
	\[w_{\lambda}(x)>\gamma_{0}\inf\limits_{\Sigma_{\lambda}}w_{\lambda}(x), \quad \forall \,\, x\in\Sigma_{\lambda}\setminus\overline{B_{R_{0}}(0)},\]
	and hence the negative minimum $\inf\limits_{\Sigma_{\lambda}}w_{\lambda}(x)$ can be attained in $B_{R_{0}}(0)\cap\Sigma_{\lambda}$. Then, from \eqref{223-8}, we infer that, if $\inf\limits_{\Sigma_{\lambda}}w_{\lambda}(x)<0$, then the negative minimum $\inf\limits_{\Sigma_{\lambda}}w_{\lambda}(x)$ can be attained in the narrow region $\left(\Sigma_{\lambda}\setminus\overline{\Sigma_{\lambda_{0}-\delta_{1}}}\right)\cap B_{R_{\ast}}(0)$. Therefore, from \emph{Narrow region principle} Theorem \ref{2NRP} (see Remark \ref{rem13}), we get, for any $\lambda\in[\lambda_{0},\lambda_{0}+\delta_{2}]$,
	\begin{equation}\label{223-9}
	w_{\lambda}(x)>0, \qquad \forall \,\, x\in\left(\Sigma_{\lambda}\setminus\overline{\Sigma_{\lambda_{0}-\delta_{1}}}\right)\cap B_{R_{\ast}}(0),
	\end{equation}
	and hence
	\begin{equation}\label{223-10}
	w_{\lambda}(x)\geq0, \qquad \forall \,\, x\in\Sigma_{\lambda}.
	\end{equation}
	Thus \eqref{223-6} holds, which contradicts the definition \eqref{223-5} of $\lambda_{0}$. Hence \eqref{223-2} must be valid. This completes the proof of Theorem \ref{T:sub_symm}.
\end{proof}
\begin{rem}\label{rem15}
	If we use \emph{Decay at infinity (I)} Theorem \ref{P:decay} in the proof of Theorem \ref{T:sub_symm}, then we will need the stronger assumption $$\lim\limits_{|x|\rightarrow+\infty}u(x)=l<\left(\frac{1}{p}\right)^{\frac{1}{p-1}}$$
	instead of \eqref{sub_con}. One can observe that, by using \emph{Decay at infinity (II)} Theorem \ref{dati} instead of Theorem \ref{P:decay}, the ``limit" can be weaken into ``superior limit" in assumption \eqref{sub_con}.
\end{rem}

\section{Maximum principles in unbounded domains and applications}

In this section, we will establish various maximum principles for $\mathrm{F}_{s}$ and $\mathbf{F_{s}}$ in unbounded domains. As applications, we will develop the sliding method (on general unbounded domains) for $\mathrm{F}_{s}$ and $\mathbf{F_{s}}$ and apply the moving planes method for $\mathbf{F_{s}}$ to investigate monotonicity, uniqueness and asymptotic property of solutions to various problems involving the uniformly elliptic nonlocal Bellman operators $\mathrm{F}_{s}$ and $\mathbf{F_{s}}$.

\subsection{Maximum principles in unbounded domains}
First, we can prove the following strong maximum principle.
\begin{lem}(Strong maximum principle)\label{SMP-Ubdd}
	Suppose that $u\in\mathcal{L}_{s}(\mathbb{R}^{n})$ and $u\geq0$ in $\mathbb{R}^{n}$. If there exists $x_{0}\in\mathbb{R}^{n}$ such that, $u(x_{0})=0$, $u$ is $C^{1,1}$ near $x_{0}$ and $\mathrm{F}_{s}u(x_{0})\leq0$, then $u=0$ a.e. in $\mathbb{R}^n$.
\end{lem}
\begin{proof}
	Since there exists $x_0\in\mathbb{R}^{n}$ such that $u(x_0)=\min\limits_{x\in\mathbb{R}^n}u(x)=0$, it follows that
	\begin{align*}
	0&\geq \mathrm{F}_{s}u(x_0)\\
	&=\inf P.V. \int_{\mathbb{R}^{n}} \frac{u(y)-u(x_0)}{|A^{-1}(x_0-y)|^{n+2s}}\mathrm{d}y\\
	&=\inf P.V. \int_{\mathbb{R}^{n}} \frac{u(y)}{|A^{-1}(x_0-y)|^{n+2s}}\mathrm{d}y\\
	&\geq 0.
	\end{align*}
	Thus we must have $u=0$ a.e. in $\mathbb{R}^{n}$. This finishes the proof of Lemma \ref{SMP-Ubdd}.
\end{proof}

\begin{thm}(Maximum Principles in unbounded open sets)\label{MP-Ubdd}
	Let $D$ be an open set in $\mathbb{R}^n$, possibly unbounded and disconnected. Assume that $\overline{D}$ is disjoint from an infinite open domain $\Gamma\subset\overline{D}^c$ satisfying
	\begin{eqnarray}\label{domain}
	\frac{|\Gamma\cap (B_{c_1r_x}(x)\backslash B_{r_x}(x))|}{|B_{c_1r_x}(x)\backslash B_{r_x}(x)|}\geq c_{0}>0, \qquad \forall \,\, x\in D
	\end{eqnarray}
	for some constants $c_1>1$, $c_{0}>0$ independent of $x$ and $r_x>0$ possibly depending on $x$.
	
	Suppose that $u\in \mathcal{L}_{s}(\mathbb{R}^{n})\cap C_{\text {loc}}^{1,1}(D)$ is bounded from above, and satisfies
	\begin{equation}\label{MP2-2}
	\begin{cases} \mathrm{F}_{s} u(x)-c(x)u(x)\geq 0 \quad \text{at points} \,\, x\in D \,\, \text{where} \,\,u(x)>0,\\ u(x)\leq 0, \quad x\in \mathbb{R}^{n}\setminus D,
	\end{cases}
	\end{equation}
	where $c(x)$ is nonnegative in the set $\{x\in D \mid u(x)>0\}$. Then $u\leq 0$ in $D$.
	
	Furthermore, assume that
	\begin{equation}\label{MP2-CON}
	\mathrm{F}_{s}{u}(x)\geq0 \quad \text{at points} \,\, x\in D \,\, \text{where} \,\, u(x)=0,
	\end{equation}
	then we have
	\begin{eqnarray}\label{3NRP-4}
	\mbox{ either } \quad u(x)<0 \mbox{ in } D \quad \mbox{ or } \quad u(x)=0 \mbox{ a.e. in } \mathbb{R}^n. \end{eqnarray}
\end{thm}
\begin{proof}
	Suppose on the contrary that there exists one point $x\in D$ such that $u(x)>0$, then we have
	\begin{equation}\label{MP2-3}
	0<M:=\sup\limits_{x\in \mathbb{R}^{n}} u(x) <\infty.
	\end{equation}
	There exists sequences $x^k\in D$ and $0<\beta_k<1$ with $\beta_k\rightarrow 1$ as $k\rightarrow \infty$ such that
	\begin{equation}\label{MP2-4}
	u(x^k)\geq \beta_k M.
	\end{equation}
	Let
	\begin{equation*}
	\psi(x)=\begin{cases}e^{\frac{|x|^{2}}{|x|^2-1}}, \qquad |x|<1\\ 0, \qquad \qquad\, |x|\geq 1.\end{cases}
	\end{equation*}
	It is well known that $\psi\in C_0^\infty(\mathbb{R}^n)$, therefore $|\mathrm{F}_{s}\psi(x)|\leq C_0$ for any $x \in \mathbb{R}^n$. Moreover, $\mathrm{F}_{s}\psi(x)\sim|x|^{-n-2s}$ as $|x|\rightarrow+\infty$.
	
	Define $$ \Psi_k(x):=\psi\left(\frac{x-x^k}{r_{x^k}}\right).$$
	Take $\epsilon_k:=(1-\beta_k)M$. Since $u\leq M$ and $\Psi_k=0$ in $\mathbb{R}^n\setminus B_{r_{x^k}}(x^k)$, we have
	\begin{equation}\label{MP2-5}
	u(x^k)+\epsilon_k\Psi_{k}(x^k)\geq M\geq u(x)+\epsilon_k\Psi_{k}(x),
	\end{equation}
	for any $x\in \mathbb{R}^n\setminus B_{r_{x^k}}(x^k)$.
	Consequently, there exists $\bar{x}^k\in B_{r_{x^k}}(x^k)$ such that
	\begin{equation}\label{MP2-6}
	u(\bar{x}^k)+\epsilon_k\Psi_{k}(\bar{x}^k)= \max\limits_{x\in\mathbb{R}^{n}}[u(x)+\epsilon_k\Psi_{k}(x)]\geq M,
	\end{equation}
	which also implies that
	\begin{equation}\label{MP2-7}
	u(\bar{x}^k)\geq u(x^k)+\epsilon_k\Psi_{k}(x^k)-\epsilon_k\Psi_{k}(\bar{x}^k)\geq u(x^k)\geq \beta_{k}M>0.
	\end{equation}
	
	Therefore, we deduce from \eqref{MP2-6} that
	\begin{align}\label{MP2-8}
	&\quad \mathrm{F}_{s}[u+\epsilon_k\Psi_{k}](\bar{x}^k)\nonumber\\
	&=\inf P.V.\int_{\mathbb{R}^{n}}\frac{u(y)+\epsilon_k\Psi_{k}(y)-u(\bar{x}^k)-\epsilon_k\Psi_{k}(\bar{x}^k)}{|A^{-1}(\bar{x}^k-y)|^{n+2s}} \mathrm{d}y\nonumber\\	
	&=\inf\left[ P.V.\int_{B_{r_{x^k}}(x^k)}\frac{u(y)+\epsilon_k\Psi_{k}(y)-u(\bar{x}^k)-\epsilon_k\Psi_{k}(\bar{x}^k)}{|A^{-1}(\bar{x}^k-y)|^{n+2s}} \mathrm{d}y\right.\nonumber\\
	&\left.\qquad+\int_{\left(B_{r_{x^k}}(x^k)\right)^{c}}\frac{u(y)+\epsilon_k\Psi_{k}(y)-u(\bar{x}^k)-\epsilon_k\Psi_{k}(\bar{x}^k)}{|A^{-1}(\bar{x}^k-y)|^{n+2s}} \mathrm{d}y\right]\nonumber\\
	&\leq \inf\int_{\left(B_{r_{x^k}}(x^k)\right)^{c}}\frac{u(y)+\epsilon_k\Psi_{k}(y)-u(\bar{x}^k)-\epsilon_k\Psi_{k}(\bar{x}^k)}{|A^{-1}(\bar{x}^k-y)|^{n+2s}} \mathrm{d}y\\
	&\leq \inf\int_{\Gamma\cap \left(B_{c_{1}r_{x^k}}(x^k)\setminus B_{r_{x^k}}(x^k)\right)}\frac{-u(\bar{x}^k)-\epsilon_k\Psi_{k}(\bar{x}^k)}{|A^{-1}(\bar{x}^k-y)|^{n+2s}} \mathrm{d}y\nonumber\\
	&\leq -C\left(u(\bar{x}^k)+\epsilon_k\Psi_{k}(\bar{x}^k)\right)\int_{\Gamma\cap\left(B_{c_{1}r_{x^k}}(x^k)\setminus B_{r_{x^k}}(x^k)\right)}\frac{1}{|\bar{x}^k-y|^{n+2s}} \mathrm{d}y\nonumber\\
	&\leq -C\left(u(\bar{x}^k)+\epsilon_k\Psi_{k}(\bar{x}^k)\right)\frac{1}{r_{x^k}^{2s}}\nonumber\\
	&\leq -\frac{CM}{r_{x^k}^{2s}}\nonumber.
	\end{align}
	
	Next, we will evaluate the lower bound of $\mathrm{F}_{s}[u+\epsilon_k\Psi_{k}](\bar{x}^k)$.
	
	Indeed, since \eqref{MP2-7} implies $u(\bar{x}^{k})>0$ and hence $\bar{x}^{k}\in D$, we conclude from \eqref{MP2-2} and  $|\mathrm{F}_{s}\Psi_{k}(x)|\leq C_0$ for any $x\in\mathbb{R}^{n}$ that
	\begin{align}\label{MP2-9}
	&\quad \mathrm{F}_{s}[u+\epsilon_k\Psi_{k}](\bar{x}^k)\nonumber\\
	&\geq \mathrm{F}_{s}u(\bar{x}^k)+\mathrm{F}_{s}\epsilon_k\Psi_{k}(\bar{x}^k)\\
	&\geq c(\bar{x}^{k})u(\bar{x}^k)-\frac{C_0\epsilon_{k}}{r_{x^k}^{2s}}\nonumber\\
	&\geq-\frac{C_0\epsilon_{k}}{r_{x^k}^{2s}}.\nonumber
	\end{align}
	
	Combining \eqref{MP2-8} and \eqref{MP2-9}, we derive
	\begin{equation}\label{MP2-11}
	\frac{C_0\epsilon_{k}}{r_{x^k}^{2s}}\geq \frac{CM}{r_{x^k}^{2s}},
	\end{equation}
	which implies $$C_0(1-\beta_k)\geq C.$$ This will lead to a contradiction for $k$ sufficiently large.
	
	Furthermore, if there exists a point $\tilde{x}\in D$ such that $u(\tilde{x})=0$, then it follows immediately from \eqref{MP2-CON} and Lemma \ref{SMP-Ubdd} that $u=0$ a.e. in $\mathbb{R}^n$. Therefore, we have
	$$\mbox{ either } \quad {u}(x)<0 \mbox{ in } D \quad \mbox{ or } \quad u(x)=0 \mbox{ a.e. in } \mathbb{R}^n.$$
	This completes our proof of Theorem \ref{MP-Ubdd}.
\end{proof}
\begin{rem}
	For fractional Laplacians $(-\Delta)^{s}$ ($0<s<1$), Dipierro, Soave and Valdinoci proved in \cite{DSV} Maximum Principles in unbounded open set $D$ by using Silvestre's growth lemma (\cite{S}) under the {\em exterior cone condition} that the complement of $D$ contains an infinite open connected cone $\Sigma$. Subsequently, Chen and Liu \cite{CLiu}, Chen and Wu \cite{CW2} introduced new ideas in the proof and thus significantly weakens the {\em exterior cone condition} to the following condition:
	\begin{eqnarray}\label{domain2}
	\liminf\limits_{k\rightarrow \infty}\frac{|D^c\cap (B_{2^{k+1}r}(q)\backslash B_{2^kr}(q))|}{|B_{2^{k+1}r}(q)\backslash B_{2^kr}(q)|}=c_{0}>0, \qquad \forall \,\, q\in D
	\end{eqnarray}
	for some $c_{0}>0$ and $r>0$. Typical examples of $D$ which satisfy condition (\ref{domain2}) but does not satisfy the {\em exterior cone condition} include: stripes, annulus and Archimedean spiral (refer to \cite{CLiu,CW2} for details).
	Our assumption \eqref{domain} is rather  weaker than \eqref{domain2}.
\end{rem}

From the proof of Theorem \ref{MP-Ubdd}, we can deduce the following narrow region principle in unbounded open sets.
\begin{thm}[Narrow region principle in unbounded open sets]\label{32NRP}
Let $D$ be an open set in $\mathbb{R}^{n}$ (possibly unbounded and disconnected) and $d(D):=\sup\limits_{x\in D}dist(x,D^{c})$ be the width of $D$. Assume that $D$ satisfies \eqref{domain} with $r_{x}=dist(x,D^{c})\leq d(D)$. Suppose that $u\in {{\mathcal{L}}_{s}(\mathbb{R}^{n})}\cap C_{loc}^{1,1}(D)$ is bounded from above and satisfies
	\begin{eqnarray}\label{32NRP-1}\left\{\begin{array}{ll}
	\mathrm{F}_{s}{u}(x)-c(x){u}(x) \geq 0 &\quad \text{at points} \,\, x\in D \,\, \text{where} \,\, u(x)>0, \\
	{u}(x)\leq 0 &\quad \mbox{ in } D^c,
	\end{array} \right. \end{eqnarray}
	where $c(x)$ is uniformly bounded from below (w.r.t. $d(D)$) in $\left\{x\in D \mid u(x)>0\right\}$. If we assume that
	\begin{eqnarray}\label{32NRP-2}
	\inf\limits_{\{x\in D \mid u(x)>0\}}\, c(x)>-\frac{C}{4d(D)^{2s}},
	\end{eqnarray}
	where $C>0$ is the same constant as in \eqref{MP2-8}. Then
	\begin{eqnarray}\label{32NRP-3}
	u(x) \leqslant 0 \quad \mbox{ in } D.
	\end{eqnarray}
	Furthermore, assume that
	\begin{equation}\label{32NRP-7}
	\mathrm{F}_{s}{u}(x)\geq 0 \quad \text{at points} \,\, u\in D \,\, \text{where} \,\, u(x)=0,
	\end{equation}
	then we have
	\begin{eqnarray}\label{32NRP-4}
	\mbox{ either } \quad u(x) < 0 \mbox{ in } D \quad \mbox{ or } \quad u(x)=0 \mbox{ a.e. in } \mathbb{R}^n. \end{eqnarray}.
\end{thm}
\begin{proof}
	Indeed, we infer from \eqref{MP2-8} and \eqref{MP2-9} that
	\begin{equation}\label{MP2-12}
	c(\bar{x}^{k})u(\bar{x}^k)-\frac{C_0\epsilon_{k}}{r_{x^k}^{2s}}\leq -\frac{C M}{r_{x^k}^{2s}}.
	\end{equation}
	For $k$ sufficiently large such that $\beta_k\geq \max\left\{1-\frac{C}{2C_0}, \frac{1}{2}\right\}$, we derive from \eqref{MP2-7} and \eqref{MP2-12} that
	\begin{equation}\label{MP2-13}
	c(\bar{x}^{k})\leq -\frac{C M}{2r_{x^k}^{2s}u(\bar{x}^k)}\leq -\frac{C }{4d(\Omega)^{2s}},
	\end{equation}
	which contradicts \eqref{32NRP-2}.
	
	Furthermore, if there exists a point $\tilde{x}\in D$ such that $u(\tilde{x})=0$, then it follows immediately from \eqref{MP2-CON} and Lemma \ref{SMP-Ubdd} that $u=0$ a.e. in $\mathbb{R}^n$. Therefore, we have
	$$\mbox{ either } \quad {u}(x)<0 \mbox{ in } D \quad \mbox{ or } \quad u(x)=0 \mbox{ a.e. in } \mathbb{R}^n.$$
	This completes our proof of Theorem \ref{32NRP}.
\end{proof}

\begin{rem}\label{rem3}
In Theorem \ref{32NRP}, if the positive maximum of $u$ is attained in $D$, then we only need to assume $\mathrm{F}_{s}{u}(x)-c(x){u}(x) \geq 0$ at points $x\in D$ where $u(x)=\sup\limits_{D}u(x)>0$, and $\inf\limits_{\{x\in D \mid u(x)>0\}}\, c(x)$ in \eqref{32NRP-2} can be replaced by the infimum of $c(x)$ over the set consisting of positive maximum points of $u$, the same conclusions are still valid. Indeed, if the positive maximum of $u$ is attained at some points $\hat{x}\in D$, then in the proof of Theorem \ref{32NRP}, we may simply replace $\bar{x}^{k}$ by the positive maximum point $\hat{x}$ and take $\epsilon_{k}=0$, and we get
\begin{equation}\label{MP2-12'}
  c(\hat{x})u(\hat{x})\leq -\frac{C M}{r_{\hat{x}}^{2s}}
\end{equation}
instead of \eqref{MP2-12}, where $M:=sup_{D}u(x)>0$. It follows that $c(\hat{x})\leq-\frac{C}{d(\Omega)^{2s}}$, which yields a contradiction immediately.
\end{rem}

\begin{rem}\label{rem2}
	Denote $D^+:=\{x\in D\mid u(x)>0\}$. Theorem \ref{32NRP} implies that, if $c(x)$ is bounded from below, then there exists $r_0>0$ and $0<\beta_0<1$ close to $1$ such that
	\begin{equation*}
	\text{if}\,\,x\in D \,\, \text{satisfying}\,\, u(x)\geq\beta_{0}\sup\limits_{D}u>0, \quad \text{then} \,\, dist(x,\partial D^+)>r_0.
	\end{equation*}
	This indicates that the ``almost" positive maximal points must be away from the boundary of $D^+$.
\end{rem}

The maximum principles in Lemma \ref{SMP-Ubdd}, Theorems \ref{MP-Ubdd} and \ref{32NRP} are also valid for $\mathbf{F}_{s}$.

\subsection{Monotonicity in Epigraph $E$}

Let the epigraph
$$E:=\left\{x=(x',x_n)\in\mathbb{R}^{n} \mid x_n>\varphi(x')\right\},$$
where $\varphi:\,\mathbb{R}^{n-1}\rightarrow\mathbb{R}$ is a continuous function. A typical example of epigraph $E$ is the upper half-space $\mathbb{R}^{n}_{+}$ ($\varphi\equiv0$).

By applying the maximum principles established in subsection 4.1, we can show the following monotonicity result on the epigraph $E$ via sliding method.
\begin{thm}\label{A}
	Let $u \in {\mathcal{L}_{s}}(\mathbb{R}^n)\cap C_{loc}^{1,1}(E)$ be a bounded solution of
	\begin{eqnarray}\label{332-PDE}
	\left\{\begin{array}{ll}
	-\mathrm{F}_{s} u(x)=f(u(x)), &x \in E,\\
	u(x)=0, & x \in \mathbb{R}^n\setminus E,
	\end{array}
	\right.
	\end{eqnarray}
	where $f(\cdot)$ is nonincreasing in the range of $u$.
	Assume that there exists $l>0$ such that
	\begin{equation}\label{332-con}
	u\geq0 \,\,\,\, \text{in} \,\, \left\{x=(x',x_n)\in E\,|\,\varphi(x')<x_{n}<\varphi(x')+l\right\}.
	\end{equation}
	Then, either $u\equiv0$ in $\mathbb{R}^{n}$ and $f(0)=0$, or $u$ is strictly monotone increasing in the $x_n$ direction and hence $u>0$ in $E$.
	
	If, in addition, $F_{s}$ in \eqref{332-PDE} is replaced by $\mathbf{F_{s}}$ and $E$ is contained in a half-space $\Sigma$ such that $\mathbf{e}\perp\partial\Sigma$, the same conclusion can be reached without the assumption \eqref{332-con}. Furthermore, if $E$ itself is exactly a half-space with $\mathbf{e}\perp\partial E$, then
	$$u(x)=u\left(\langle\left(x',x_{n}-\varphi(0')\right),{\bf \nu}\rangle\right),$$
	where ${\bf \nu}=\pm\mathbf{e}$ is the unit inner normal vector to the hyper-plane $\partial E$ and $\langle\cdot,\cdot\rangle$ denotes the inner product in Euclidean space. In particular, if $\mathbf{e}=\mathbf{e_{n}}$ and $E=\mathbb{R}^{n}_{+}$, then $u(x)=u(x_{n})$.
\end{thm}
\begin{proof}
	For any $0<\tau<l$, let
	$$u^\tau(x):=u(x', x_n+\tau)$$
	and
	$$w^\tau(x):=u^\tau(x)-u(x).$$
	
	Since $f(\cdot)$ is nonincreasing, we have
	$$
	\mathrm{F}_{s}w^\tau(x)\leq \mathrm{F}_{s}u^\tau(x)-\mathrm{F}_{s} u(x)=-f(u^\tau(x))+f(u(x))\leq 0
	$$
	at points $x\in E$ where $w^\tau(x)<0$. In addition, for any $0<\tau<l$, we have
	$$w^\tau(x)\geq 0, \quad \forall \,\, x \in \mathbb{R}^n\setminus D.$$
	Thus it follows immediately from Theorem \ref{MP-Ubdd} that, for any $0<\tau<l$,
	$$w^\tau(x)\geq 0, \quad \forall \,\, x\in E.$$
	
	Now, suppose that $u\not\equiv0$ in $E$, then there exists a $\hat{x}\in E$ such that $u(\hat{x})>0$. We are to show that, for any $0<\tau<l$,
	\begin{equation}\label{332-1}
	w^\tau(x)>0, \quad \forall \,\, x \in E.
	\end{equation}
	If not, there exists a point $x^{\tau}\in E$ such that
	$$
	w^\tau(x^{\tau})=0=\min\limits_{\mathbb{R}^n}w^\tau(x).
	$$
	Then we have
	$$
	\mathrm{F}_{s}{w}^\tau(x^{\tau})\leq f(u(x^{\tau}))-f(u^\tau(x^{\tau}))=0,
	$$
	it follows immediately from Lemma \ref{SMP-Ubdd} that $w^\tau=0$ a.e. in $\mathbb{R}^{n}$. This contradicts $u(\hat{x})>0$ and $u=0$ in $\mathbb{R}^{n}\setminus E$. Therefore, \eqref{332-1} holds and hence $u$ is strictly monotone increasing in the $x_N$ direction. In particular, $u>0$ in $E$.

\medskip
	
	If, in addition, $F_{s}$ in \eqref{332-PDE} is replaced by $\mathbf{F_{s}}$ and $E$ is contained in a half-space $\Sigma$ with $\mathbf{e}\perp\partial E$, then the above sliding method is also valid for $\mathbf{F_{s}}$. We will prove that
$$u\geq 0\qquad \text{in}\,\,E$$
and hence the assumption \eqref{332-con} is redundant.
	
	Without loss of generalities, we may assume that $\mathbf{e}=\mathbf{e_{n}}=(0,\cdots,0,1)$ and $E\subseteq \mathbb{R}^n_+$, let
	\begin{equation}\label{332-2}
	T_{0}:=\left\{x \in \mathbb{R}^{n} | x_n=0\right\},
	\end{equation}
	\begin{equation}\label{332-3}
	\Sigma_{0}:=\left\{x \in \mathbb{R}^{n} | x_{n}>0\right\}
	\end{equation}
	be the region above the plane $T_0$, and
	$$
	x^{0}:=\left(x_{1}, x_{2}, \ldots, -x_{n}\right)
	$$
	be the reflection of $x$ about the plane $T_{0}$. We denote $u_{0}(x):=u\left(x^{0}\right)$ and $w_{0}(x)=u_{0}(x)-u(x)$. For $x\in \Sigma_{0}$ where $w_{0}(x)>0$, we derive from \eqref{332-PDE} that, $x\in E$ and
	\begin{equation*}
	\mathbf{F_{s}}{w}_0(x)\leq f(u_0(x))-f(u(x))\leq 0 \quad \mbox{ at points } x\in E \mbox{ where } w_0(x)<0.
	\end{equation*}
Hence, we obtain from Theorem \ref{MP_anti_ubdd} that $w_0\leq 0$ in $\Sigma_{0}$, which implies immediately $u\geq 0$ in $E$.
	
Furthermore, suppose $E$ itself is exactly a half-space with $\mathbf{e}\perp\partial E$. Without loss of generalities, we may assume that $\mathbf{e}=\mathbf{e_{n}}$ and $E=\mathbb{R}^n_+$. We will show that $u(x)$ depends on $x_n$ only.
	
	In fact, when $E=\mathbb{R}^n_+$, it can be seen from the above sliding procedure that the methods should still be valid if we replace $u^\tau(x):=u(x+\tau \mathbf{e_{n}})$ by $u(x+\tau\nu)$, where $\nu=(\nu_1,\cdots,\nu_n)$ is an arbitrary vector such that $\langle\nu,\mathbf{e_{n}}\rangle=\nu_{n}>0$. Applying similar sliding methods as above, we can derive that, for arbitrary such vector $\nu$,
	$$u(x+\tau\nu)>u(x) \quad \text{in} \,\, \mathbb{R}^n_+, \quad \forall \,\, \tau>0.$$
	Let $\nu_n\rightarrow0+$, from the continuity of $u$, we deduce that
	$$u(x+\tau\nu)\geq u(x)$$
	for arbitrary vector $\nu$ with $\nu_n=0$. By replacing $\nu$ by $-\nu$, we arrive at
	$$u(x+\tau\nu)=u(x)$$
	for arbitrary vector $\nu$ with $\nu_n=0$, this means that $u(x)$ is independent of $x'$, hence $u(x)=u(x_n)$. This finishes the proof of Theorem \ref{A}.
\end{proof}

\begin{rem}\label{rem0}
A typical form of nonlinearity $f(u)$ satisfying the assumption in Theorem \ref{A} is $f(u)=e^{\kappa u}$ with $\kappa\leq0$.
\end{rem}

\subsection{Asymptotic behavior}

As an application of Theorem \ref{MP-Ubdd} and Lemma \ref{SMP-Ubdd}, we can prove the following Lemma.
\begin{lem}\label{lem-ub}
Let $\Omega\subset\mathbb{R}^{n}$ be a domain, $u\in C_{\text {loc}}^{1,1}(\Omega)\cap\mathcal{L}_{s}(\mathbb{R}^{n})$ be a solution of
\begin{equation}\label{eqn0}
  -\mathrm{F}_{s}u(x)=f(u(x)), \quad \forall \,\, x\in\Omega
\end{equation}
such that $u$ is bounded from above and
$$u(x)=\phi(x)<\mu, \quad \forall \,\, x\in \mathbb{R}^{n}\setminus\Omega.$$
Assume $f$ satisfies the assumption:\\
{\textbf{$(H_1)$}} $f(t)>0$ on $(0,\mu)$, $f(\mu)=0$ and $f(t)\leq 0$ for $t\geq \mu$.
	
Then $u<\mu$ in $\Omega$.
\end{lem}
\begin{proof}
	We first show that $u\leq \mu$ in $\Omega$. To this end, define $w(x)=u(x)-\mu$, then $w$ is bounded from above and $w(x)<0$ in $\mathbb{R}^{n}\setminus\Omega$. From equation \eqref{eqn0}, we infer that, at points $x\in\Omega$ where $u(x)>\mu$,
	\begin{equation}
	-\mathrm{F}_{s} w(x)=f(u(x))\leq 0.
	\end{equation}
	It follows from Theorem \ref{MP-Ubdd} that $w(x)\leq 0$ in $\Omega$. Thus we arrive at $u\leq\mu$ in $\Omega$.
	
	Furthermore, by strong maximum principle Lemma \ref{SMP-Ubdd}, we conclude that $u<\mu$ in $\Omega$. This finishes the proof of Lemma \ref{lem-ub}.
\end{proof}

Now we consider the following equation
\begin{equation}\label{MP-13}
\begin{cases}-\mathrm{F}_{s} u(x)=f(u(x)), \quad u(x)>0, \quad \forall \,\, x\in E,\\ u(x)=\phi(x)\geq0, \quad \forall \,\, x\in\mathbb{R}^{n}\setminus E,\end{cases}
\end{equation}
where $E:=\left\{x=(x',x_n)\in\mathbb{R}^{n} \mid x_n>\varphi(x')\right\}$ is the epigraph and $\varphi:\,\mathbb{R}^{n-1}\rightarrow\mathbb{R}$ is a continuous function.

To derive the asymptotic behavior of solutions to \eqref{MP-13}, we need the following hypotheses on $f$: \\
{\textbf{$(H_2)$}} $f(t)\geq c_{0}t$ on $[0,\delta_0]$ for some small $c_{0}>0$ and $\delta_0>0$. \\
{\textbf{$(H_3)$}} $f(t)$ is nonincreasing on $(\mu-\delta_1,\mu)$ for some small $0<\delta_1<\mu$.

\smallskip

We first prove the following Lemma by sliding method, which indicates that the solution of \eqref{MP-13} is bounded away from zero at points far away from the boundary.
\begin{lem}\label{lem-as}
	Suppose that $u\in C_{\text {loc}}^{1,1}(E)\cap\mathcal{L}_{s}(\mathbb{R}^{n})$ is a solution of \eqref{MP-13} and $f(\cdot)$ satisfies $(H_2)$. Then, there exist $0<\varepsilon_0<\mu$ and $M_0>0$ large such that
	\begin{equation}\label{ub-24}
	u(x)>\varepsilon_0, \qquad \forall \,\, x\in E, \,\,  dist(x, \partial E)>M_0.
	\end{equation}
\end{lem}
\begin{proof}
Let $\lambda_1$ be the first eigenvalue of $-\mathrm{F}_{s}$ in $B_1(0)$. Assume $\psi$ is the corresponding eigenfunction satisfying $\psi(0)=\max\limits_{B_1(0)}\psi=1$ and
	\begin{equation}\label{ub-25}
	\begin{cases} -\mathrm{F}_{s} \psi=\lambda_1\psi, \,\,\,\, \psi>0, \qquad \text{in} \,\, B_1(0),\\ \\
	\psi=0, \qquad \text{in} \,\, \mathbb{R}^{n}\setminus B_1(0).\end{cases}
	\end{equation}
For eigenvalue and eigenfunction to $\mathrm{F}_{s}$ and general nonlinear integro-differential operators, please refer to Biswas \cite{Bi} and the references therein.
	
	By hypothesis $(H_2)$,  for any $0<\varepsilon\leq\delta_0$ and $M_0:=\left(\frac{\lambda_1}{c_0}\right)^{\frac{1}{2s}}$, we have
	\begin{align}\label{ub-26}
	-\mathrm{F}_{s}\left(\varepsilon\psi\left(\frac{x}{M_{0}}\right)\right)&=\frac{\lambda_1}{M_{0}^{2s}}\varepsilon\psi\left(\frac{x}{M_{0}}\right)\nonumber\\
	&=c_0\varepsilon\psi\left(\frac{x}{M_{0}}\right)\\
	&\leq f\left(\varepsilon\psi\left(\frac{x}{M_{0}}\right)\right).\nonumber
	\end{align}
	
For an arbitrarily fixed point $y_0\in E$ with $dist(y_0,\partial E)>M_0$, set
\begin{equation}\label{eq-a1}
  \varepsilon_0:=\min\left\{\delta_{0},\frac{1}{2}\inf\limits_{B_{M_0}(y_0)}u\right\}>0.
\end{equation}
Then, we have
	\begin{equation}\label{ub-27}
	u(x)>\varepsilon_0\psi\left(\frac{x-y_0}{M_{0}}\right) \qquad\text{in}\,\,B_{M_0}(y_0).
	\end{equation}
	For any other $y\in E$ with $dist(y,\partial E)>M_0$, we can link $y_0$ and $y$ by a smooth curve $y(t):[0,1]\rightarrow \{x\in E\mid dist(x,\partial E)>M_0\}$ with $y(0)=y_0$ and $y(1)=y$.
	Denote $$v_t(x):=u(x)-\varepsilon_0\psi\left(\frac{x-y(t)}{M_0}\right).$$
	It follows from \eqref{ub-27} that $v_0(x)>0$ for any $x\in B_{M_0}(y(0))$. We intend to prove, for all $t\in[0,1]$,
	\begin{equation}\label{ub-28}
	v_t(x)>0, \quad \forall\,x\in B_{M_0}(y(t)).
	\end{equation}
	Suppose not, let $0<t_\ast<1$ be the smallest $t$ such that \eqref{ub-28} fails. Then, we must have $v_{t_\ast}\geq 0$ in $B_{M_0}(y(t_\ast))$ and there is some point $x^\ast\in B_{M_0}(y(t_\ast))$ such that $v_{t_\ast}(x^{\ast})=0$. On the one hand, we deduce from \eqref{MP-13} and \eqref{ub-26} that
	\begin{equation*}
	-\mathrm{F}_{s} v_{t_\ast}(x^\ast)\geq f(u(x^\ast))-f\left(\varepsilon_0\psi\left(\frac{x^\ast-y(t_{\ast})}{M_{0}}\right)\right)=0.
	\end{equation*}
	However, on the other hand, direct calculation shows
	\begin{align*}
	-\mathrm{F}_{s} v_{t_\ast}(x^\ast)=\inf \, P.V.\int_{\mathbb{R}^n}\frac{-v_{t_\ast}(y)}{|A^{-1}(x^\ast-y)|^{n+2s}}\mathrm{d}y<0.
	\end{align*}
	This is a contradiction! Hence, we have \eqref{ub-28} holds. In particular, for $t=1$ and $x=y$, \eqref{ub-28} gives
	$$u(y)>\varepsilon_0.$$
	Since $y\in E$ with $dist(y,\partial E)>M_0$ is arbitrary, we concludes the proof of Lemma \ref{lem-as}.
\end{proof}

Now, with the help of Lemmas \ref{lem-ub} and \ref{lem-as}, we can prove the following asymptotic property of solution $u(x)$ to \eqref{MP-13} when $x$ is far away from $\partial E$.

\begin{thm}\label{thm-ub}
Assume $u\in C_{\text {loc}}^{1,1}(E)\cap\mathcal{L}_{s}(\mathbb{R}^{n})$ is a solution of \eqref{MP-13} such that $u$ is bounded from above and $u(x)=\phi(x)<\mu$ for any $x\in\mathbb{R}^{n}\setminus E$. Suppose $f(\cdot)$ is continuous and satisfies assumptions $(H_1)$, $(H_2)$ and $(H_3)$. Then, $u(x)\rightarrow \mu$ in $E$ as $dist(x,\partial E)\rightarrow +\infty$.
\end{thm}
\begin{proof}
From Lemma \ref{lem-ub}, we know that $0<u<\mu$ in $E$. Let $\psi$ denote the eigenfunction associated with the first eigenvalue $\lambda_{1}$ of $-\mathrm{F}_{s}$ in $B_{1}(0)$ as in the proof of Lemma \ref{lem-as}.
	
	By the hypothesis $(H_1)$ and the continuity of $f$, one has, there exists $c_1>0$ small such that
\begin{equation}\label{eq-a0}
  f(t)\geq c_1,\qquad \forall\,\, t\in [\varepsilon_0, \mu-\delta_1].
\end{equation}
For any $y\in E$ with $d_{y}:=dist(y, \partial E)$ large enough such that $d_{y}>2M_0$ and $\left(\frac{2}{d_{y}}\right)^{2s}<\frac{c_1}{\lambda_1\mu}$, by Lemma \ref{lem-as}, we have
$$u(x)>\varepsilon_0 \qquad\text{in} \,\, B_{\frac{d_{y}}{2}}(y).$$
	Set $\psi^y(x):=\psi\left(\frac{2(x-y)}{d_{y}}\right)$, then
	\begin{equation}\label{ub-31}
	-\mathrm{F}_{s}\psi^y(x)\leq \lambda_1\left(\frac{2}{d_{y}}\right)^{2s}, \qquad \forall \,\, x\in B_{\frac{d_{y}}{2}}(y).
	\end{equation}
	
	Since $\max\limits_{x\in B_{\frac{d_{y}}{2}}(y)}\psi^y(x)=\psi^y(y)=1$, we have, for $0<\eta\leq\varepsilon_0$,
$$\eta \psi^y(x)<u(x), \qquad \forall \,\, x\in E.$$
	Let
$$\eta_\ast:=\sup\{\eta\mid \eta\psi^y(x)<u(x) \,\,\,\, \text{in} \,\,E\}$$
be the least $\eta$ such that $\eta\psi^y$ touches $u$ in $E$. Since $u<\mu$, $\varepsilon_0\leq \eta_{\ast}<\mu$ is well-defined.
	
	By the definition of $\eta_\ast$, we have $u(x)\geq\eta_\ast\psi^y(x)$ in $\mathbb{R}^n$ and there exists a point $x_0\in B_{\frac{d_{y}}{2}}(y)$ such that $u(x_0)=\eta_\ast\psi^y(x_0)$. This means $x_0$ is the minimum point of the function $u(x)-\eta_\ast\psi^y(x)$, hence we obtain
$$-\mathrm{F}_{s}(u-\eta_\ast\psi^{y})(x_0)\leq 0,$$
	which combined with \eqref{ub-31}, implies that
	\begin{equation}\label{ub-32}
	-\mathrm{F}_{s}u(x_0)=f(u(x_{0}))\leq \eta_\ast\lambda_1\left(\frac{2}{d_{y}}\right)^{2s}<\eta_\ast\lambda_1\frac{c_1}{\lambda_1\mu}<c_1.
	\end{equation}
At the same time, we can conclude that
	\begin{equation}\label{ub-31-1}
	\varepsilon_0<u(x_0)=\eta_\ast\psi^y(x_0)\leq\eta_\ast\psi^y(y)\leq u(y)<\mu.
	\end{equation}
	Combining \eqref{eq-a0}, \eqref{ub-32} with \eqref{ub-31-1} yields that $\mu-\delta_{1}<u(x_0)\leq u(y)<\mu$. Then, we can deduce from $(H_3)$ and \eqref{ub-32} that
	\begin{equation}\label{ub-33}
	-\mathrm{F}_{s}u(y)=f(u(y))\leq f(u(x_{0}))\leq \eta_\ast\lambda_1\left(\frac{2}{d_{y}}\right)^{2s}.
	\end{equation}
	Therefore, we have $\mu-\delta_{1}<u(y)<\mu$ for any $y\in E$ with $d_{y}:=dist(y,\partial E)$ sufficiently large, and $f(u(y))\rightarrow 0$ as $dist(y, \partial E) \rightarrow+\infty$. Combining this with the hypothesis $(H_3)$ implies immediately that
	\begin{equation*}
	\lim\limits_{dist(x, \partial E) \rightarrow+\infty}u(x)=\mu,
	\end{equation*}
	which concludes the proof of Theorem \ref{thm-ub}.
\end{proof}

\begin{rem}\label{rem17}
If $\mu=1$, then the De Giorgi type nonlinearity $f(u)=u-u^{3}$ satisfies all the assumptions $(H_1)$, $(H_2)$ and $(H_3)$ in Theorem \ref{thm-ub} and Lemmas \ref{lem-ub} and \ref{lem-as}.
\end{rem}

Since the maximum principles in Lemma \ref{SMP-Ubdd}, Theorems \ref{MP-Ubdd} and \ref{32NRP} are valid for $\mathbf{F}_{s}$, we can apply the sliding method (on general unbounded domains) to $\mathbf{F}_{s}$ and hence the results in Theorem \ref{A}, Lemmas \ref{lem-ub} and \ref{lem-as} and Theorem \ref{thm-ub} are still valid for $\mathbf{F}_{s}$.

\subsection{Monotonicity in $\mathbb{R}^n_+$}
Without loss of generalities, we assume $\mathbf{e}=\mathbf{e_{n}}$ and consider the following Dirichlet problem on half-space $\mathbb{R}^{n}_{+}$:
\begin{equation}\label{ub-34}
\begin{cases}-\mathbf{F_{s}} u(x)=f(u(x)), \quad \forall \,\, x\in \mathbb{R}^n_+,\\ u(x)=0, \quad \forall \,\, x\in \mathbb{R}^{n}\setminus\mathbb{R}^n_{+},\end{cases}
\end{equation}
where $0<s<1$ and $n\geq1$.

By using the maximum principles established in Section 2 and subsection 4.1, we will prove the following monotonicity result for \eqref{ub-34} via the method of moving planes for $\mathbf{F_{s}}$.
\begin{thm}\label{M-half}
	Let $u \in {\mathcal{L}_{s}}(\mathbb{R}^n)\cap C_{loc}^{1,1}(\mathbb{R}^n_+)$ be a nonnegative nontrivial bounded solution of \eqref{ub-34}. Assume that $f(\cdot)$ is Lipschitz in the range of $u$ and satisfies either $f(0)\not=0$ or  $(H_2)$.
	
	Then, $u$ is strictly monotone increasing in the $x_n$ direction and hence $u>0$ in $\mathbb{R}^n_+$.
\end{thm}
\begin{proof}
We prove Theorem \ref{M-half} via the method of moving planes for $\mathbf{F_{s}}$.

For arbitrary $\lambda>0$, let
$$
T_{\lambda}:=\left\{x \in \mathbb{R}^{n} | x_{n}=\lambda\right\}
$$
be the moving planes,
\begin{equation}
\Sigma_{\lambda}:=\left\{x \in \mathbb{R}^{n} | x_{n}<\lambda\right\}
\end{equation}
be the region below the plane,
$$D_\lambda:=\Sigma_{\lambda}\cap\mathbb{R}^n_+=\left\{x \in \mathbb{R}^{n} | 0<x_{n}<\lambda\right\},$$
and
$$
x^{\lambda}:=\left(x_{1}, x_{2}, \ldots, 2 \lambda-x_{n}\right)
$$
be the reflection of $x$ about the plane $T_{\lambda}$.

Assume that $u$ is a nonnegative nontrivial bounded solution to problem \eqref{ub-34}. To compare the values of $u(x)$ with $u_{\lambda}(x):=u\left(x^{\lambda}\right)$, we denote
$$
w_{\lambda}(x):=u_{\lambda}(x)-u(x).
$$

We aim at proving that $w_{\lambda}>0$ in $D_\lambda$ for any $\lambda>0$, which gives the desired strict monotonicity.

\medskip

The following Lemma is necessary in our proof of Theorem \ref{M-half}.
\begin{lem}\label{lem1}
	Assume that $u \in {\mathcal{L}_{s}}(\mathbb{R}^n)\cap C_{loc}^{1,1}(\mathbb{R}^n_+)$ satisfies \eqref{ub-34} and $u\geq 0$ in $\mathbb{R}^n_+$. If $w_\lambda\equiv0$ in $D_\lambda$ for some $\lambda>0$, then $u\equiv0$ and $f(0)=0$.
\end{lem}
\begin{proof}
	If $w_\lambda\equiv0$ in $D_\lambda$ for some $\lambda>0$, by the strong maximum principle Lemma \ref{SMP-anti}, we have $w_\lambda\equiv0$ in $\Sigma_\lambda$, and hence $u(x)=0$ for $x_n\geq 2\lambda$. Suppose that $u\not\equiv 0$. For $t>2\lambda$, by \eqref{ub-34}, we have
	\begin{align*}
	f(0)=-\mathbf{F_{s}}u(te_n)=\inf \int_{0<y_n<2\lambda} \frac{-u(y)}{|A^{-1}(te_n-y)|^{n+2s}} \mathrm{d}y<0.
	\end{align*}
    However, we can deduce from $u\in{\mathcal{L}_{s}}(\mathbb{R}^n)$ that $$\inf \int_{0<y_n<2\lambda} \frac{-u(y)}{|A^{-1}(te_n-y)|^{n+2s}} \mathrm{d}y\rightarrow 0, \quad\text{as} \,\,t\rightarrow +\infty.$$
    This is a contraction!
\end{proof}

Let $D_\lambda^-:=\{x\in D_\lambda\mid w(x)<0\}$. Then, for any $x\in D_\lambda^-$, we have
\begin{equation}\label{ub-35}
-\mathbf{F_{s}} w_\lambda(x)\geq f(u_\lambda(x))-f(u(x))
=c_\lambda(x) w_\lambda(x),
\end{equation}
where \begin{equation*}
c_{\lambda}(x):=\begin{cases} \frac{f\left(u_{\lambda}(x)\right)-f(u(x))}{u_{\lambda}(x)-u(x)},\quad \text{if}\,\,u_{\lambda}(x)\not=u(x),\\ \\ 0,\quad \text{if}\,\,u_{\lambda}(x)=u(x),
\end{cases}
\end{equation*}
is bounded by the Lipschitz constant of $f$.

\medskip

Now, we continue our proof of Theorem \ref{M-half}. The proof can be divided into two steps.

\medskip

\emph{Step 1.} We will first show that $w_\lambda>0$ in $D_\lambda$ for $\lambda>0$ small.

For $\lambda>0$ small, $D_\lambda$ is an unbounded narrow region, it follows immediately from the \emph{Narrow region principle in unbounded open sets} Theorem \ref{NRP-Anti-ubdd} that $$w_\lambda\geq 0 \quad \text{in}\,\, D_\lambda,$$ and if $w_\lambda=0$ at some point in $D_\lambda$, we have $w_\lambda\equiv 0 \,\,\text{in}\,\, D_\lambda$. Then, Lemma \ref{lem1} implies $u\equiv0$ in $\mathbb{R}^n$. This contradicts with the assumption that $u$ is nontrivial. Therefore, we have, for $\lambda>0$ small,
\begin{equation}\label{ub-36}
w_\lambda> 0 \quad \text{in}\,\, D_\lambda.
\end{equation}

\emph{Step 2.} Step 1 provides a starting point for us to carry out the moving planes procedure. Now we increase $\lambda$ from close to $0$ to $+\infty$ as long as inequality \eqref{ub-36} holds until its limiting position. Define
\begin{equation}\label{ub-37}
\lambda_{0}:=\sup \left\{\lambda>0 \mid w_{\mu}> 0 \,\, \text{in} \,\, D_{\mu}, \,\, \forall \, 0<\mu<\lambda\right\}.
\end{equation}
We aim to prove that
$$
\lambda_{0}=+\infty.
$$

Otherwise, suppose on the contrary that $0<\lambda_{0}<+\infty$, by the definition of $\lambda_{0}$ and Lemma \ref{lem1}, we have
$$w_{\lambda_0}> 0 \quad \text{in}\,\, D_{\lambda_0}.$$
Thus there exists a sequence $\{\lambda_k\}$ such that $\lambda_k>\lambda_0$, $\lambda_k\rightarrow\lambda_0$ as $k\rightarrow +\infty$ and $D_{\lambda_k}^-\not=\emptyset$.
Setting $m_k:=\inf w_{\lambda_k}<0$, then we have $m_k\rightarrow 0$ as $k\rightarrow +\infty$. Let $v_k:=w_{\lambda_k}\chi_{D_{\lambda_k}^-}$, we have, for any $x\in D_{\lambda_k}^-$,
\begin{align}\label{ub-tpde}
-\mathbf{F_{s}} v_k(x)&=\inf \int_{\mathbb{R}^n} \frac{v_k(x)-v_k(y)}{|A^{-1}(x-y)|^{n+2s}} \mathrm{d}y \nonumber\\
&=\inf \left[\int_{D_{\lambda_k}^-} \frac{w_{\lambda_k}(x)-w_{\lambda_k}(y)}{|A^{-1}(x-y)|^{n+2s}} \mathrm{d}y+\int_{\mathbb{R}^n\setminus D_{\lambda_k}^-} \frac{w_{\lambda_k}(x)}{|A^{-1}(x-y)|^{n+2s}} \mathrm{d}y\right]\nonumber\\
&\geq \inf \int_{\mathbb{R}^n} \frac{w_{\lambda_k}(x)-w_{\lambda_k}(y)}{|A^{-1}(x-y)|^{n+2s}} \mathrm{d}y\\
&=-\mathbf{F_s}w_{\lambda_k}(x)\geq f(u_{\lambda_k}(x))-f(u(x))\nonumber\\
&=c_{\lambda_k}(x)v_k(x).\nonumber
\end{align}
Then, applying Remark \ref{rem2} to $v_k$, we deduce that there exist a sequence of points $\{x^k\}$, constants $0<r_0<\frac{\lambda_{0}}{4}$ small and $0<\beta_0<1$ close to $1$ satisfy
\begin{equation}\label{ub-38}
w_{\lambda_k}(x^k)\leq \beta_0 m_k \quad \text{and} \quad 2r_0<(x^k)_n<\lambda_0-2r_0,
\end{equation}
where $(x^k)_n$ denotes the $n$-th component of $x^k$.

Let
\begin{equation*}
\gamma(x)=\begin{cases}e^{\frac{|x|^{2}}{|x|^2-1}}, \quad\,|x|<1\\ 0, \qquad\,\, \quad|x|\geq 1.\end{cases}
\end{equation*}
It is well known that $\gamma\in C_0^\infty(\mathbb{R}^N)$, thus $|\mathbf{F_{s}}\gamma(x)|\leq C$ for all $x \in \mathbb{R}^n$. Moreover, $\mathbf{F_{s}}\gamma(x)\sim|x|^{-n-2s}$ as $|x|\rightarrow+\infty$.

Set
\begin{equation}\label{ub-39}
w_k(x):=w_{\lambda_{k}}(x)-\varepsilon_k\left[\gamma\left(\frac{x-x^k}{r_0}\right)-\gamma\left(\frac{x-(x^k)^{\lambda_k}}{r_0}\right)\right],
\end{equation}
where $\varepsilon_k:=-(1-\beta_{0})m_k$, then we have
$$w_k(x^k)\leq m_k.$$
Note that $w_k$ is also anti-symmetric with respect to $T_{\lambda_k}$.

Since $w_k\geq m_k$ in $\Sigma_{\lambda_k}\setminus B_{r_0}(x^k)$, there exists $\bar{x}^k\in B_{r_0}(x^k)$ such that $$w_k(\bar{x}^k)=\min\limits_{x\in\Sigma_{\lambda_k}} w_k(x).$$

On the one hand,
\begin{equation}\label{ub-40}
-\mathbf{F_s}w_k(\bar{x}^k)\geq f(u_{\lambda_k}(\bar{x}^k))-f(u(\bar{x}^k))-\frac{C\varepsilon_k}{r_0^{2s}}\geq Lw_k(\bar{x}^k)-\frac{C\varepsilon_k}{r_0^{2s}},
\end{equation}
where $L$ denotes the Lipschitz constant of $f$. It is easy to see that
\begin{equation}\label{ub-40'}
Lw_k(\bar{x}^k)-\frac{C\varepsilon_k}{r_0^{2s}}\rightarrow 0,\quad \text{as}\,\,k\rightarrow +\infty.
\end{equation}

On the other hand,
\begin{align}\label{ub-41}
&\quad -\mathbf{F_{s}}w_k(\bar{x}^k)\nonumber\\
&=\inf P.V.\int_{\mathbb{R}^{n}}\frac{w_k(\bar{x}^k)-w_k(y)}{|A^{-1}(\bar{x}^k-y)|^{n+2s}} \mathrm{d}y\nonumber\\
&\leq C\inf P.V.\int_{\mathbb{R}^{n}}\frac{w_k(\bar{x}^k)-w_k(y)}{|\bar{x}^k-y|^{n+2s}} \mathrm{d}y\nonumber\\
&=C\left[ P.V.\int_{\Sigma_{\lambda_k}}\frac{w_k(\bar{x}^k)-w_k(y)}{|\bar{x}^k-y|^{n+2s}} \mathrm{d}y+\int_{\Sigma_{\lambda_k}}\frac{w_k(\bar{x}^k)+w_k(y)}{|\bar{x}^k-y^{\lambda_{k}}|^{n+2s}} \mathrm{d}y\right]\nonumber\\
&=C \left[ P.V.\int_{\Sigma_{\lambda_k}}\left(\frac{1}{|\bar{x}^k-y|^{n+2s}}-\frac{1}{|\bar{x}^k-y^{\lambda_k}|^{n+2s}}\right)\left(w_k(\bar{x}^k)-w_k(y)\right) \mathrm{d}y\right.\\
&\qquad\left.+2w_k(\bar{x}^k)\int_{\Sigma_{\lambda_k}}\frac{1}{|\bar{x}^k-y^{\lambda_k}|^{n+2s}} \mathrm{d}y\right]\nonumber\\
&\leq C\int_{\Sigma_{\lambda_k}\setminus B_{r_0}(\bar{x}^k)}\left(\frac{1}{|\bar{x}^k-y|^{n+2s}}-\frac{1}{|\bar{x}^k-y^{\lambda_k}|^{n+2s}}\right)\left(w_k(\bar{x}^k)-w_k(y)\right) \mathrm{d}y\nonumber\\
&=  C\int_{\Sigma_{\lambda_k-(\bar{x}^k)_n}\setminus B_{r_0}(0)}\left(\frac{1}{|y|^{n+2s}}-\frac{1}{|y^{\lambda_k-(\bar{x}^k)_n}|^{n+2s}}\right)\left(w_k(\bar{x}^k)-w_k(y+\bar{x}^k)\right) \mathrm{d}y\leq 0.\nonumber
\end{align}
Up to an subsequence, we may assume that $(\bar{x}^k)_n\rightarrow r_1\in[r_0, \lambda_0-r_0]$ as $k\rightarrow +\infty$.

Since $-\mathbf{F_s}$ is uniformly elliptic and $u$ is bounded, from the interior regularity in \cite{CS3} and the boundary regularity in \cite{RS}, we deduce that $\tilde{w}_k(x):=w_k(x+\bar{x}^k)$ is uniformly H\"older continuous. Therefore, by the Arzel\`{a}-Ascoli Theorem, there exists a function $w^\infty_{\lambda_0-r_1}$ such that
$$\tilde{w}_k\rightarrow w^\infty_{\lambda_0-r_1} \quad \text{uniformly in} \,\,\mathbb{R}^n, \qquad \text{as}\,\,k\rightarrow +\infty.$$
By the Lebesgue's dominated convergence theorem, we have
\begin{align}\label{ub-41'}
&\quad C\int_{\Sigma_{\lambda_k-(\bar{x}^k)_n}\setminus B_{r_0}(0)}\left(\frac{1}{|y|^{n+2s}}-\frac{1}{|y^{\lambda_k-(\bar{x}^k)_n}|^{n+2s}}\right)\left(w_k(\bar{x}^k)-w_k(y+\bar{x}^k)\right) \mathrm{d}y\\
&\rightarrow -C\int_{\Sigma_{\lambda_{0}-r_1}\setminus B_{r_0}(0)}\left(\frac{1}{|y|^{n+2s}}-\frac{1}{|y^{\lambda_{0}-r_1}|^{n+2s}}\right)w^\infty_{\lambda_0-r_1}(y)\mathrm{d}y\leq 0,\nonumber
\end{align}
as $k\rightarrow +\infty$.

Combining \eqref{ub-40}, \eqref{ub-40'}, \eqref{ub-41} and \eqref{ub-41'}, we obtain
$$ C\int_{\Sigma_{\lambda_{0}-r_1}\setminus B_{r_0}(0)}\left(\frac{1}{|y|^{n+2s}}-\frac{1}{|y^{\lambda_{0}-r_1}|^{n+2s}}\right)w^\infty_{\lambda_0-r_1}(y)\mathrm{d}y=0,$$
which implies that
$$u_{\lambda_0-r_1}^{\infty}(x)\equiv u^\infty(x),\quad \forall x\in \Sigma_{\lambda_{0}-r_1}\setminus B_{r_0}(0),$$
where $u_{\lambda_0-r_1}^{\infty}(x)$ and $u^\infty(x)$ are the limits of $u_{\lambda_k-(\bar{x}^k)_n}(x+\bar{x}^k)$ and $u(x+\bar{x}^k)$ respectively.
By regularity theory, $u^\infty(x)$ satisfies
\begin{equation}\label{ub-42}
\begin{cases}-\mathbf{F_s}u^\infty(x)=f(u^\infty(x)), \qquad \forall \,\, x\in \mathbb{R}^n_{+}-r_{1}e_n, \\ \\ u^{\infty}(x)=0, \qquad \forall \,\, x\in \mathbb{R}^n\setminus(\mathbb{R}^n_{+}-r_{1}e_n).\end{cases}
\end{equation}
Then, by the strong maximum principle Lemma \ref{SMP-anti}, we have $w^\infty_{\lambda_0-r_1}\equiv0$. Consequently, we derive from Lemma \ref{lem1} that $u^\infty\equiv0$ and $f(0)=0$. This leads to a contradiction if $f(0)\not=0$.
	
If $f(0)=0$ and $f$ satisfies $(H_2)$,  we infer from Lemma \ref{lem-as} that $u_{\lambda_0-r_1}^{\infty}(x)\geq\varepsilon_0$ for $x$ with $x_n$ sufficiently negative, while $u^\infty(x)=0$. This is a contradiction!

Therefore, we must have $\lambda_0=+\infty$ and hence $w_{\lambda}>0$ in $D_\lambda$ for any $\lambda>0$. This finishes our proof of Theorem \ref{M-half}.
\end{proof}

\begin{rem}\label{rem19}
Typical kinds of nonlinearities $f(u)$ satisfying all the assumptions in Theorem \ref{M-half} include: De Giorgi type nonlinearity $f(u)=u-u^{3}$ and $f(u)=e^{\kappa u}$ with $\kappa\in\mathbb{R}$.
\end{rem}

\begin{rem}\label{rem18}
For monotonicity of solutions to PDEs involving fractional Laplacians $(-\Delta)^{s}$ on half-space $\mathbb{R}^{n}_{+}$, please refer to Barrios, Del Pezzo, Garc\'{\i}a-Meli\'{a}n and Quaas \cite{BDGQ} and Barrios, Garc\'{\i}a-Meli\'{a}n and Quaas \cite{BGQ}.
\end{rem}

\section*{Acknowledgements}
The authors are grateful to the referees for their careful reading and valuable comments and suggestions that improved the presentation of the paper.


\begin{thebibliography}{99}
\bibitem{AK} H. Abels and M. Kassmann, {\it An analytic approach to purely nonlocal Bellman equations arising in models of stochastic control}, J. Differential Equations, \textbf{236} (2007), no. 1, 29-56.

\bibitem{BCN1} H. Berestycki, L. A. Caffarelli and L. Nirenberg, {\it Inequalitites for second-order elliptic equations with applications to unbounded domains. I}, Duke Math. J., \textbf{81} (1996), 467-494.

\bibitem{BCN2} H. Berestycki, L. A. Caffarelli and L. Nirenberg, {\it Monotonicity for elliptic equations in unbounded Lipschitz domains}, Comm. Pure Appl.  Math., \textbf{50} (1997), 1089-1111.

\bibitem{BCPS} C. Brandle, E. Colorado, A. de Pablo and U. Sanchez, {\it A concave-convex elliptic problem involving the fractional Laplacian}, Proc. Royal Soc. Edinburgh-A: Math., \textbf{143} (2013), 39-71.

\bibitem{BDGQ} B. Barrios, L. Del Pezzo, J. Garc\'{\i}a-Meli\'{a}n and A. Quaas, {\it Monotonicity of solutions for some nonlocal elliptic problems in half-spaces}, Calc. Var. Partial Differential Equations, \textbf{56} (2017),  no. 2, Art. 39, 16 pp.

\bibitem{Be} J. Bertoin, {\it L\'{e}vy Processes}, Cambridge Tracts in Mathematics, \textbf{121}, Cambridge University Press, Cambridge, 1996.

\bibitem{BGQ} B. Barrios, J. Garc\'{\i}a-Meli\'{a}n and A. Quaas, {\it A note on the monotonicity of solutions for fractional equations in half-spaces}, Proc. Amer. Math. Soc., \textbf{147} (2019), no. 7, 3011-3019.

\bibitem{BHM} H. Berestycki, F. Hamel and R. Monneau, {\it One-dimensional symmetry of bounded entire solutions of some elliptic equations}, Duke Math. J., \textbf{103} (2000), 375-396.

\bibitem{Bi} A. Biswas, {\it Principal eigenvalues of a class of nonlinear integro-differential operators}, J. Differential Equations, \textbf{268} (2020), no. 9, 5257-5282.

\bibitem{BN1} H. Berestycki and L. Nirenberg, {\it Monotonicity, symmetry and antisymmetry of solutions of semilinear elliptic equations}, J. Geom. Phys., \textbf{5} (1988), 237-275.

\bibitem{BN2} H. Berestycki and L. Nirenberg, {\it Some qualitative properties of solutions of semilinear elliptic equations in cylindrical domains}, Analysis, et Cetera, 115-164, Academic Press, Boston, MA, 1990.

\bibitem{BN3} H. Berestycki and L. Nirenberg, {\it On the method of moving planes and the sliding method}, Bol. Soc. Brasil. Mat. (N.S.), \textbf{22} (1991), 1-37.

\bibitem{Ca} L. A. Caffarelli, {\it Interior $W^{2,p}$ estimates for solutions of the Monge-Amp\`{e}re equation}, Ann. of Math., \textbf{131} (1990), no. 1, 135-150.

\bibitem{CBL} X. Chen, G. Bao and G. Li, {\it The sliding method for the nonlocal Monge-Amp\`{e}re operator}, Nonlinear Analysis, \textbf{196} (2020), 111786.

\bibitem{CC} L. A. Caffarelli and F. Charro, {\it On a fractional Monge-Amp\`{e}re operator}, Ann. PDE., \textbf{1} (2015), 4.

\bibitem{CD} D. Cao and W. Dai, {\it Classification of nonnegative solutions to a bi-harmonic equation with Hartree type nonlinearity}, Proc. Royal Soc. Edinburgh-A: Math., \textbf{149} (2019), 979-994.

\bibitem{CDQ} W. Chen, W. Dai and G. Qin, {\it Liouville type theorems, a priori estimates and existence of solutions for critical order Hardy-H\'{e}non equations in $\mathbb{R}^N$}, preprint, submitted for publication, arXiv: 1808.06609.

\bibitem{CG} S.-Y. A. Chang and M. d. M. Gonz\`{a}lez, {\it Fractional Laplacian in conformal geometry}, Adv. Math., \textbf{226} (2011), 1410-1432.

\bibitem{CGS} L. A. Caffarelli, B. Gidas and J. Spruck, {\it Asymptotic symmetry and local behavior of semilinear elliptic equations with critical Sobolev growth}, Comm. Pure Appl. Math., \textbf{42} (1989), 271-297.

\bibitem{Cheng1} T. Cheng, {\it Monotonicity and symmetry of solutions of fractional Laplacian equations}, Disc. Cont. Dyn. Syst. - A, \textbf{37} (2017), 3587-3599.
	
\bibitem{CHL} T. Cheng, G. Huang and C. Li, {\it The maximum principles for fractional Laplacian equations and their applications}, Commun. Contemp. Math., \textbf{19} (2017), no. 6, 1750018, 12 pp.

\bibitem{CL} W. Chen and C. Li, {\it Classification of solutions of some nonlinear elliptic equations}, Duke Math. J., \textbf{63} (1991), no. 3, 615-622.

\bibitem{CL2} W. Chen and C. Li, {\it Maximum principles for the fractional $p$-Laplacian and symmetry of solutions}, Adv. Math., \textbf{335} (2018), 735-758.

\bibitem{CLiu} W. Chen and Z. Liu, {\it Maximum principles and monotonicity of solutions for fractional $p$-equations in unbounded domains}, preprint, 2019, arXiv: 1905.06493.

\bibitem{CLL} W. Chen, C. Li and Y. Li, {\it A direct method of moving planes for the fractional Laplacian}, Adv. Math., \textbf{308} (2017), 404-437.

\bibitem{CLO} W. Chen, C. Li and B. Ou, {\it Classification of solutions for an integral equation}, Comm. Pure Appl. Math., \textbf{59} (2006), 330-343.

\bibitem{CLM} W. Chen, Y. Li and P. Ma, {\it The Fractional Laplacian}, World Scientific Publishing Co. Pte. Ltd., 2020, 344 pp, https://doi.org/10.1142/10550.

\bibitem{CNS} L. A. Caffarelli, L. Nirenberg and J. Spruck, {\it The Dirichlet problem for nonlinear second-order elliptic equations. I. Monge-Amp\`{e}re equation}, Commun. Pure Appl. Math., \textbf{37} (1984), no. 3, 369-402.

\bibitem{Co} P. Constantin, {\it Euler equations, Navier-Stokes equations and turbulence, in Mathematical Foundation of Turbulent Viscous Flows}, Vol. 1871 of Lecture Notes in Math., 1-43, Springer, Berlin, 2006.

\bibitem{CQ} W. Chen and S. Qi, {\it Direct methods on fractional equations}, Disc. Cont. Dyn. Syst. - A, \textbf{39} (2019), 1269-1310.

\bibitem{CS} L. A. Caffarelli and L. Silvestre, {\it An extension problem related to the fractional Laplacian}, Comm. PDEs, \textbf{32} (2007), 1245-1260.

\bibitem{CS1} L. A. Caffarelli and L. Silvestre, {\it A nonlocal Monge-Amp\`{e}re equation}, Commun. Anal. Geom., \textbf{24} (2016), no. 2, 307-335.

\bibitem{CS2} L. A. Caffarelli and L. Silvestre, {\it Regularity theory for fully nonlinear integro-differential equations}, Commun. Pure and Appl. Math., \textbf{62} (2009), no. 5, 597-638.

\bibitem{CS3} L. A. Caffarelli and L. Silvestre, {\it The Evans-Krylov theorem for nonlocal fully nonlinear equations}, Ann. of Math., \textbf{174} (2011), no. 2, 1163-1187.

\bibitem{CT} X. Cabr\'{e} and J. Tan, {\it Positive solutions of nonlinear problems involving the square root of the Laplacian}, Adv. Math., \textbf{224} (2010), 2052-2093.

\bibitem{CV} L. A. Caffarelli and L. Vasseur, {\it Drift diffusion equations with fractional diffusion and the quasi-geostrophic equation}, Ann. of Math., \textbf{171} (2010), no. 3, 1903-1930.

\bibitem{CW1} W. Chen and L. Wu, {\it The sliding methods for the fractional $p$-Laplacian}, Adv. Math., \textbf{361} (2020), 106933, 26 pp.

\bibitem{CW2} W. Chen and L. Wu, {\it Monotonicity of solutions for fractional equations with De Giorgi type nonlinearities}, preprint, 2019, arXiv: 1905.09999.

\bibitem{CW3} W. Chen and L. Wu, {\it A maximum principle on unbounded domains and a Liouville theorem for fractional \emph{p}-harmonic functions}, preprint, 2019, arXiv:1905.09986.

\bibitem{CY} S.-Y. A. Chang and P. C. Yang, {\it On uniqueness of solutions of $n$-th order differential equations in conformal geometry}, Math. Res. Lett., \textbf{4} (1997), 91-102.

\bibitem{DF} G. De Philippis and A. Figalli, {\it $W^{2,1}$ regularity for solutions of the Monge-Amp\`{e}re equation}, Invent. Math., \textbf{192} (2013), no. 1, 55-69.

\bibitem{DFQ} W. Dai, Y. Fang and G. Qin, {\it Classification of positive solutions to fractional order Hartree equations via a direct method of moving planes}, J. Differential Equations, \textbf{265} (2018), 2044-2063.

\bibitem{DLW} W. Dai, Z. Liu and P. Wang, {\it Monotonicity and symmetry of positive solutions to fractional $p$-Laplacian equation}, Commun. Contemp. Math., \textbf{24} (2022), Paper No. 2150005, 17 pp, DOI: 10.1142/S021919972150005X.

\bibitem{DQ} W. Dai and G. Qin, {\it Classification of nonnegative classical solutions to third-order equations}, Adv. Math., \textbf{328} (2018), 822-857.

\bibitem{DQ2} W. Dai and G. Qin, {\it Liouville theorems for poly-hamonic functions on $\mathbb{R}^{n}_{+}$}, Archiv der Mathematik, 2020, 11 pp, https://doi.org/10.1007/s00013-020-01464-1.

\bibitem{DQW} W. Dai, G. Qin and D. Wu, {\it Direct methods for pseudo-relativistic Schr\"{o}dinger operators}, J. Geom. Anal., \textbf{31} (2021), no. 6, 5555-5618.

\bibitem{DSV} S. Dipierro, N. Soave and E. Valdinoci, {\it On fractional elliptic equations in Lipschitz sets and epigraphs: regularity, monotonicity and rigidity results}, Math. Ann., \textbf{369} (2017), 1283-1326.

\bibitem{Evans} L. C. Evans, {\it Classical solutions of fully nonlinear, convex, second-order elliptic equations}, Comm. Pure Appl. Math., \textbf{35} (1982), no. 3, 333-363.

\bibitem{EL} M. J. Esteban and P.-L. Lions, {\it Existence and nonexistence results for semilinear elliptic problems in unbounded domains}, Proc. Royal Soc. Edinburgh-A: Math., \text{93} (1982/1983), no. 1-2, 1-14.

\bibitem{F} M. M. Fall, {\it Entire $s$-harmonic functions are affine}, Proc. Amer. Math. Soc., \textbf{144} (2016), 2587-2592.

\bibitem{FLS} R. L. Frank, E. Lenzmann and L. Silvestre, {\it Uniqueness of radial solutions for the fractional Laplacian}, Comm. Pure Appl. Math., \textbf{69} (2013), no. 9, 1671-1726.

\bibitem{FW} P. Felmer and Y. Wang, {\it Radial symmetry of positive solutions to equations involving the fractional Laplacian}, Commun. Contemp. Math., \textbf{16}(2014), no. 1, 1350023, 24 pp.

\bibitem{G} P. Guan, {\it $C^2$ a priori estimates for degenerate Monge-Amp\`{e}re equations}, Duke Math. J., \textbf{86} (1997), no. 2, 323-346.

\bibitem{GNN1} B. Gidas, W. Ni and L. Nirenberg, {\it Symmetry and related properties via maximum principle}, Comm. Math. Phys., \textbf{68} (1979), 209-243.

\bibitem{GTW} P. Guan, N. Trudinger and X. Wang, {\it On the Dirichlet problem for degenerate Monge-Amp\`{e}re equations}, Acta Math., \textbf{182} (1999), no. 1, 87-104.

\bibitem{JW} H. Jian and X. Wang, {\it Existence of entire solutions to the Monge-Amp\`{e}re equation}, Amer. J. Math., \textbf{136} (2014), no. 4, 1093-1106.

\bibitem{K} N. V. Krylov, {\it Controlled Diffusion Processes}, Applications of Mathematics, vol. 14, Springer, New York, 1980.

\bibitem{Li1} C. Li, {\it Monotonicity and symmetry of solutions of fully nonlinear elliptic equations on bounded domains}, Comm. PDEs, \textbf{16} (1991), 491-526.

\bibitem{Li2} C. Li, {\it Monotonicity and symmetry of solutions of fully nonlinear elliptic equations on unbounded domains}, Comm. PDEs, \textbf{16} (1991), 585-615.

\bibitem{LiC} C. Li, {\it Local asymptotic symmetry of singular solutions to nonlinear elliptic equations}, Invent. Math., \textbf{123} (1996), no. 2, 221-231.

\bibitem{Lin} C. S. Lin, {\it A classification of solutions of a conformally invariant fourth order equation in $\mathbb{R}^{n}$}, Comment. Math. Helv., \textbf{73} (1998), 206-231.

\bibitem{Lions} P.-L. Lions, {\it Optimal control of diffusion processes and Hamilton-Jacobi-Bellman equations. II. Viscosity solutions and uniqueness}, Comm. Partial Differential Equations, \textbf{8} (1983), no. 11, 1229-1276.

\bibitem{LD} Z. Liu and W. Dai, {\it A Liouville type theorem for poly-harmonic system with Dirichlet boundary conditions in a half space}, Advanced Nonlinear Studies, \textbf{15} (2015), 117-134.

\bibitem{RS} X. Ros-Oton and J. Serra, {\it Boundary regularity for fully nonlinear integro-differential equations}, Duke Math. J., \textbf{165} (2016), no. 11, 2079-2154.

\bibitem{Serrin} J. Serrin, {\it A symmetry problem in potential theory}, Arch. Rational Mech. Anal., \textbf{43} (1971), 304-318.

\bibitem{S} L. Silvestre, {\it Regularity of the obstacle problem for a fractional power of the Laplace operator}, Comm. Pure Appl. Math., \textbf{60} (2007), 67-112.

\bibitem{TW} N. Trudinger and X. Wang, {\it Boundary regularity for the Monge-Amp\`{e}re and affine maximal surface equations}, Ann. of Math., \textbf{167} (2008), no. 3, 993-1028.

\bibitem{WX} J. Wei and X. Xu, {\it Classification of solutions of higher order conformally invariant equations}, Math. Ann., \textbf{313} (1999), no. 2, 207-228.

\bibitem{Yau} S.-T. Yau, {\it On the Ricci Curvature of a compact K\"{a}hler manifold and the complex Monge-Amp\`{e}re equation. I}, Comm. Pure Appl. Math., \textbf{31} (1978), 339-411.

\bibitem{ZW} Z. Zhang and K. Wang, {\it Existence and non-existence of solutions for a class of Monge-Amp\`{e}re equations}, J. Differential Equations, \textbf{246} (2009), 2849-2875.

\end{thebibliography}
\end{document}